\documentclass[final,times,11pt]{elsarticle}
\usepackage{amssymb}
 \usepackage{amsthm}

\usepackage{graphicx} % <- Preamble
\usepackage{epstopdf}
\usepackage{multirow}
\usepackage{mathtools}
\usepackage{subfigure}

\usepackage[colorlinks]{hyperref}

\allowdisplaybreaks

\usepackage[algo2e,ruled]{algorithm2e}
\usepackage{algorithmic} % Preamble
\usepackage{subcaption}
\usepackage{adjustbox}

\newtheorem{remark}{Remark}

\newtheorem{theorem}{Theorem}
\newtheorem{lemma}{Lemma}

\journal{}

\begin{document}

\begin{frontmatter}

%% Title, authors and addresses

%% use the tnoteref command within \title for footnotes;
%% use the tnotetext command for theassociated footnote;
%% use the fnref command within \author or \address for footnotes;
%% use the fntext command for theassociated footnote;
%% use the corref command within \author for corresponding author footnotes;
%% use the cortext command for theassociated footnote;
%% use the ead command for the email address,
%% and the form \ead[url] for the home page:
%% \title{Title\tnoteref{label1}}
%% \tnotetext[label1]{}
%% \author{Name\corref{cor1}\fnref{label2}}
%% \ead{email address}
%% \ead[url]{home page}
%% \fntext[label2]{}
%% \cortext[cor1]{}
%% \affiliation{organization={},
%%             addressline={},
%%             city={},
%%             postcode={},
%%             state={},
%%             country={}}
%% \fntext[label3]{}

\title{Vertex-based auxiliary space multigrid method and its application to linear elasticity equations}
%% use optional labels to link authors explicitly to addresses:
%% \author[label1,label2]{}
%% \affiliation[label1]{organization={},
%%             addressline={},
%%             city={},
%%             postcode={},
%%             state={},
%%             country={}}
%%
%% \affiliation[label2]{organization={},
%%             addressline={},
%%             city={},
%%             postcode={},
%%             state={},
%%             country={}}

\author[2,3]{Jiayin Li}
\ead{ljy@pku.edu.cn}

\author[2,3]{Jinbiao Wu}
\ead{}

\author[2]{Wenqian Zhang}
\ead{}

\author[3]{Jiawen Liu}
\ead{}

%\address[1]{School of Mathematics and Statistic, Lanzhou University, Lanzhou 730000, Gansu, China}

\address[2]{School of Mathematical Science, Peking University, Beijing 100871, China.}
\address[3]{Peking University Chongqing Research Institute of Big Data, Chongqing 400000, China.}

%\cortext[cor1]{Corresponding author}
% \affiliation{organization={},%Department and Organization
%             addressline={},
%             city={},
%             postcode={},
%             state={},
%             country={}}

\begin{abstract}
%% Text of abstract
In this paper, a vertex-based auxiliary space multigrid(V-ASMG) method as a preconditioner of the PCG method is proposed for solving the large sparse linear equations derived from the linear elasticity equations. 
The main key of such V-ASMG method lies in an auxiliary region-tree structure based on the geometrically regular subdivision. 
The computational complexity of building such a region-tree is $\mathcal{O}\left(q N\log_2 N\right)$, where $N$ is the number of the given original grid vertices and $q$ is the power of the ratio of the maximum distance $d_{max}$ to minimum distance $d_{min}$ between the given original grid vertices.
The process of constructing the auxiliary region-tree is similar to the method in \cite{gwx2016}, but the selection of the representative points is changed.
To be more specific, instead of choosing the barycenters, the correspondence between each grid layer is constructed based on the position relationship of the grid vertices.
There are two advantages for this approach:
the first is its simplicity, there is no need to deal with hanging points when building the auxiliary region-tree, and it is possible to construct the restriction/prolongation operator directly by using the bilinear interpolation function, and it is easy to be generalized to other problems as well, due to all the information we need is only the grid vertices;
the second is its strong convergence, the corresponding relative residual can quickly converge to the given tolerance(It is taken to be  $10^{-6}$ in this paper), thus obtaining the desired numerical solution.
Two- and three-dimensional numerical experiments are given to verify the strong convergence of the proposed V-ASMG method as a preconditioner of the PCG method.
\end{abstract}

%%Graphical abstract
% \begin{graphicalabstract}
% %\includegraphics{grabs}
% \end{graphicalabstract}

%%Research highlights
% \begin{highlights}
% \item Research highlight 1
% \item Research highlight 2
% \end{highlights}

\begin{keyword}
%% keywords here, in the form: keyword \sep keyword
V-ASMG method \sep Auxiliary region-tree \sep GS method \sep PCG method \sep Large sparse linear equations \sep Linear elasticity equations \sep Finite element method
%% PACS codes here, in the form: \PACS code \sep code

%% MSC codes here, in the form:
\MSC[2020] 35B50 \sep 35K55 \sep 65M12 \sep 65R20
%% or \MSC[2008] code \sep code (2000 is the default)

\end{keyword}

\end{frontmatter}

%% \linenumbers

%% main text
\section{Introduction}\label{sec:intro}
%%%%%%%%%%%%%%%%%%%%%%%%%%%%%%%%%%%%%%%%%%%%%%%%
In this paper, a vertex-based auxiliary space multigrid(V-ASMG) method is proposed for solving the large sparse linear equations of the following form
\begin{align}
AU=F.
\end{align}
These large sparse linear equations are obtained when solving the linear elasticity equations using the finite element method.
In general, the methods for solving large sparse linear equations can be divided into two categories: direct method and iterative method.
Compared with the direct method, the iterative method requires less editing of the original matrix, less memory, and lower algorithm complexity.
However, its numerical performance is largely affected by matrix properties, which may lead to slow-convergence or even non-convergence.
As a result, many preprocessing and acceleration techniques have been developed.
Among them, the multigrid(MG) method is not only an iterative method, but also an acceleration technology.
Its main idea is to promote the uniform attenuation of error components at different wavelengths through the information exchange between coarse and fine grids, so as to achieve the purpose of accelerating convergence.
The MG method can be further roughly divided into two categories: geometric multigrid(GMG) method and algebraic multigrid(AMG) method.
The GMG method depends on geometric grids, which is more suitable for regular grids, but more difficult to implement for complex geometries and unstructured grids.
While the AMG method can automatically generate coarse grids and interpolation operators based on algebraic properties of coefficient matrices for arbitrary geometric networks, but its construction mode is relatively complex.
%The purpose of this paper is to synthesize the advantages of above two main algorithms, and put forward a V-ASMG method, which has both the simplicity and efficiency of the GMG method and the universality to geometric grid of the AMG method.

The GMG method can be traced back to the 60s and 70s in the 20th century. 
In 1964, Fedorenko proposed an iterative method based on two-layer grid in his work \cite{f1964}, which introduced the idea of coarse grid correction for the first time. 
And then, Brandt's work \cite{b1977} in 1977 systematically established the basic framework for GMG method, including the core ideas of multi-level grid structure, relaxation iteration and coarse grid correction, etc.
These seminal papers laid the foundation for the development of MG methods and promoted their widespread application in scientific computing and engineering.
The 1980s was the period of rapid development of GMG method, and the related research work mainly focused on the theory perfection and application expansion.
The work of Hackbusch\cite{h1980, h1985}, Brandt et al.\cite{b1982} further refined the theoretical basis of GMG method.
Brandt\cite{b1984}, McCormick et al.\cite{m1987} further explored the application of GMG method in the fields of fluid dynamics and partial differential equation solving.
The AMG method also originated during this period, and its  original purpose was to overcome the dependence of GMG method on geometric grids.
The conference paper \cite{bmr1983} of Brandt et al. in 1983 is an early exploration of AMG method, proposing the idea of constructing coarse grids based on matrix algebraic properties such as sparsity and strong connectivity.
Although the complete framework of AMG was not fully formed at this time, it laid the foundation for the subsequent work \cite{rs1987} of Ruge and St$\mathrm{\ddot{u}}$ben in 1987.
In \cite{rs1987}, the basic framework of AMG method was presented systematically for the first time, including coarse grid generation, interpolation operator design and algorithm flow, 
which laid the theoretical foundation of AMG method and became the core reference for subsequent research and application.
The 1990s was the theoretical refinement stage of AMG method. 
In 1996, Van$\mathrm{\check{e}}$k et al.\cite{vmb1996} presented an AMG method based on smoothed aggregation technique, which significantly improved the robustness of AMG method in dealing with complex problems.
Then, in 1999, St$\mathrm{\ddot{u}}$ben published a review\cite{s1999}, which systematically summarized the development process, theoretical basis and application fields of AMG method. 
This review is an authoritative review in the field of AMG and provides researchers with a holistic perspective.
In the early 21st century, AMG methods were widely used in scientific computing and engineering.
In their paper \cite{fy2002} published in 2002,  
Falgout and Yang introduced the AMG solver in the high performance computing (HPC) library Hypre, demonstrating the practical implementation of AMG method in HPC and providing an efficient solution to scientific and engineering problems.
From 2010s to now, thanks to the development of HPC, both GMG and AMG methods have made significant advance in parallel computing and graphics processing unit (GPU) acceleration\cite{ghjrw2016, grsww2015, rsc2014, xz2017}.
The implementation and optimization on GPU significantly improve the computational efficiency of MG method, and provide an important reference for HPC.

The above is a brief history of the development of MG. 
In fact, in addition to aforementioned algorithms, there are many variations.
For example, Bramble et al.\cite{bpx1991} proposed a MG algorithm with non-nested spaces or non-inherited quadratic forms, in which the forms on the coarser grids need not be related to that on the finest.
Then, Xu\cite{x1996} put forward an abstract framework of auxiliary space method, which is a non-nested two-level preprocessing technique based on simple smoother and auxiliary space.
With such an auxiliary space, it is much easier to apply the nested MG methods and easier to generalize to general unstructured grids without adding too much programming effort than traditional solving methods.
Therewith, a new parallel auxiliary grid AMG method was developed by Wang et al.\cite{whcx2013} in 2013.
The coarsing process of such method is based on the region quadtree generated from the auxiliary grid, and makes it possible to explicitly control the sparse pattern and operator complexity of the AMG solver.
Grasedyck et al.\cite{gwx2016} presented an idea of constructing the hierarchical auxiliary coarse grid, on which the GMG method can be utilized after the second layer of coarse grid.
Such a construction is realized by a cluster tree, which in turn is used for the definition of the grid hierarchy from coarse to fine.
Inspired by the method in \cite{gwx2016}, we propose a V-ASMG method, its innovation mainly lies in two points. 
The first point is that we bridge the gap between the primitive unstructured fine grid and the structured coarse grids by the medium of an auxiliary grid and decide whether to subdivide the coarse grid twice according to whether the number of fine grid points in the coarse grid exceeds the given threshold.
This makes it possible to apply the GMG method after the second layer and the correspondence between coarse and fine grid points more balanced.
The other point is that we choose the vertices of the grid as representative points. 
The benefits of this approach are three:
Firstly, the construction process of auxiliary coarse grid is simpler, specifically, there is no need to deal with hanging points. 
Secondly, the prolongation/restriction operator can be constructed by using bilinear interpolation directly. 
Thirdly, the magnitude of the stiffness matrix corresponding to each grid layer  decreases more gently, thus better eliminating the error components of different frequencies.

The rest of the paper is organized as follows. 
Sect. 2 mainly introduces some preliminary knowledge, that is, the main idea of finite element method, some classical algorithms involved in this paper, such as Gauss-Seidel(GS) method and preconditioned conjugate gradient(PCG) method. 
In Sect. 3, we describe the proposed V-ASMG method in detail, including its specific implementation steps and some properties.
Sect. 4 displays some theories of the V-ASMG method as a preconditioner, and it is proved that the V-ASMG method can indeed accelerate the iterative convergence speed by reducing the matrix condition number.
Finally, in Sect. 5, some numerical examples are given to prove that our algorithm as a preconditioner can accelerate the convergence process of the iterative method.
Concluding remarks are given in the last section.

%%%%%%%%%%%%%%%%%%%%%%%%%%%%%%%%%%%%%%%%%%%%%%%%
\section{Preliminaries}\label{sec:pre}
%%%%%%%%%%%%%%%%%%%%%%%%%%%%%%%%%%%%%%%%%%%%%%%%
In this paper, a V-ASMG method is proposed to solve the linear elasticity equations of the following form
\begin{align}\label{prob}
-\nabla\cdot\sigma(u)& =f, &x \in \Omega,\\
\sigma(u)&=2\mu\epsilon(u)+\lambda {\rm tr} \left(\epsilon(u)\right)\mathrm{I},&x\in \Omega,\\
\epsilon(u)&=\frac12\left(\nabla u+(\nabla u)^\top\right),&x\in \Omega.
\end{align}
The above three equations are the equilibrium equation(Navier's equation), the constitutive equation(generalized Hooke's law) and the geometric equation(Cauchy strain-displacement equation) respectively, which form the basis for analyzing the force and deformation of the elastic body.
The equilibrium equation describes the equilibrium relationship between the stress field $\sigma$ and the volume force $f=(f_1,f_2,\dots,f_d)$ under the static condition, where  $u=(u_1,u_2,\dots,u_d)$ is the displacement field and $d=2,3$ represents the spatial dimension.
The constitutive equation shows the material constitutive relationship between stress $\sigma$ and strain $\epsilon$ for linearly elastic and isotropic materials, where $\mu=\frac{E}{2(1+\nu)}$ and $\lambda=\frac{E\nu}{(1+\nu)(1-2\nu)}$ are Lam$\acute{\mathrm{e}}$ constants, $E$ and $\nu$ are the Young's modulus and Poisson's ratio, respectively.
The geometric equation exhibits the geometric relationship between strain $\epsilon$ and displacement $u$ under the small deformation hypothesis.
$\Omega\subset \mathbb{R}^d$ is an open, connected and bounded domain with the Lipschitz continuous boundary $\Gamma=\partial\Omega$. 

Imposed by the homogeneous Dilichlet boundary condition, the following weak form for \eqref{prob} can be obtained: find $u\in V$ such that 
\begin{align}\label{weak}
a(u,v)=(f,v), \mbox{ for all } v\in V,
\end{align}
where 
\begin{align}
a(u,v)=\int_\Omega \sigma:\epsilon(v) {\rm d}x, \, (f,v)=\int_\Omega f\cdot v{\rm d}x,
\nonumber
\end{align}
\begin{align}
Q:T=\left(\begin{array}{ll}q_{11}&q_{12}\\q_{21}&q_{22}\\\end{array}\right):\left(\begin{array}{ll}t_{11}&t_{12}\\t_{21}&t_{22}\\\end{array}\right)=q_{11}t_{11}+q_{12}t_{12}+q_{21}t_{21}+q_{22}t_{22}
\nonumber
\end{align}
for the two-dimensional case and $V=H^1_0(\Omega)=\{u\in H^1(\Omega)\,|\, u|_{\Gamma}=0\}$. 

For the simplicity of presentation, we only focus on the two-dimensional problem, i.e. $d=2$, higher dimensional case can be treated in a similar way. 
Given a quasi-uniform triangulation partition $\mathcal{T}_h$ of $\Omega$, $\bar{\Omega}=\cup K$. 
It should be noted that triangulation partition is merely taken as an example.
In fact, it doesn't matter how we split the computational domain, because what we really need is the information about the vertices of the grid.
Let $V_h\subset V$ be a finite element space with respect to the triangulation $\mathcal{T}_h$ as follows
\begin{align*}
V_h&=\{u_h\in H^1_0(\Omega)\,|\, u_h \mbox{ is a piecewisely linear continuous function}\}\\
&=\mbox{Span}(\phi_i, i=1,2,..,N),
\end{align*}
where $N$ is the number of vertices and $\phi_i$ is the nodal basis function of $V_h$, which is $1$ at vertice $p_i$ and $0$ at every other vertices $p_j$ $(j\neq i)$. Then we get the discretization form of \eqref{weak}: find $u_h\in V_h$ such that
\begin{align}\label{discrete}
a(u_h,v_h)=(f,v_h), \,\mbox{ for all }v_h\in V_h.
\end{align}

By setting $f=(f_{1},f_{2})$ and $u_h=(u_{1h},u_{2h})$ with $u_{1h}=\Phi U_1$ and $u_{2h}=\Phi U_2$, where 
\begin{align}
\Phi=(\phi_1,\phi_2,\cdots,\phi_N),
U_{i}=(u_{i1},u_{i2},\cdots,u_{iN})^\top, \, i=1,2,
\nonumber
\end{align}
we can also obtain the matrix form of \eqref{discrete} as follows
\begin{align}
\left(
\begin{array}{ll}
\lambda A_{11}+2\mu A_{11}+\mu A_{22}&\lambda A_{12}+\mu A_{21}\\
\lambda A_{21}+\mu A_{12}&\lambda A_{22}+2\mu A_{22}+\mu A_{11}\\
\end{array}
\right)
\left(
\begin{array}{ll}
U_1\\
U_2\\
\end{array}
\right)=
\left(
\begin{array}{ll}
F_1\\
F_2\\
\end{array}
\right),
\end{align}
where 
\begin{align}
\nonumber
A_{11}=\int_\Omega\Phi_x^\top \Phi_x {\rm d}x{\rm d}y&, A_{12}=\int_\Omega\Phi_x^\top \Phi_y {\rm d}x{\rm d}y,\\ 
\nonumber
A_{21}=\int_\Omega\Phi_y^\top \Phi_x {\rm d}x{\rm d}y&, A_{22}=\int_\Omega\Phi_y^\top \Phi_y {\rm d}x{\rm d}y,\\ 
\nonumber
F_1 = \int_\Omega f_1\Phi^\top  {\rm d}x{\rm d}y&, F_2 = \int_\Omega f_2\Phi^\top {\rm d}x{\rm d}y,
\nonumber
\end{align}
$\Phi_x$ and $\Phi_y$ represent the partial derivatives of $\Phi$ with respect to $x$ and $y$, respectively. At this point, we get the following system of equations to be solved in this paper 
\begin{align}\label{equ}
AU=F.
\end{align}
It is easy to verify that $A$ is a symmetric positive-definite matrix, thus the discrete system \eqref{discrete} admits a unique solution.

There are two kinds of methods for solving the linear equations such as \eqref{equ}: direct method and iterative method.
Compared with the direct method, the iterative method has the advantages of less editing of the original matrix, less memory, lower algorithm complexity and so on, thus the iterative method is chosen in this paper.
The iterative method can be divided into two categories, stationary iterative method\cite{hy1981, v2000, y1971} (such as Jacobi method, GS method, etc.) and non-stationary iterative method\cite{bbcdddeprv1994, hs1952, ss1986} (for example, PCG method, Krylov method and so on). 
The difference between these two types of methods is whether the iterative matrix is constant.
The implementation of stationary iteration method is relatively simple, but for some complex problems its efficiency is not as high as that of non-stationary iteration method.
The main idea of the non-stationary iteration method is to project a large problem onto a lower-dimensional subspace, on which the optimal approximation of \eqref{equ} in a certain sense can be obtained.
In the next two subsections, we will briefly introduce the GS and PCG methods to be used in this paper, in which GS method acts as the smoothing algorithm of V-ASMG method, and PCG method with V-ASMG method as its preconditioner is chosen to be the solution method.

\subsection{GS method}
GS method\cite{gv2013, s2003} is a classical stationary iterative method.
It is an improved version of Jacobi method, which accelerates the convergence rate by utilizing the most recently calculated components.
For the system of linear equations \eqref{equ}, the updated formula of GS method is as follows
\begin{align}
u_{i}^{(k+1)}=\frac{1}{a_{ii}}\left[f_i-\sum_{j<i}a_{ij}u_{j}^{(k+1)}-\sum_{j>i}a_{ij}u_j^{(k)}\right],\ i=1,2,...,n=2N,
\end{align}
where $u_{i}^{(k)}$ is the approximation of the $i$-th component of unknown quantity $U$ in the $k$-th iteration, 
$a_{ij}$ is the element of stiffness matrix $A$ in the $i$-th row and $j$-th column, $f_i$ is the $i$-th component of load vector $F$ and $n$ is the dimension of the matrix $A$. 
Dividing the original stiffness matrix $A$ into the sum of the diagonal matrix $D$, the strictly lower trigonometric matrix $L$, and the strictly superior trigonometric matrix $S$, i.e. 
\begin{align}
\begin{split}
{A}&={D}-{L}-{S}\\
&=\begin{pmatrix}
a_{11}&&&\\
&a_{22}&&\\
&&\ddots&\\
&&&a_{nn}
\end{pmatrix}+\begin{pmatrix}
0& 0 & \cdots & 0\\
-a_{21} &0 & \cdots & 0\\
\vdots&\vdots&\ddots&\vdots\\
-a_{n1} &-a_{n2} & \cdots & 0
\end{pmatrix}+\begin{pmatrix}
0& -a_{12} & \cdots & -a_{1n}\\
0 &0 & \cdots & -a_{2n}\\
\vdots&\vdots&\ddots&\vdots\\
0 &0 & \cdots & 0
\end{pmatrix},
\end{split}
\nonumber
\end{align}
the following matrix form of GS method can be further obtained
\begin{align}
{U}^{(k+1)}= ({D}-{L})^{-1}{S}{U}^{(k)}+({D}-{L})^{-1}{F}.
\end{align}
The convergence of the GS method depends on the properties of the stiffness matrix $A$. 
If the stiffness matrix $A$ satisfies one of the following conditions:

\begin{itemize}
\item[(1)] the stiffness matrix $A$ ia a strictly diagonally dominant matrix;
\item[(2)] the stiffness matrix $A$ is a symmetric positive-definite matrix;
\item[(3)] the spectral radius of the iterative matrix is less than 1, i.e. $\rho\left(({D}-{L})^{-1}{S}\right)<1$;
\end{itemize}
then the GS method converges.
The pseudo-code framework of GS method is as Algorithm \ref{gs}.
\begin{algorithm2e}
 \caption{GS algorithm}
 \label{gs}
 \renewcommand{\algorithmicrequire}{\textbf{Input: }}
 \renewcommand{\algorithmicensure }{\textbf{Output:}}
 \begin{algorithmic}[1]
 \REQUIRE the stiffness matrix $A$, 
          the load vector $F$, 
          the initial solution $U^{(0)}$, 
          maximum number of iteration steps $k_{max}$ and 
          maximum allowable error $\epsilon$;
 \STATE{\bf{Initialization:} $U = U^{(0)}, k = 0;$}
 \WHILE{$k < k_{max}$}
 \STATE{$U_{old} = U$;  $\%$ Save the current solution}
 \FOR{$i$ from $1$ to $n$}
   \STATE{$\sigma = 0$;}
   \FOR{$j$ from $1$ to $i-1$}
     \STATE{$\sigma = \sigma + A[i][j] * U[j]$;}
   \ENDFOR
   \FOR{$j$ from $i+1$ to $n$}
     \STATE{$\sigma = \sigma + A[i][j] * U_{old}[j]$;}
   \ENDFOR
 \ENDFOR
 \STATE{$U[i] = (F[i] - \sigma) / A[i][i]$;}

 \STATE{$\%$ Check the convergence condition}
 \IF{$err=||U - U_{old}|| < \epsilon$}
   \STATE{Break;}
 \ENDIF
 \STATE{$k = k + 1$;}
 \ENDWHILE
 \ENSURE the final solution $U$, 
         the iteration steps $k$ and 
         the error $err$. 
 \end{algorithmic}
\end{algorithm2e}
The GS method is simple to implement and requires less computation, but its convergence is highly dependent on the properties of matrix $A$ and may not converge or converge slowly for non-diagonally dominant or ill-conditioned matrices. 
In this paper, it is selected as the post- and back-smoothing algorithms of the V-ASMG method.

\subsection{PCG method}
CG method\cite{gv2013, hs1952, s2003, tb1997} is an iterative algorithm suitable for solving symmetric positive-definite linear equations such as \eqref{equ}. 
It is a kind of Krylov subspace method, which has fast convergence rate and low computational complexity, especially suitable for solving large-scale sparse linear system. 
Its main idea is to approach the exact solution by constructing a set of conjugate directions.
In theory, the CG method needs at most $n$ iterations to converge to the exact solution.
However, in practical applications, more iterations are usually required due to numerical errors.
The error of the $k$-th iteration of CG method satisfies
\begin{align}
\Vert {U}^{(k)}-{U}^*\Vert_{A}\leq 2\left(\frac{\sqrt{\kappa}-1}{\sqrt{\kappa}+1}\right)^k\Vert {U}^{(0)}-{U}^*\Vert_{A},
\end{align}
where ${U}^*$ denotes the exact solution, $\kappa=\Vert A\Vert\Vert A^{-1}\Vert$ is the condition number for matrix $A$.
In other words, the smaller the condition number of matrix $A$, the faster the convergence of the CG method. 
Therefore, it is common to improve the condition number of the stiffness matrix $A$ by introducing a preconditioned matrix $M$, so as to accelerate the convergence rate of the CG method.
This is the so-called PCG method.
Its pseudo-code framework is shown in Algorithm \ref{pcg}.
\begin{algorithm2e}
 \caption{PCG algorithm}
 \label{pcg}
 \renewcommand{\algorithmicrequire}{\textbf{Input: }}
 \renewcommand{\algorithmicensure }{\textbf{Output:}}
 \begin{algorithmic}[1]
 \REQUIRE the preconditioned matrix $M$,
          the stiffness matrix $A$, 
          the load vector $F$, 
          the initial solution $U^{(0)}$, 
          maximum number of iteration steps $k_{max}$ and 
          maximum allowable residual $\epsilon$;
 \STATE{\bf{Initialization:} $U = U^{(0)}, r = F - AU, res_{0} = \Vert r \Vert, solve Mz = r, p = z, \rho_{old} = r^{\top}z, k = 1;$}
 \WHILE{$k < k_{max}$}
 \STATE{$q = Ap$;}
 \STATE{$\alpha = \rho_{old} / p^{\top}q$;}
 \STATE{$U = U + \alpha p$;}
 \STATE{$r = r - \alpha q, res = \Vert r \Vert$;}
 \STATE{solve $Mz = r$;}
 \STATE{$\rho_{new} = r^{\top}z$;}

 \STATE{$\%$ Check the convergence condition}
 \IF{$rel\_res = res / res_{0} < \epsilon$}
   \STATE{Break;}
 \ENDIF
 \STATE{$\beta = \rho_{new} / \rho_{old}$;}
 \STATE{$p = z + \beta p$;}
 \STATE{$\rho_{old} = \rho_{new}$;  $\%$ Save the current value}
 \STATE{$k = k + 1$;}
 \ENDWHILE
 \ENSURE the final solution $U$, 
         the iteration steps $k$ and 
         the relative residual $rel\_res$. 
 \end{algorithmic}
\end{algorithm2e}
In particular, when the preconditioned matrix $M$ is the identity matrix $I$, the PCG method degenerates to the ordinary CG method.
The choice of the preconditioned matrix $M$ is crucial to the performance of PCG method.
Common pretreatment methods include Jacobi preconditioning, incomplete Cholesky factorization, symmetric successive over-relaxation preconditioning, MG method and so on.
In this paper, the PCG method is selected as the iterative solution method, and the V-ASMG method is chosen to be its preconditioner.

%%%%%%%%%%%%%%%%%%%%%%%%%%%%%%%%%%%%%%%%%%%%%%%%%%%%%%%%%
\section{V-ASMG method}
%%%%%%%%%%%%%%%%%%%%%%%%%%%%%%%%%%%%%%%%%%%%%%%%%%%%%%%%%
As illustrated in the introduction section, the GMG method is highly dependent on a given hierarchy of geometric grids, which is not readily available in most of the time.
The AMG method is proposed as a means to generalize the GMG method for solving the systems that share properties with their homologous discretized partial differential equations(PDEs), which may potentially generate unstructured grids in the discretization process.
Unfortunately, its ability to handle complex problems also increases the difficulty of operation.
In this paper, the relationship between the primitive unstructured fine grid and the structured coarse grids is established through the medium of an auxiliary grid, which makes it possible to apply the GMG method after the second layer.
The main key of such V-ASMG method lies in an auxiliary region-tree structure based on the geometrically regular subdivision. 

The three key points of MG method are: fine grid smoothing, coarse grid correction and prolongation/restriction operator construction. 
Among them, the high frequency error component is eliminated by fine grid smoothing, and the low frequency error component is eliminated by coarse grid correction, and different grid layers are connected by prolongation/restriction operator.
It is such alternating iteration on the coarse and fine grids that promotes the uniform decay of the error components at each frequency.
The cycle modes of different grid layers of MG method are V-cycle, W-cycle, etc., as shown in \figurename \ref{cycle}.
\begin{figure}[!ht]
	\centerline{
\subfigure[V-cycle.] {\includegraphics[width=0.195\textwidth]{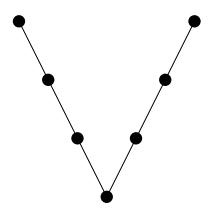}}
\subfigure[W-cycle.] {\includegraphics[width=0.413\textwidth]{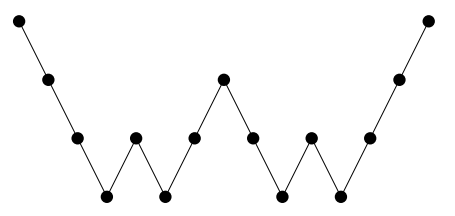}}}
	\caption{The cycle modes of different grid layers of MG method.}
	\label{cycle}
\end{figure}
It should be noted that the multi-layer MG method is a recursive form of the two-layer MG method. 
Therefore, in this paper, we take the two-layer V-cycle MG method as an example to introduce the basic flow for solving the linear equations \eqref{equ}:
\begin{itemize}
\item[Step 1.] Setup step:
Before using the MG method to perform recursive iterative computation, it is necessary to construct the mapping relationship between the fine grid and the coarse grid, namely, the prolongation operator $\mathcal{P}$ and the restriction operator $\mathcal{R}$(The specific construction will be introduced in detail in the Sec. \ref{section:PR}).
And these two operators transpose each other, which is satisfied
\begin{align}\label{R}
\mathcal{R}=\mathcal{P}^{\top}.
\end{align} 
Furthermore, the stiffness matrix of the coarse grid is calculated according to the following equation
\begin{align}\label{Ac}
A^c=\mathcal{R}A^f\mathcal{P}=\mathcal{R}A\mathcal{P}.
\end{align} 

\item[Step 2.] Application step:
\begin{itemize}
\item[Step 2.1.] Pre-smoothing:
Given an initial solution $U^{(0)}$, relax $n_1$ times for \eqref{equ} by adopting a general iteration method(Such as Jacobi method, GS method and so on) to eliminate the high frequency error component, and obtain $U^{(1)}$ and the corresponding residual
\begin{align}\label{rf}
r^f=F-AU^{(1)}.
\end{align}
\item[Step 2.2.] Restriction:
After several iterations of smoothing on the fine grid, the high frequency error component has become very small.
At this time, restrict the residual $r^f$ on the fine grid to the coarse grid according to the mapping relationship \eqref{rc}, so that the low frequency error component can be changed into the relative high frequency error component on the coarse grid.
\begin{align}\label{rc}
r^c=\mathcal{R}r^f.
\end{align}
\item[Step 2.3.] Solving:
Solve the residual equation \eqref{res_equ} on the coarse grid to eliminate the low frequency error component that decays slowly on the fine grid.
Due to the large scale of the coarse grid, the magnitude of the coarse grid stiffness matrix $A^c$ is usually much smaller than that of the original stiffness matrix $A^f=A$, so the residual equation \eqref{res_equ} on the coarse grid can be solved by a direct method(Such as LU decomposition, Cholesky decomposition and so on).
\begin{align}\label{res_equ}
A^ce^c=r^c.
\end{align}
\item[Step 2.4.] Prolongation:
Map the solution $e^c$ of equation \eqref{res_equ} onto the fine grid to obtain the corresponding error $e^f$, and use it to correct the solution $U^{(1)}$ on the fine grid
\begin{align}\label{U2}
U^{(2)}=U^{(1)}+e^f=U^{(1)}+\mathcal{P}e^c.
\end{align}
\item[Step 2.5.] Post-smoothing:
Taken the corrected solution $U^{(2)}$ as the initial solution, relax $n_2$(In general, $n_2=n_1$) times for \eqref{equ} by using the same iterative method as Step 2.1 to obtain the more accurate solution $U^{(3)}$. 
It should be noted that the iterative order of the pre- and post-smoothing is opposite, taking GS method as an example, the pre-smoothing is updated from 1 to $n$ in turn, while the post-smoothing needs to be updated from $n$ to 1 in the fourth step of Algorithm 1.
\end{itemize}
\end{itemize}
The pseudo-code framework of two-layer V-cycle MG method is as Algorithm \ref{mg}.
\begin{algorithm2e}
 \caption{Two-layer V-cycle MG algorithm}
 \label{mg}
 \renewcommand{\algorithmicrequire}{\textbf{Input: }}
 \renewcommand{\algorithmicensure }{\textbf{Output:}}
 \begin{algorithmic}[1]
 \REQUIRE the stiffness matrix $A$, 
          the load vector $F$, 
          the initial solution $U^{(0)}$, 
          maximum number of iteration steps $k_{max}$, 
          maximum allowable error $\epsilon$,
          the umber of pre-smoothing $n_1$ and
          the umber of post-smoothing $n_2$;
 \STATE{{\bf{Setup:}} construct the prolongation operator $\mathcal{P}$ and the restriction operator $\mathcal{R}$, and the coarse grid stiffness matrix $A^c$;}
 \STATE{\bf{Initialization:} $U = U^{(0)}, res_{0} = \Vert F-AU^{(0)} \Vert, k = 0;$}
 \WHILE{$k < k_{max}$}
 \STATE{{\bf{Pre-smoothing:}} $U^{(1)}=GS(A,F,U,n_1,\epsilon);$}
 \STATE{{\bf{Restriction:}} $r^f=F-AU^{(1)}$, $r^c=\mathcal{R}r^f$;}
 \STATE{{\bf{Solving:}} solve $A^ce^c=r^c$;}
 \STATE{{\bf{Prolongation:}} $e^f=\mathcal{P}e^c$, $U^{(2)}=U^{(1)}+e^f$;}
 \STATE{{\bf{Post-smoothing:}} $U^{(3)}=GS(A,F,U^{(2)},n_2,\epsilon);$}
 \STATE{$\%$ Check the convergence condition}
 \IF{$rel\_{res}=||F-AU^{(3)}||/res_0 < \epsilon$}
   \STATE{Break;}
 \ENDIF
 \STATE{$k = k + 1$, $U = U^{(3)}$;}
 \ENDWHILE
 \ENSURE the final solution $U$, 
         the iteration steps $k$ and 
         the error $rel\_{res}$. 
 \end{algorithmic}
\end{algorithm2e}
Next, we will introduce how to construct the grid hierarchy relationship, and how to construct the prolongation/restriction operator in detail.

\subsection{Grid hierarchy relationship}
In this section, we will give the detailed expression on how to generate a hierarchy of structured coarse grids based on the given unstructured fine grid. 
The main idea is to embed the vertices of the original unstructured fine grid into the first layer of the structured coarse grids regardless of the initial fine grid structure and then apply the GMG method after the second layer. 
And based on such a hierarchy of structured coarse grids, build the grid hierarchy relationship and the prolongation/restriction operator.

\subsubsection{The relationship between the fine grid and the coarse grid of the first layer}
Take the two-dimensional case as an example, and the three-dimensional case can be generalized similarly.
Given a unstructured triangulation partition $\mathcal{T}_h$ (cf.\figurename \ref{fine_grid}) with all vertexes stored in $\mathcal{V}^f=\{p_i^f=(x_i^f,y_i^f)|_{i=1}^{N^f}\}$, where $N^f=N=19$ represents the number of fine grid vertices, the following steps allow to construct an auxiliary region-tree structure:
\begin{figure}[!ht]
	\centerline{{\includegraphics[width=0.4\textwidth]{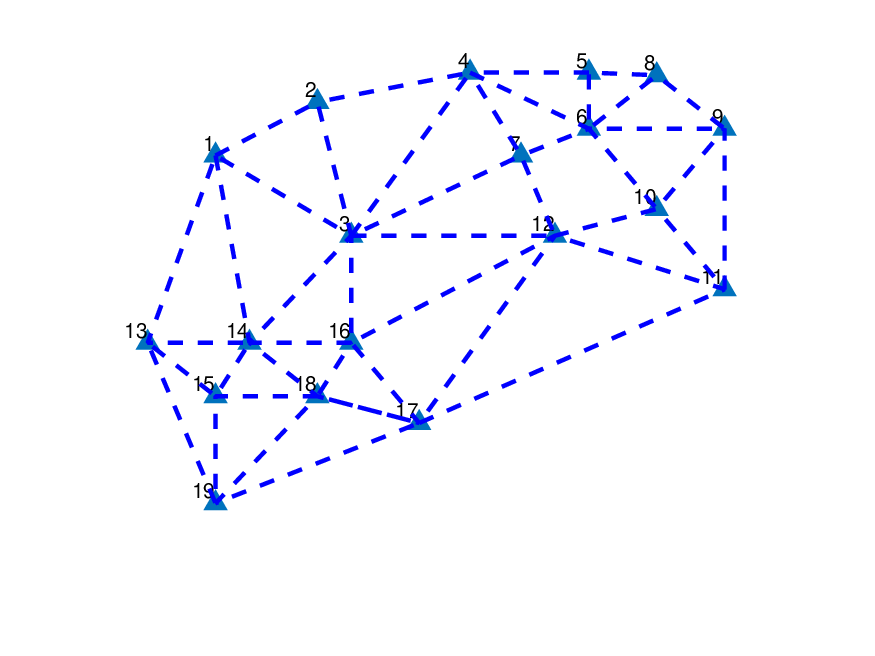}}}
	\caption{The profile of unstructured triangulation partition $\mathcal{T}_h$.}
	\label{fine_grid}
\end{figure}
\begin{itemize}
\item[Step 1.] 
Ignore the grid structure and find a square region $\mathcal{D}^1$ to overlap all vertexes containing in $\mathcal{V}^f$. 
A direct choice is to choose $\mathcal{D}^1$ as $[\mathcal{L}^{1},\mathcal{R}^{1}]\times[\mathcal{B}^{1},\mathcal{U}^{1}]$(cf.\figurename \ref{coarse_grid0}), where
\begin{align}
\mathcal{L}^{1} = \min_{p_i^f\in \mathcal{V}^f} x_i^f,
\mathcal{B}^{1} = \min_{p_i^f\in \mathcal{V}^f} y_i^f,&
\mathcal{L}^{1} = \min\{\mathcal{L}^{1}, \mathcal{B}^{1}\},
\mathcal{B}^{1} = \mathcal{L}^{1}; \\
\mathcal{R}^{1} = \max_{p_i^f\in \mathcal{V}^f} x_i^f,
\mathcal{U}^{1} = \max_{p_i^f\in \mathcal{V}^f} y_i^f,&
\mathcal{R}^{1} = \max\{\mathcal{R}^{1}, \mathcal{U}^{1}\},
\mathcal{U}^{1} = \mathcal{R}^{1}.
\end{align}
$\mathcal{D}^1$ corresponds to the root node of the auxiliary region-tree, which is initially marked as a leaf node.
$\mathcal{V}^1$ is its corresponding vertex set, which is empty at the initial time.
\begin{figure}[!ht]
	\centerline{{\includegraphics[width=0.4\textwidth]{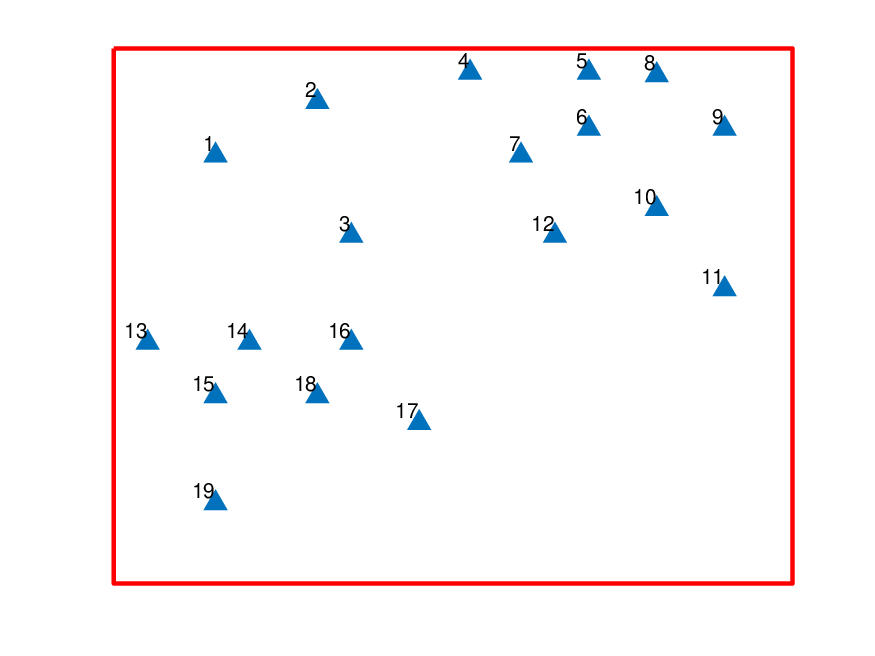}}}
	\caption{The profile of square region $\mathcal{D}^1$.}
	\label{coarse_grid0}
\end{figure}

\item[Step 2.] 
Insert all vertexes into the auxiliary region-tree in turn. 
Once the number of vertexes in the vertex set of a certain square region exceeds the given threshold (In our case, it is taken as 4), unmark its leaf node sign, split it into four child regions and mount the members in its vertex set to its four child regions.
Until all vertexes are inserted into the region-tree, the auxiliary coarse grid is obtained.
Specific steps are as follows: 
\begin{itemize}
\item[Step 2.1.] 
Insert vertexes 1-4 in sequence to the vertex set $\mathcal{V}^1$ of square region $\mathcal{D}^1$(cf.\figurename \ref{coarse_grid1}). 
The corresponding auxiliary region-tree is shown in \figurename \ref{region-tree1}.
\begin{figure}[!ht]
	\centerline{{\includegraphics[width=0.4\textwidth]{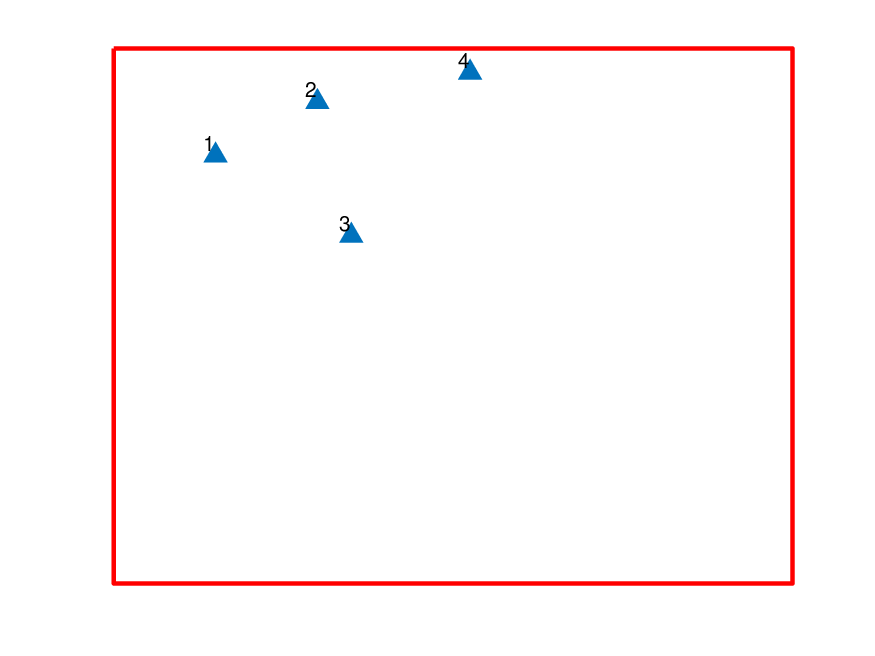}}}
	\caption{Coarse grid subdivision and vertex distribution corresponding to Step 2.1.}
	\label{coarse_grid1}
\end{figure}
\begin{figure}[!ht]
	\centerline{{\includegraphics[width=0.15\textwidth]{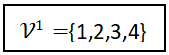}}}
	\caption{The auxiliary region-tree corresponding to Step 2.1.}
	\label{region-tree1}
\end{figure}
\item[Step 2.2.] 
The vertex set $\mathcal{V}^1=\{1,2,3,4\}$ of square region $\mathcal{D}^1$ is overloaded when inserting vertex 5(cf.\figurename \ref{coarse_grid2}(a)).
Unmark the leaf node sign of $\mathcal{D}^1$ and split it into four child regions $\mathcal{D}^2_j, j=1,2,3,4$(Note that these child regions are numbered in clockwise order and labeled as leaf nodes at the initial moment), where 
\begin{align}
\mathcal{D}^2_1=\left[\mathcal{L}^{1},\frac{\mathcal{L}^{1}+\mathcal{R}^{1}}{2}\right]\times\left[\frac{\mathcal{B}^{1}+\mathcal{U}^{1}}{2},\mathcal{U}^{1}\right],\\
\mathcal{D}^2_2=\left[\frac{\mathcal{L}^{1}+\mathcal{R}^{1}}{2},\mathcal{R}^{1}\right]\times\left[\frac{\mathcal{B}^{1}+\mathcal{U}^{1}}{2},\mathcal{U}^{1}\right],\\
\mathcal{D}^2_3=\left[\frac{\mathcal{L}^{1}+\mathcal{R}^{1}}{2},\mathcal{R}^{1}\right]\times\left[\mathcal{B}^{1},\frac{\mathcal{B}^{1}+\mathcal{U}^{1}}{2}\right],\\
\mathcal{D}^2_4=\left[\mathcal{L}^{1},\frac{\mathcal{L}^{1}+\mathcal{R}^{1}}{2}\right]\times\left[\mathcal{B}^{1},\frac{\mathcal{B}^{1}+\mathcal{U}^{1}}{2}\right].
\end{align}
Release $\mathcal{V}^1$ and insert vertexes in $\mathcal{V}^1$  into the vertex sets $\mathcal{V}^2_j, j=1,2,3,4$ of corresponding child regions(cf.\figurename \ref{coarse_grid2}(b)). 
Continue to insert vertexes 5-7 into vertex set $\mathcal{V}^2_2$(cf.\figurename \ref{coarse_grid2}(c)).
The corresponding auxiliary region-tree is shown in \figurename \ref{region-tree2}.
\begin{figure}[!ht]
	\centerline{
\subfigure[] {\includegraphics[width=0.40\textwidth]{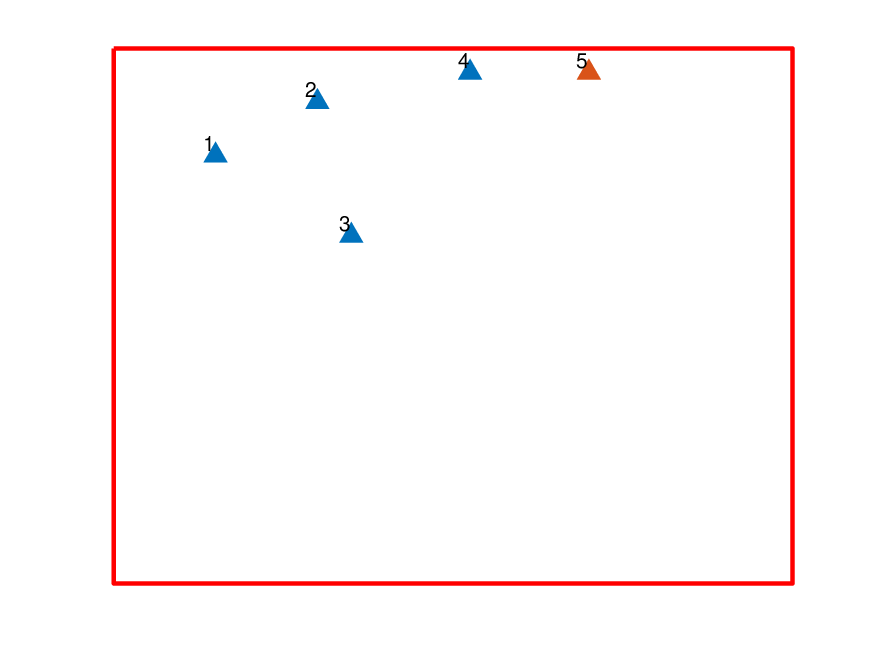}}\hspace{-0.65cm}
\subfigure[] {\includegraphics[width=0.40\textwidth]{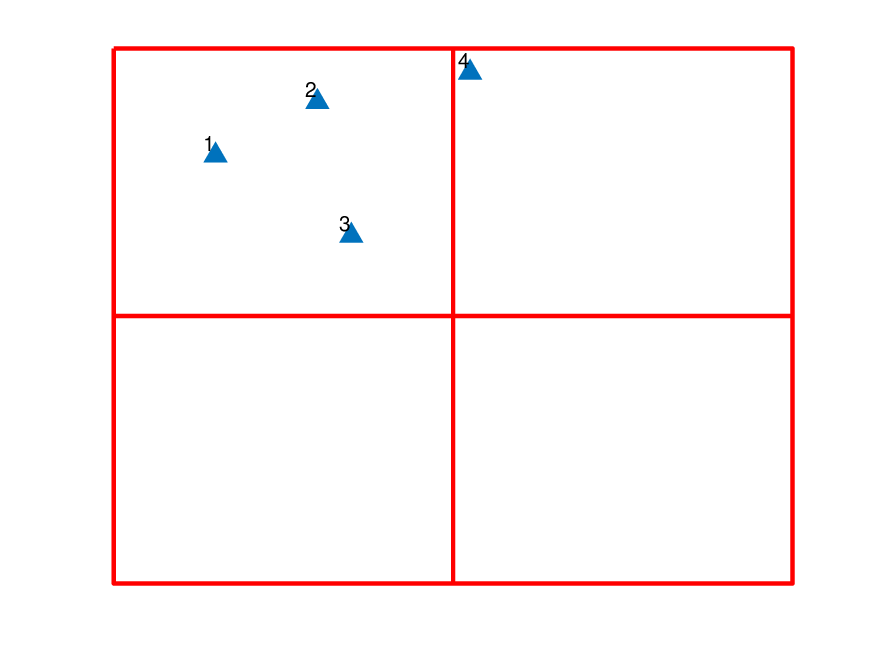}}\hspace{-0.65cm}
\subfigure[] {\includegraphics[width=0.40\textwidth]{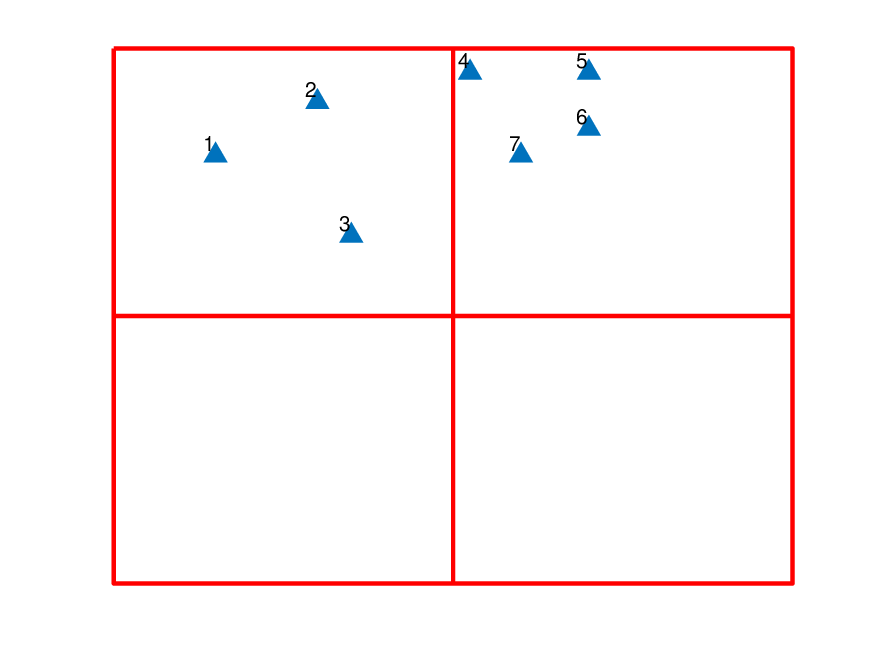}}}
	\caption{Coarse grid subdivision and vertex distribution corresponding to Step 2.2.}
	\label{coarse_grid2}
\end{figure}
\begin{figure}[!ht]
	\centerline{{\includegraphics[width=0.45\textwidth]{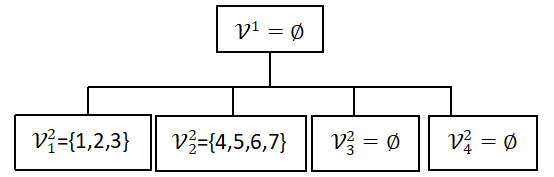}}}
	\caption{The auxiliary region-tree corresponding to Step 2.2.}
	\label{region-tree2}
\end{figure}
\item[Step 2.3.] 
The vertex set $\mathcal{V}^2_2=\{4,5,6,7\}$ of square region $\mathcal{D}^2_2$ is overloaded when inserting vertex 8(cf.\figurename \ref{coarse_grid3}(a)).
Unmark the leaf node sign of $\mathcal{D}^2_2$ and split it into four child regions $\mathcal{D}^3_j, j=1,2,3,4$(The partitioning mode is similar to $\mathcal{D}^1$). 
Release $\mathcal{V}^2_2$ and insert vertexes in $\mathcal{V}^2_2$  into the vertex sets $\mathcal{V}^3_j, j=1,2,3,4$ of corresponding child regions(cf.\figurename \ref{coarse_grid3}(b)). 
Continue to insert vertexes 8-9 into vertex set $\mathcal{V}^3_2$, vertexes 10-11 into vertex set $\mathcal{V}^3_3$, vertex 12 into vertex set $\mathcal{V}^3_4$, vertexes 13-16 into vertex set $\mathcal{V}^2_4$(cf.\figurename \ref{coarse_grid3}(c)).
The corresponding auxiliary region-tree is shown in \figurename \ref{region-tree3}.
\begin{figure}[!ht]
	\centerline{
\subfigure[] {\includegraphics[width=0.4\textwidth]{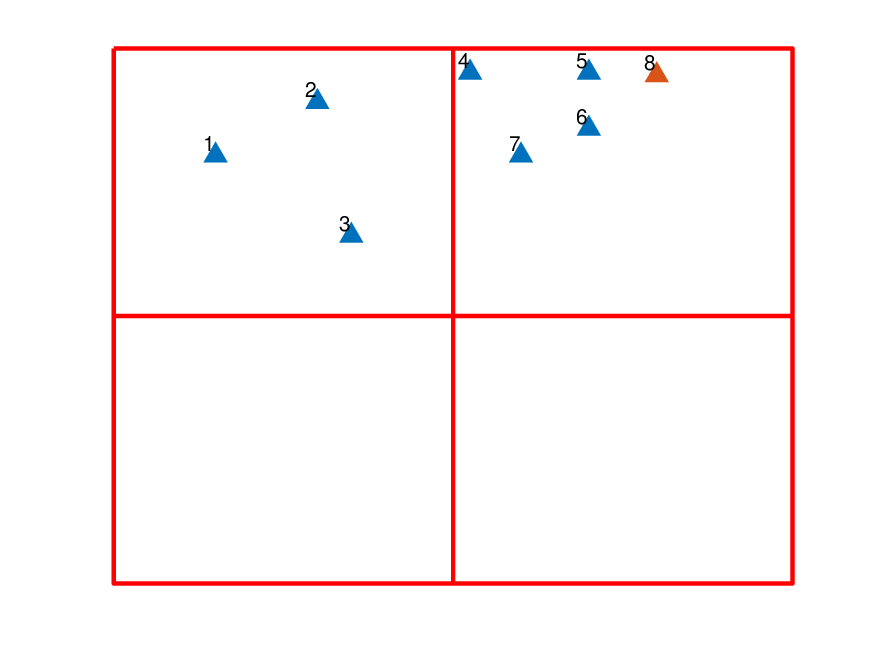}}\hspace{-0.65cm}
\subfigure[] {\includegraphics[width=0.4\textwidth]{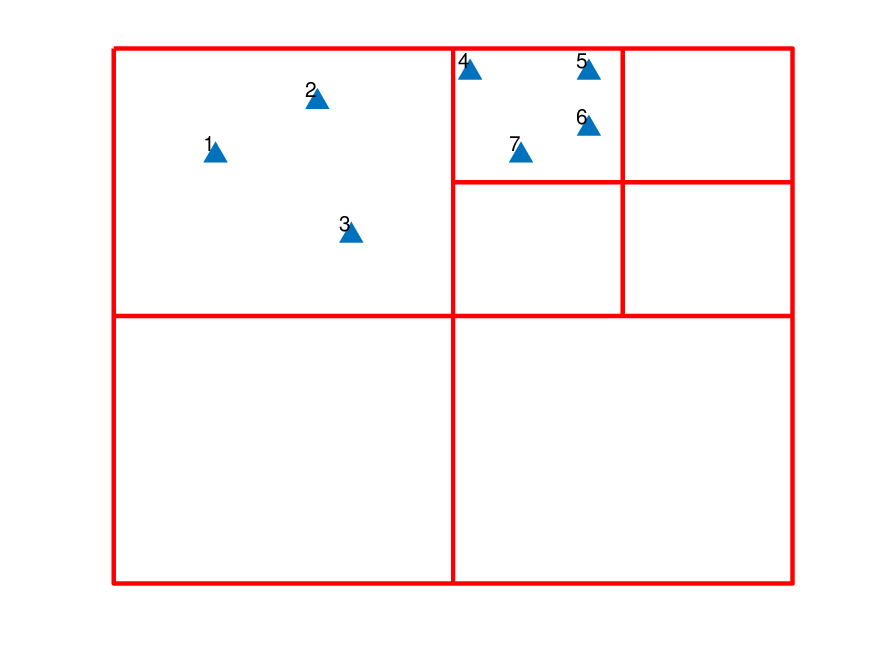}}\hspace{-0.65cm}
\subfigure[] {\includegraphics[width=0.4\textwidth]{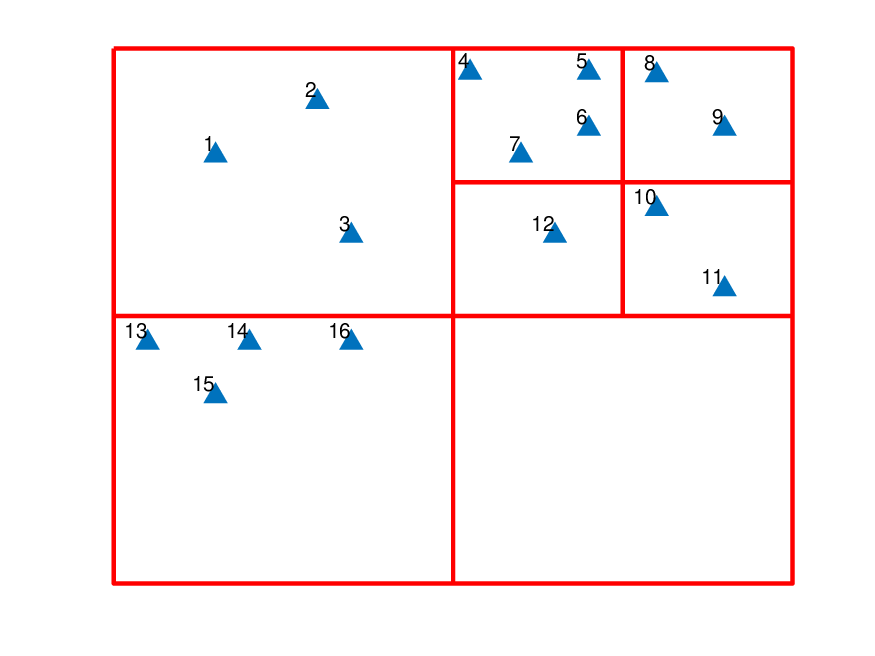}}}
	\caption{Coarse grid subdivision and vertex distribution corresponding to Step 2.3.}
	\label{coarse_grid3}
\end{figure}
\begin{figure}[!ht]
	\centerline{{\includegraphics[width=0.6\textwidth]{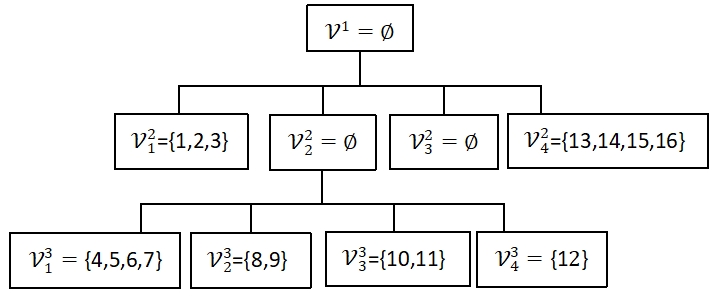}}}
	\caption{The auxiliary region-tree corresponding to Step 2.3.}
	\label{region-tree3}
\end{figure}
\item[Step 2.4.] 
The vertex set $\mathcal{V}^2_4=\{13,14,15,16\}$ of square region $\mathcal{D}^2_4$ is overloaded when inserting vertex 17(cf.\figurename \ref{coarse_grid4}(a)).
Unmark the leaf node sign of $\mathcal{D}^2_4$ and split it into four child regions $\mathcal{D}^3_j, j=5,6,7,8$(The partitioning mode is similar to $\mathcal{D}^1$). 
Release $\mathcal{V}^2_4$ and insert vertexes in $\mathcal{V}^2_4$  into the vertex sets $\mathcal{V}^3_j, j=5,6,7,8$ of corresponding child regions(cf.\figurename \ref{coarse_grid4}(b)). 
Continue to insert vertexes 17-18 into vertex set $\mathcal{V}^3_6$, vertex 19 into vertex set $\mathcal{V}^3_8$(cf.\figurename \ref{coarse_grid4}(c)).
The corresponding auxiliary region-tree is shown in \figurename \ref{region-tree4}.
At this point, all the vertexes 1-19 in $\mathcal{V}^f$ are inserted into the region-tree.
\begin{figure}[!ht]
	\centerline{
\subfigure[] {\includegraphics[width=0.4\textwidth]{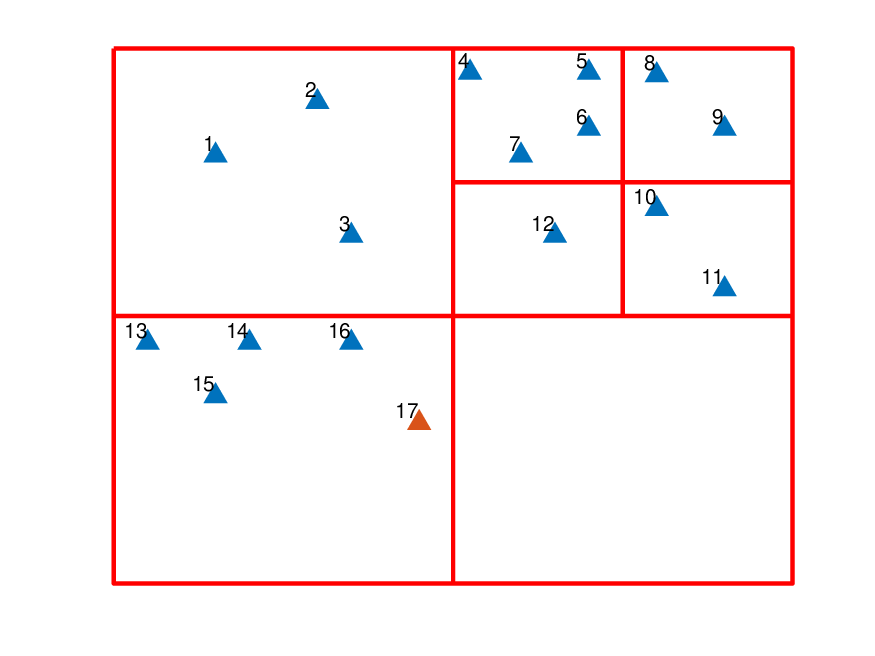}}\hspace{-0.65cm}
\subfigure[] {\includegraphics[width=0.4\textwidth]{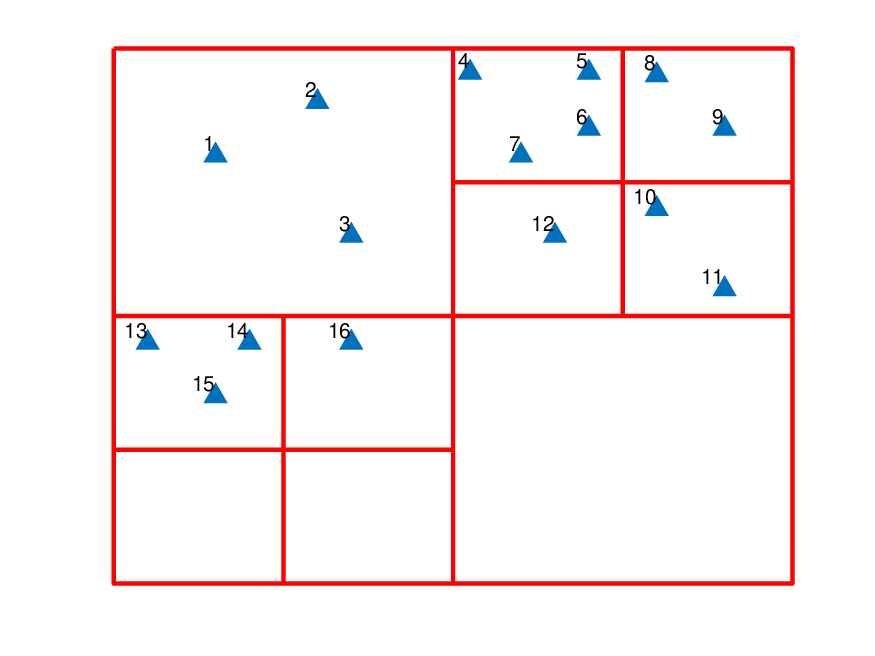}}\hspace{-0.65cm}
\subfigure[] {\includegraphics[width=0.4\textwidth]{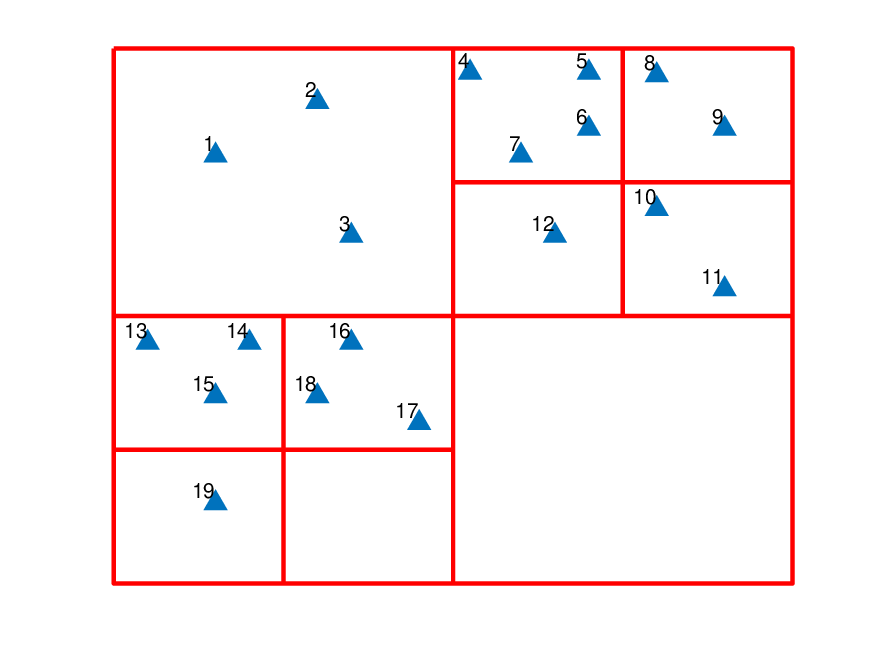}}}
	\caption{Coarse grid subdivision and vertex distribution corresponding to Step 2.4.}
	\label{coarse_grid4}
\end{figure}
\begin{figure}[!ht]
	\centerline{{\includegraphics[width=0.95\textwidth]{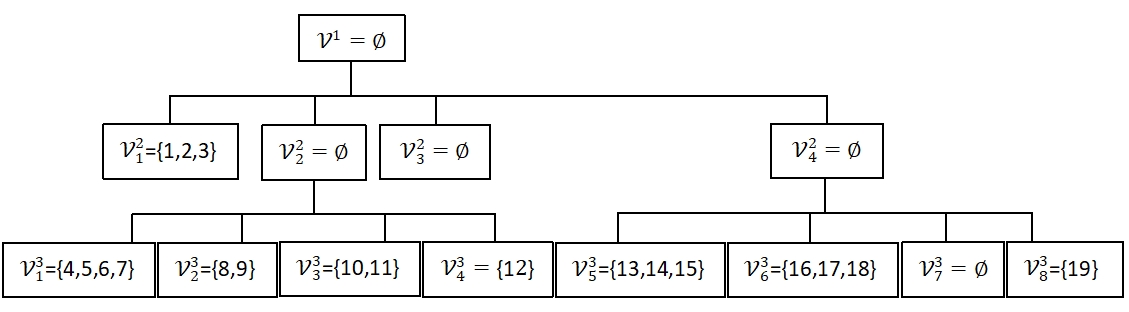}}}
	\caption{The auxiliary region-tree corresponding to Step 2.4.}
	\label{region-tree4}
\end{figure}
\end{itemize}

\item[Step 3.] 
Eliminate the leaf regions(i.e. $\mathrm{isleaf}(\mathcal{D}^i_j)=1$) that do not contain vertexes(i.e. $\mathcal{V}^i_j=\emptyset$), to obtain the final coarse grid subdivision $\mathcal{T}_H=\cup_{\mathcal{V}^i_j\neq\emptyset \& \mathrm{isleaf}(\mathcal{D}^i_j)=1}\mathcal{D}^i_j$ with all vertexes stored in 
$\mathcal{V}^c=\cup_{\mathcal{D}^i_j\in\mathcal{T}_H}\left\{(\mathcal{L}^i_j,\mathcal{U}^i_j),(\mathcal{R}^i_j,\mathcal{U}^i_j),(\mathcal{R}^i_j,\mathcal{B}^i_j),(\mathcal{L}^i_j,\mathcal{B}^i_j)\right\}$(There are $N^c=17$ coarse grid vertices, which are also numbered clockwise), see \figurename \ref{coarse_grid_1}. 
The corresponding auxiliary region-tree is shown in \figurename \ref{region-tree}, whose height is 3.
\begin{figure}[!ht]
	\centerline{{\includegraphics[width=0.4\textwidth]{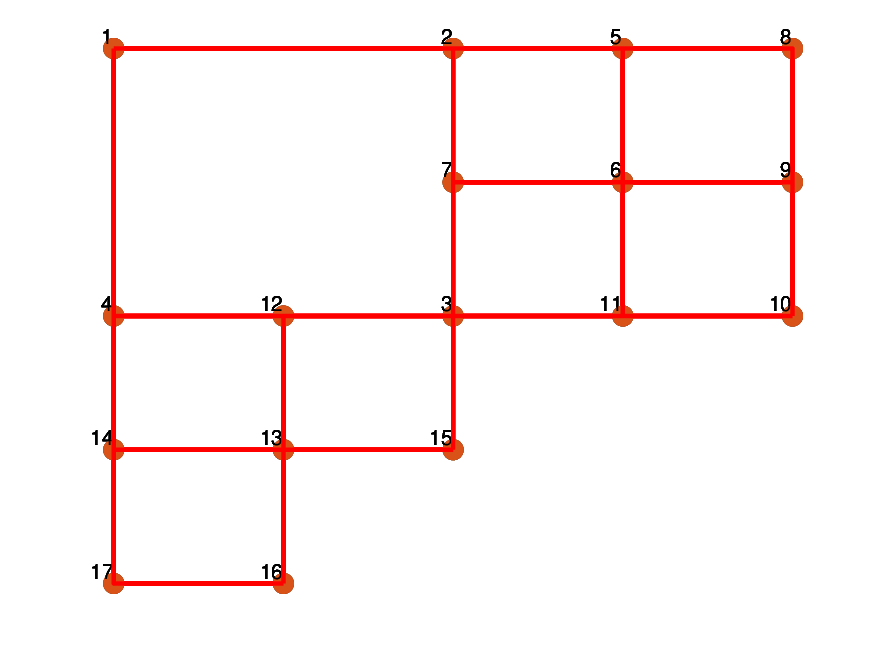}}}
	\caption{The profile of structured coarse grid subdivision $\mathcal{T}_H$.}
	\label{coarse_grid_1}
\end{figure}
\begin{figure}[!ht]
	\centerline{{\includegraphics[width=0.90\textwidth]{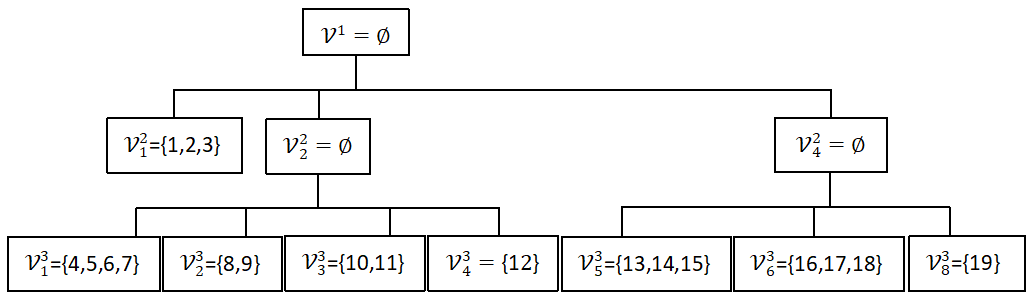}}}
	\caption{The auxiliary region-tree corresponding to structured coarse grid subdivision $\mathcal{T}_H$.}
	\label{region-tree}
\end{figure}

\item[Step 4.] 
Find the correspondence between the fine and coarse grid vertexes. 
Using the auxiliary coarse grid as the medium, the fine grid vertexes $\mathcal{V}^i_j$ located in the coarse grid region $\mathcal{D}^i_j$ are associated with its four vertices $\left\{(\mathcal{L}^i_j,\mathcal{U}^i_j),(\mathcal{R}^i_j,\mathcal{U}^i_j),(\mathcal{R}^i_j,\mathcal{B}^i_j),(\mathcal{L}^i_j,\mathcal{B}^i_j)\right\}$, see \figurename \ref{relationship}.
Specifically, it is as follows
\begin{align}
\{\textcircled{1},\textcircled{2},\textcircled{3}\}
&\rightarrow\mathcal{D}^2_1\rightarrow
\{\textcircled{1},\textcircled{2},\textcircled{3},\textcircled{4}\}, \\
\{\textcircled{4},\textcircled{5},\textcircled{6},\textcircled{7}\}
&\rightarrow\mathcal{D}^3_1\rightarrow
\{\textcircled{2},\textcircled{5},\textcircled{6},\textcircled{7}\}, \\
\{\textcircled{8},\textcircled{9}\}
&\rightarrow\mathcal{D}^3_2\rightarrow
\{\textcircled{5},\textcircled{8},\textcircled{9},\textcircled{6}\}, \\
\{\textcircled{10},\textcircled{11}\}
&\rightarrow\mathcal{D}^3_3\rightarrow
\{\textcircled{6},\textcircled{9},\textcircled{10},\textcircled{11}\}, \\
\{\textcircled{12}\}
&\rightarrow\mathcal{D}^3_4\rightarrow
\{\textcircled{7},\textcircled{6},\textcircled{11},\textcircled{3}\}, \\
\{\textcircled{13},\textcircled{14},\textcircled{15}\}
&\rightarrow\mathcal{D}^3_5\rightarrow
\{\textcircled{4},\textcircled{12},\textcircled{13},\textcircled{14}\}, \\
\{\textcircled{16},\textcircled{17},\textcircled{18}\}
&\rightarrow\mathcal{D}^3_6\rightarrow
\{\textcircled{12},\textcircled{3},\textcircled{15},\textcircled{13}\}, \\
\{\textcircled{19}\}
&\rightarrow\mathcal{D}^3_8\rightarrow
\{\textcircled{14},\textcircled{13},\textcircled{16},\textcircled{17}\}.
\end{align}
It is easy to find that the relationship between fine and coarse grid vertexes is many-to-many, which makes them more closely related to each other.
\begin{figure}[!ht]
	\centerline{{\includegraphics[width=0.4\textwidth]{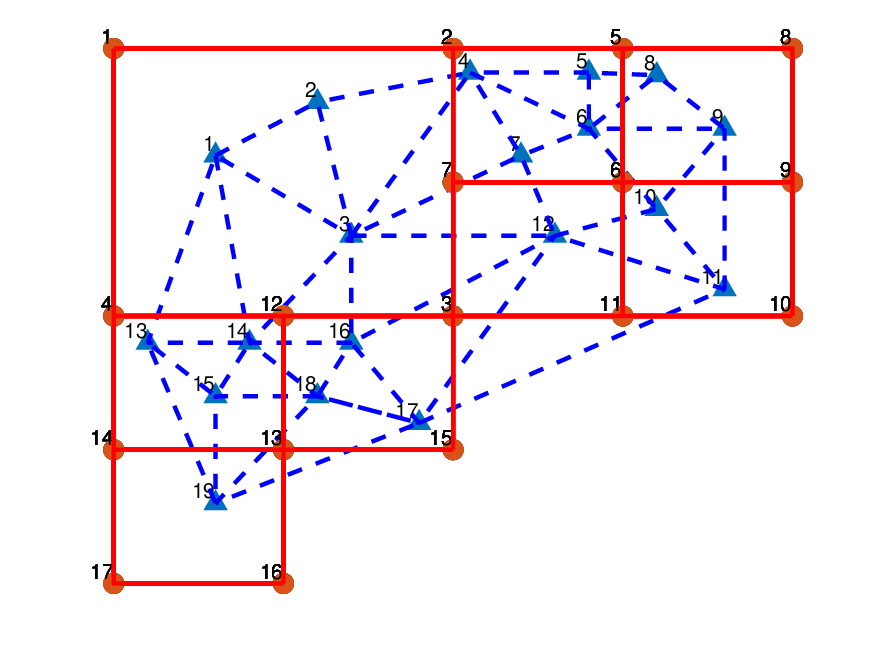}}}
	\caption{The profile of correspondence between fine and coarse grid vertexes.}
	\label{relationship}
\end{figure}
\end{itemize}
\begin{theorem}
Assume that the maximum distance $d_{max}$ and minimum distance $d_{min}$ between the vertexes of the given initial fine grid are satisfied
\begin{align}\label{hypothesis}
\frac{d_{max}}{d_{min}}=N^q,
\end{align}
where $N$ represents the number of fine grid vertexes, q is a small number, such as q=2, and
\begin{align}
d_{min} = \min_{p_i^f,p_j^f\in \mathcal{V}^f} \Vert p_i^f-p_j^f \Vert_{L_2},\\
d_{max} = \max_{p_i^f,p_j^f\in \mathcal{V}^f} \Vert p_i^f-p_j^f \Vert_{L_2}.
\end{align}
$d_{max}$ is also called the diameter of the region $\Omega$, denoted by $\mathrm{diam}(\Omega)=d_{max}$.
Based on the above assumption, the computational complexity of building such a region-tree is $\mathcal{O}\left(q N\log_2 N\right)$.
\end{theorem}

\begin{proof} 
According to the definition of the minimum distance $d_{min}$, it is easy to find that the distance between any two fine grid vertexes $p_i^f,p_j^f\in \mathcal{V}^f$ is satisfied
\begin{align}
\Vert p_i^f-p_j^f \Vert_{L_2} \geq d_{min}.
\end{align}
And then it can be obtained that the diameter of each leaf region $\mathcal{D}^i_j$ satisfies
\begin{align}
\mathrm{diam}(\mathcal{D}^i_j) \geq d_{min}.
\end{align}
According to the construction process of the region-tree, it can be further concluded that there is a relationship between each leaf region $\mathcal{D}^i_j$ and root region $\mathcal{D}^1$
\begin{align}
\mathrm{diam}(\mathcal{D}^i_j) = 2^{-(i-1)}\mathrm{diam}(\mathcal{D}^1).
\end{align}
Based on the definition of the root region $\mathcal{D}^1$, it is not difficult to get the relationship between its diameter and the maximum distance $d_{max}$ 
\begin{align}
\mathrm{diam}(\mathcal{D}^1) \leq \sqrt{2}d_{max}.
\end{align}
Therefore, it can be deduced
\begin{align}
d_{min} \leq \mathrm{diam}(\mathcal{D}^i_j) = 2^{-(i-1)}\mathrm{diam}(\mathcal{D}^1) \leq 2^{-(i-1)}\sqrt{2}d_{max}=2^{\frac{3}{2}-i}d_{max}.
\end{align}
Combined with the hypothesis \eqref{hypothesis}, it can be concluded that the depth $i$ of region $\mathcal{D}^i_j$ is satisfied
\begin{align}
i \leq \log_2\left(\frac{d_{max}}{d_{min}}\right)+\frac{3}{2}=\log_2\left(N^q\right)+\frac{3}{2}=q\log_2 N+\frac{3}{2} \approx \mathcal{O}\left(q\log_2 N\right).
\end{align}
According to the previous partitioning method, each leaf region $\mathcal{D}^i_j$ contains at least one fine grid vertex, so there are at most $N$ leaf regions in the auxiliary region-tree. 
Thus, the total computational complexity is $\mathcal{O}\left(q N\log_2 N\right)$.
\end{proof}

\subsubsection{The relationship between the coarse grids}
Previously, the auxiliary coarse grid was obtained by gradually subdividing the square region $\mathcal{D}^1$ from coarse to fine.
In this section, in order to get the coarse grids after the second layer, we merge the regions from finer to coarser little by little until the square region $\mathcal{D}^1$ comes along.
This is done as follows:
\begin{itemize}
\item[Step 1.] Merge $\mathcal{D}^3_j, j=1,2,3,4$ to obtain $\mathcal{D}^2_2$, merge $\mathcal{D}^3_j, j=5,6,8$ to obtain $\mathcal{D}^2_4$.
The following is the corresponding relationship between the first coarse grid(cf.\figurename \ref{coarse_grid_123}(a)) and the second coarse grid(cf.\figurename \ref{coarse_grid_123}(b)).
\begin{align}
\{\textcircled{1},...,\textcircled{4}\}
&\rightarrow\mathcal{D}^2_1\rightarrow
\{\textcircled{1},...,\textcircled{4}\}, \\
\{\textcircled{2},\textcircled{5},...,\textcircled{11},\textcircled{3}\}
&\rightarrow\mathcal{D}^2_2\rightarrow
\{\textcircled{2},\textcircled{5},\textcircled{6},\textcircled{3}\}, \\
\{\textcircled{4},\textcircled{12},...,\textcircled{14},\textcircled{3},\textcircled{15},...,\textcircled{17}\}
&\rightarrow\mathcal{D}^2_4\rightarrow
\{\textcircled{4},\textcircled{3},\textcircled{7},\textcircled{8}\}.
\end{align}
\item[Step 2.] Merge $\mathcal{D}^2_j, j=1,2,4$ to obtain $\mathcal{D}^1$.  
The following is the corresponding relationship between the second coarse grid(cf.\figurename \ref{coarse_grid_123}(b)) and the third coarse grid(cf.\figurename \ref{coarse_grid_123}(c)).
\begin{align}
\{\textcircled{1},...,\textcircled{8}\}
&\rightarrow\mathcal{D}^1\rightarrow
\{\textcircled{1},...,\textcircled{4}\}.
\end{align}
\end{itemize}
Similarly, the relationship between the coarse grid vertices of each layer is many-to-many.
\begin{figure}[!ht]
	\centerline{
\subfigure[The first layer.] {\includegraphics[width=0.4\textwidth]{img/coarse_grid_1.eps}}\hspace{-0.60cm}
\subfigure[The second layer.] {\includegraphics[width=0.4\textwidth]{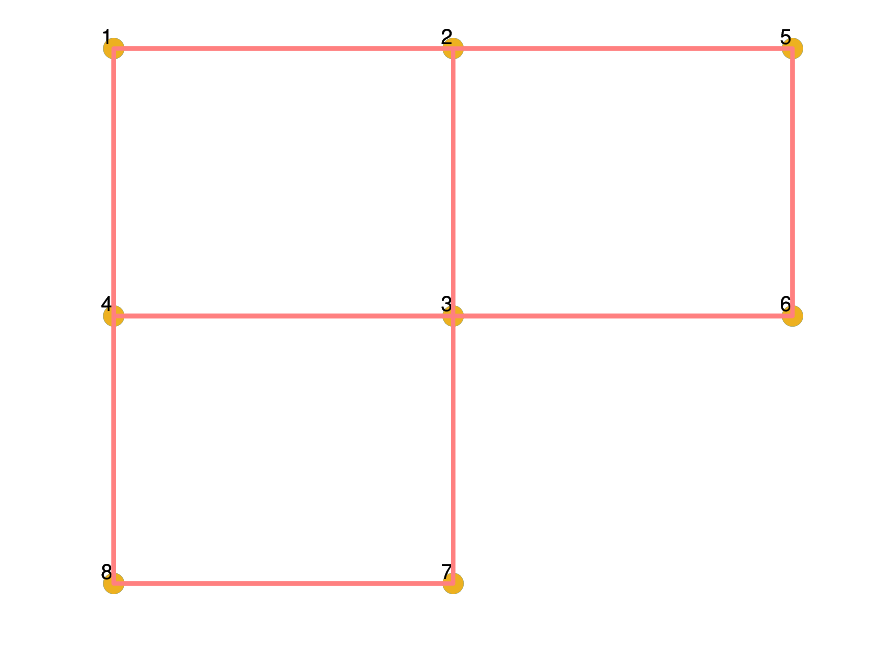}}\hspace{-0.60cm}
\subfigure[The third layer.] {\includegraphics[width=0.4\textwidth]{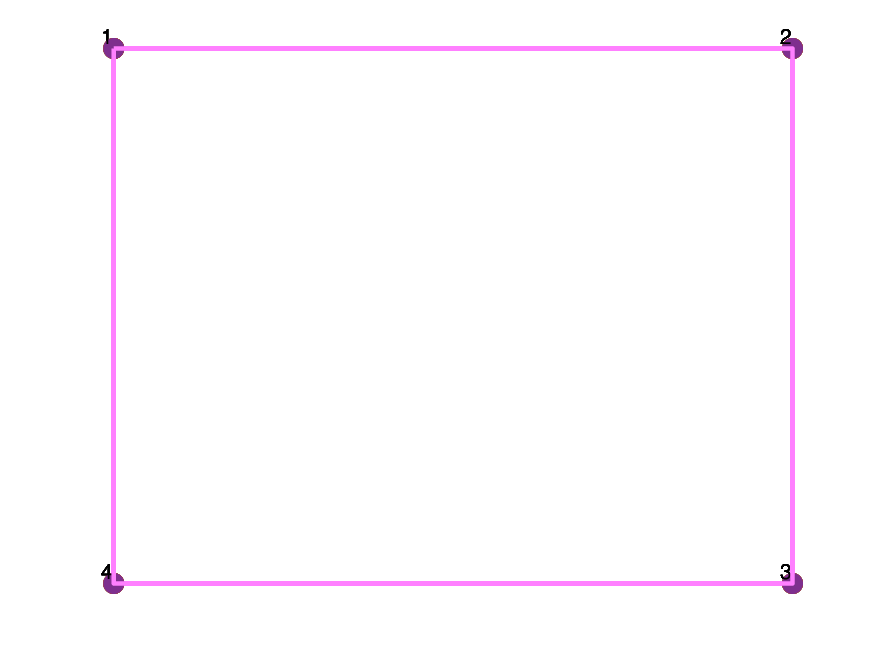}}}
	\caption{The profile of structured coarse grids of each layer.}
	\label{coarse_grid_123}
\end{figure}

\subsection{Prolongation/restriction operator}\label{section:PR}
In the previous section, we found the relationship between the grid vertices of  each layer. 
On this basis, the prolongation/restriction operator is constructed by using the bilinear interpolation function, and then the stiffness matrix corresponding to each layer is obtained.

Taking the prolongation operator $\mathcal{P}$ between the fine grid and the coarse grid of the first layer as an example, the element in column $l(l\leq N^c)$ of row $k(k\leq N^f=N)$ represents the interpolation weight of the fine grid vertex $p_k^f=(x_k^f,y_k^f)\in\mathcal{V}^f$ with respect to the coarse grid vertex $p_l^c=(x_l^c,y_l^c)\in\mathcal{V}^c$(Assume that fine grid vertex $p_k^f$ is located in the coarse grid region $\mathcal{D}^i_j$, and coarse grid vertex $p_l^c$ is a vertex of the coarse grid region $\mathcal{D}^i_j$), which corresponds to the first degree of freedom $u_1$ of the variable $u$ and can be expressed as
\begin{align}\label{wkl}
w_{kl}=w_{kl}^x\times w_{kl}^y,
\end{align} 
where $w_{kl}^x$ and $w_{kl}^y$ are the interpolation weights of fine grid vertex $p_k^f$ to the coarse grid vertex $p_l^c$ with respect to $x$ direction and $y$ direction, respectively
\begin{align}\label{wkl_xy}
w_{kl}^x = \left\{
\begin{array}{cc}
\frac{x_l^c-x_k^f}{\mathcal{R}^i_j-\mathcal{L}^i_j}, & x_k^f\leq x_l^c,\\
\frac{x_k^f-x_l^c}{\mathcal{R}^i_j-\mathcal{L}^i_j}, & x_k^f\geq x_l^c,\\
0,&\mbox{otherwise},\\
\end{array}
\right.
w_{kl}^y = \left\{
\begin{array}{cc}
\frac{y_l^c-y_k^f}{\mathcal{U}^i_j-\mathcal{B}^i_j}, & y_k^f\leq y_l^c,\\
\frac{y_k^f-y_l^c}{\mathcal{U}^i_j-\mathcal{B}^i_j}, & y_k^f\geq y_l^c,\\
0,&\mbox{otherwise}.\\
\end{array}
\right.
\end{align}
While the element in column $l(l>N^c)$ of row $k(k>N^f=N)$ denotes the interpolation weight of the fine grid vertex $p_{k-N^f}^f\in\mathcal{V}^f$ with respect to the coarse grid vertex $p_{l-N^c}^c\in\mathcal{V}^c$, except that it corresponds to the second degree of freedom $u_2$ of the variable $u$ and its expression is
\begin{align}\label{wkl_}
w_{kl}=w_{k-N^f,l-N^c}.
\end{align} 
The rest of the elements are all zero.
The prolongation/restriction operators between the coarse grids after the second layer are constructed in a similar way. 
Substituting the above construction algorithm of prolongation/restriction operator into the first step of Algorithm \ref{mg} yields our V-ASMG method.
\begin{theorem}
By adjusting the maximum number of fine grid vertexes that can be contained in each leaf region $\mathcal{D}^i_j$ to be 9, it is proved that the V-ASMG method can degenerate to the classical GMG method in the case of uniform structure grid.
\end{theorem}

\begin{proof} 
Take the $4\times4$ structured fine grid in \figurename \ref{structured_fine_coarse_grid}(a) as an example.
According to the coarse grid subdivision way of the V-ASMG method, its corresponding structured coarse grid can be obtained as shown in \figurename \ref{structured_fine_coarse_grid}(b).
Considering the most central coarse grid vertex in \figurename \ref{structured_fine_coarse_grid}(b), which is related to all fine grid vertices in \figurename \ref{structured_fine_coarse_grid}(a), and their corresponding weights are shown in in \figurename \ref{structured_fine_coarse_grid}(c)(All unmarked weights are 0).
Other coarse grid vertexes satisfy the same relationship.
It is not difficult to find that whether the relationship between the fine and coarse grid vertexes or the corresponding weights are consistent with the classical GMG method.
Certified.
\begin{figure}[!ht]
	\centerline{
\subfigure[Structured fine grid.] {\includegraphics[width=0.4\textwidth]{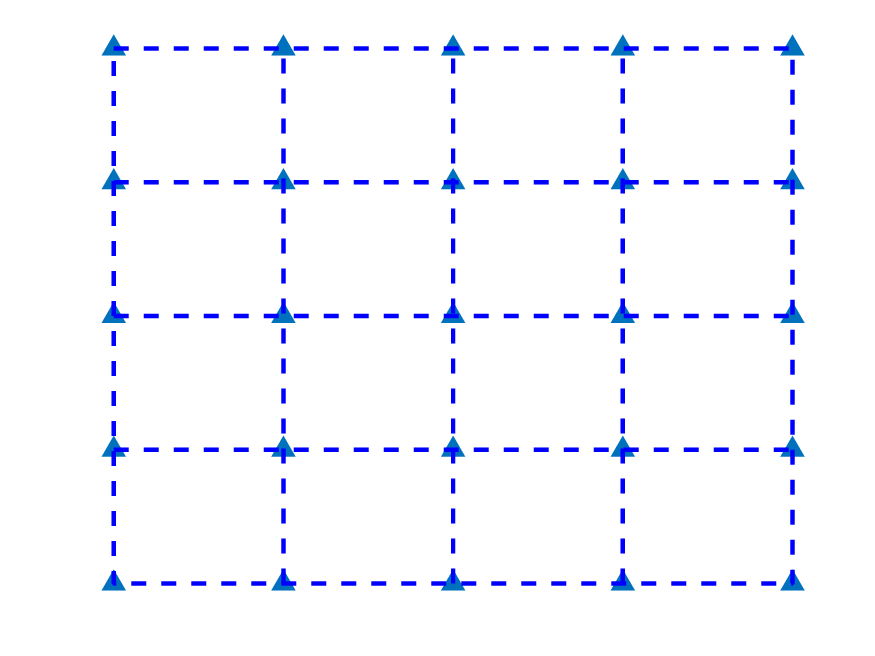}}\hspace{-0.60cm}
\subfigure[Structured coarse grid.] {\includegraphics[width=0.4\textwidth]{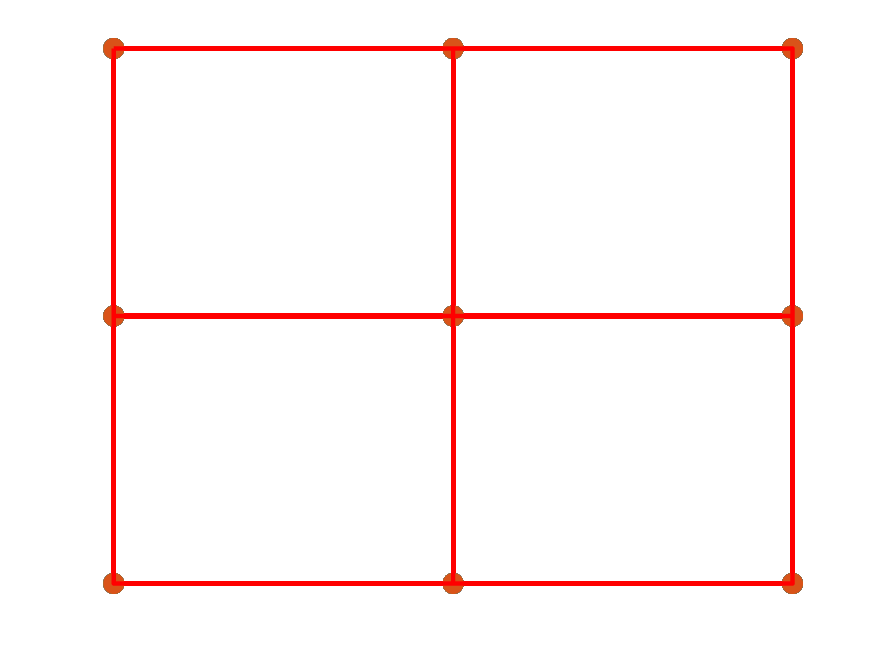}}\hspace{-0.60cm}
\subfigure[Weights.] {\includegraphics[width=0.4\textwidth]{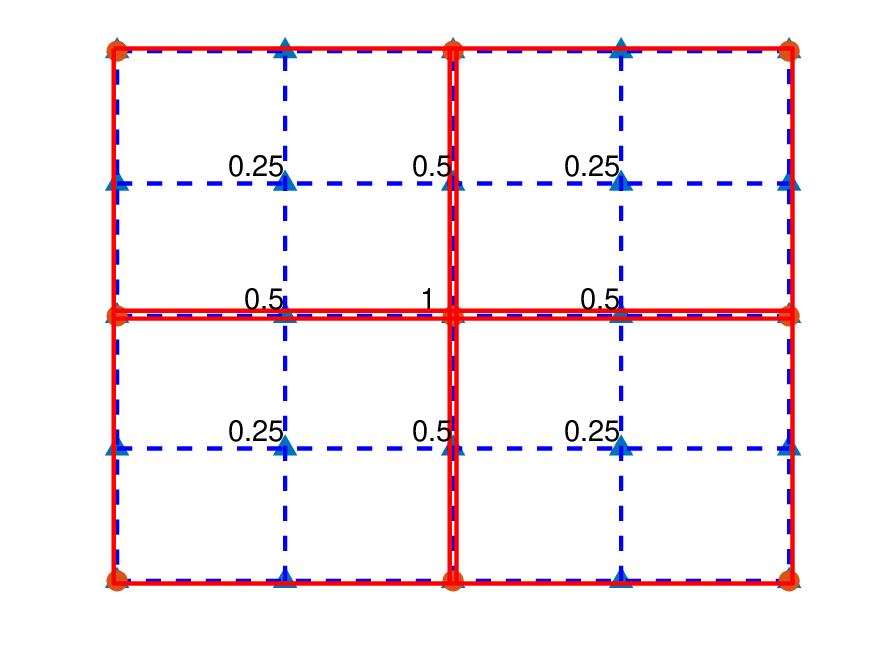}}}
	\caption{The profile of structured fine and coarse grids and their weights.}
	\label{structured_fine_coarse_grid}
\end{figure}
\end{proof}

%%%%%%%%%%%%%%%%%%%%%%%%%%%%%%%%%%%%%%%%%%%%%%%%%%%%%
\section{V-ASMG preconditioner}
%%%%%%%%%%%%%%%%%%%%%%%%%%%%%%%%%%%%%%%%%%%%%%%%%%%%%
In this section, some theories of the V-ASMG method as a   preconditioner are given, and it is proved that the V-ASMG method can indeed accelerate the iterative convergence speed by reducing the matrix condition number.

Let $\mathcal{T}_0=\cup_{\mathrm{isleaf}(\mathcal{D}^i_j)=1}\mathcal{D}^i_j$ be the union of all leaf regions $\mathcal{D}_j^i$, including those that do not contain fine grid vertexes. 
For the convenience of representation, these leaf regions are briefly denoted as $\tau_0$.
Let $\Omega_0$ denote the domain determined by $\mathcal{T}_0$, namely, $\bar{\Omega}_0=\cup_{\tau_0\in \mathcal{T}_0}\bar{\tau}_0$. 
According to the establishment process of the previous section, it is not difficult to obtain that $\Omega\subset\Omega_0$ and $h_0:=\max_{\tau_0\in \mathcal{T}_0}\mathrm{diam}(\tau_0)\approx h:=\max_{\tau\in \mathcal{T}_h}\mathrm{diam}(\tau)$.
The $H^{1}$ piecewisely bilinear finite element space is used to be the auxiliary space $V_0$ which vanishes on $\Omega_0/\Omega$.
The main technical difficulty lies in how to establish the connection between the piecewisely linear finite element space $V_{h}$ which is regarded as the fine space and the piecewisely bilinear finite element space $V_0$ that is seen as the coarse space.

Define two linear operators $A_h: V_{h} \rightarrow V_{h}$ and $A_0: V_0 \rightarrow V_0$, which are associated with $V_{h}$ and $V_0$, respectively,  as follows
\begin{align}
 (A_h u, v)=a(u, v)&, \text { for any } v \in V_{h};\\
 (A_0 u, w)=a(u, w)&, \text { for any } w \in V_0. 
\end{align}
Also denote by $\|\cdot\|$ the norm induced by $(\cdot,\cdot)$.
As is known to all, the spectral radius and the condition number of operators $A_h$ and $A_0$ are $O\left(h^{-2}\right)$ and $O\left(h_0^{-2}\right)$ (see \cite{b1993}).
Since $A_0$ is defined on the piecewisely bilinear finite element space $V_0$, the coarse problem defined on which can then be solved by many efficient ready-made solvers. 
Denote $B_0: V_0 \rightarrow V_0$ to be such a coarse solver, which is chosen to be a direct solver with  $B_0=A_0^{-1}$ in this paper. 
In addition, on the fine space $V_{h}$, we still need a symmetric and positive definite smoother $S_h: V_{h} \rightarrow V_{h}$, such as Jacobi or symmetric Gauss-Seidel method. 
At last, in order to connect the coarse space with the fine space, a prolongation operator $\Pi: V_0 \rightarrow V_{h}$ and a restriction operator $\Pi^\top: V_{h} \rightarrow V_0$ are required satisfying
\begin{equation}
(v, \Pi w)=\left(\Pi^\top v, w\right), \text { for any } v \in V_{h} \text { and } w \in V_0. 
\end{equation}
Then, following the definition in \cite{b1993, x1996}, the V-ASMG preconditioner $B_h: V_{h} \rightarrow V_{h}$ is defined as
\begin{equation}
B_h=S_h+\Pi B_0 \Pi^{\top}.
\end{equation}
As demonstrated in \cite{x1996}, the following theorem is valid.
\begin{theorem}\label{thm_mg}
Assume that for all $v \in V_{h}$ and $w \in V_0$,
\begin{align}
(A_0 w, w) \lesssim (B_0 A_0 w&, A_0 w) \lesssim (A_0 w, w) ,\label{BA}\\
\rho_{A_h}^{-1}(v, v) \lesssim(S_h v&, v) \lesssim \rho_{A_h}^{-1}(v, v), \label{R_rho}\\
\|\Pi w\|_{A_h} &\lesssim\|w\|_{A_0},\label{Pi}
\end{align}
moreover, there exists a linear operator $P: V_{h} \rightarrow V_0$ such that
\begin{align}
\|P v\|_{A_0} &\lesssim\|v\|_{A_h}, \label{P}\\
\|v-\Pi P v\|_{0,\Omega}^{2} &\lesssim \rho_{A_h}^{-1}\|v\|_{A_h}^{2},\label{err_P}
\end{align}
where $\rho_{A_h}=\rho(A_h)$ is the spectral radius of $A_h$.
Then the condition number of $B_h A_h$ satisfies $\kappa(B_h A_h) \lesssim O(1)$.
\end{theorem}
\begin{remark}
Theorem \ref{thm_mg} states that the condition number of $B_h A_h$ is uniformly bounded, that is to say, $B_h$ is a good preconditioner for $A_h$. 
Thus, $B_h$ can be used as an effective preconditioner of the PCG method for solving the linear equations \eqref{equ}.
\end{remark}
\begin{remark}
It should be noted that the introduction of operator $P$ is only for the convenience of theoretical analysis. 
In practice, only $B_0, S_h$ and $\Pi$ are required. 
Moreover, the matrix representation of the restriction operator $\Pi^{\top}$ is usually taken as the transpose of the matrix representation of the prolongation operator $\Pi$.
\end{remark}

The next step is to construct the V-ASMG preconditioner that satisfies all of the conditions in Theorem \ref{thm_mg}, namely, inequalities \eqref{BA}-\eqref{err_P}. To this end, we first introduce the orthogonal projections $Q_h: H^{1} \rightarrow V_{h}$ and $Q_0: H^{1} \rightarrow V_{0}$: for any $u\in H^1$,
\begin{align}
a\left(Q_h u, v\right)=a(u, v)&, \text { for any } v \in V_{h},\\
a\left(Q_0 u, w\right)=a(u, w)&, \text { for any } w \in V_{0},
\end{align}
then, the following error estimates(see \cite{bs1992}) can be obtained
\begin{align}
\left\|\nabla Q_h u\right\| &\lesssim\|\nabla u\|, \label{Q}\\
\left\|u-Q_h u\right\|+h\left\|\nabla\left(u-Q_h u\right)\right\| &\lesssim h|u|,\label{err_Q} 
\end{align}
and 
\begin{align}
\left\|\nabla Q_0 u\right\| &\lesssim\|\nabla u\|, \label{Q0}\\
\left\|u-Q_0 u\right\|+h_0\left\|\nabla\left(u-Q_0 u\right)\right\| &\lesssim h_0|u|.\label{err_Q0} 
\end{align}
The choices of $B_0$ and $S_h$ naturally satisfy conditions \eqref{BA}-\eqref{R_rho}. 
Moreover, the definition of the operator $\Pi$ is not so difficult, and it can be defined as $\Pi=Q_h$. 
Since $\Pi$ is a bilinear interpolation operator, it can naturally meet the conditions satisfied by $Q_h$. 
The energy-norm stability of $\Pi$ is a direct consequence of \eqref{Q}, as established in the following lemma.
\begin{lemma}
Let $\Pi=Q_h$, then $\Pi$ satisfies \eqref{Pi}, i.e.,
\begin{equation}
\|\Pi w\|_{A_h} \lesssim\|w\|_{A_0}, \text { for all } w \in V_0.
\end{equation}
\end{lemma}
\begin{proof}
The proof can be proceeded in a similar way as in \cite{csz1996,gwx2016,x1996}. Given $\tau\in\mathcal{T}_h$, let $\widetilde{\tau}_0$ be the union of elements in $\mathcal{T}_0$ that intersect with $\tau$. Thus, according to the inverse inequality and \eqref{err_Q}, we have
\begin{align*}
|w-\Pi w|_{1,\tau}=|w-Q_h w|_{1,\tau}\lesssim& |\tau|^{1/2}|w-Q_h w|_{1,\infty,\tau}\,\left(\int_\tau|u|^2\leq |\tau||u|_{\infty,\tau}^2\right)\\
\lesssim& h^{d/2}|w|_{1,\infty,\tau}\\
\lesssim& h^{d/2}h_0^{-d/2}|w|_{1,\widetilde{\tau}_0}\,\left(|u|_{\infty,\tau}^2\leq \frac{1}{|\widetilde{\tau}_0|}\int_{\widetilde{\tau}_0}|u|^2\right)\\
\lesssim& (h/h_0)^{d/2}|w|_{1,\widetilde{\tau}_0}.
\end{align*}
Summing over all $\tau\in\mathcal{T}_h$ leads to
\begin{align*}
|w-\Pi w|_{1,\Omega}\lesssim |w|_{1,\Omega_0}.
\end{align*}
Thus, we have
\begin{equation*}
\|\Pi w\|_{A_h} \lesssim|\Pi w|_{1,\Omega} \lesssim|\Pi w-w|_{1,\Omega}+|w|_{1,\Omega} \lesssim|w|_{1,\Omega_0} \lesssim\|w\|_{A_0},
\end{equation*}
which gives the desired result.
\end{proof}

Hence, it remains to construct an operator $P$ that satisfies the conditions \eqref{P}-\eqref{err_P}. 
The following lemma can be obtained by taking $P=Q_0$.
\begin{lemma}\label{lem_P}
Let $P=Q_0$, then $P$ satisfies
\begin{equation}
\|v-P v\|_{0,\Omega_0}+h_0|v-P v|_{1,\Omega_0} \lesssim h_0|v|_{1,\Omega}, \text { for any } v \in V_{h}.
\end{equation}
\end{lemma}

\begin{proof}
Inspired by the idea in \cite{bp1993,x1996}, we have the following derivation. 
Let $\widetilde{\Omega}_0$ be the domain consisting of the elements in $\Omega_0$ that do not intersect with $\partial\Omega$, then there is
\begin{align*}
\begin{aligned}
|v-Pv|_{1,\Omega_0}=|v-Q_0v|_{1,\Omega_0}\lesssim& |v-Q_0v|_{1,\widetilde{\Omega}_0}+|v-Q_0v|_{1,\Omega_0/\widetilde{\Omega}_0}\\
\lesssim&|v-Q_0v|_{1,\widetilde{\Omega}_0}+|v|_{1,\Omega_0/\widetilde{\Omega}_0}+|Q_0v|_{1,\Omega_0/\widetilde{\Omega}_0}.
\end{aligned}
\end{align*}
Since $v\in V_h$ vanishes on $\Omega_0/\Omega$, it follows from \eqref{err_Q0} that
\begin{align}
|v-Q_0v|_{1,\widetilde{\Omega}_0}\lesssim |v&-Q_0v|_{1,\Omega}\lesssim|v|_{1,\Omega}, \label{term1}\\
|v|_{1,\Omega_0/\widetilde{\Omega}_0}&\lesssim |v|_{1,\Omega}.\label{term2} 
\end{align}
Given $\tau_0\in \mathcal{T}_0\cap\Omega_0/\widetilde{\Omega}_0$, let $\widetilde{\tau}$ be the union of elements in $\mathcal{T}_h$ that intersect with $\tau_0$. Then, we have
\begin{align*}
|v-Q_0v|_{1,\tau_0}\lesssim &|\tau_0|^{1/2}|v-Q_0v|_{1,\infty,\tau_0}\lesssim h_0^{d/2}|v|_{1,\infty,\tau_0}\lesssim h_0^{d/2}h^{-d/2}|v|_{1,\widetilde{\tau}}\lesssim (h_0/h)^{d/2}|v|_{1,\widetilde{\tau}}.
\end{align*}
This implies that
\begin{align}
|Q_0v|_{1,\Omega_0/\widetilde{\Omega}_0}\lesssim |v-Q_0v|_{1,\Omega_0/\widetilde{\Omega}_0}+|v|_{1,\Omega_0/\widetilde{\Omega}_0}\lesssim |v|_{1,\Omega}.\label{term3}
\end{align}
Synthesizing \eqref{term1}-\eqref{term3} gives
\begin{align}
|v-Pv|_{1,\Omega_0}\lesssim|v|_{1,\Omega}.
\end{align}

Since $\tau_0$ is a square ($d=2$) or a cube ($d=3$), by Poincare inequality and scaling, the following inequality
\begin{align*}
\|v\|^2_{0,\tau_0}\lesssim l^2|v|_{1,\tau_0}^2,
\end{align*}
holds for all $v$ vanishing at one point in $\tau_0$, where $l=O(h_0)$ denotes its side length. 
The strip $\Omega_0/\widetilde{\Omega}_0$ can be covered by these $\tau_0$, one can easily deduce 
\begin{align}
\|v\|_{0,\Omega_0/\widetilde{\Omega}_0}\lesssim h_0|v|_{1,\Omega}.
\end{align}
According to \eqref{err_Q0}, the following inequalities hold
\begin{align}
\|Q_0v\|_{0,\Omega_0/\widetilde{\Omega}_0}\lesssim \|v-Q_0v\|_{0,\Omega_0/\widetilde{\Omega}_0}&+\|v\|_{0,\Omega_0/\widetilde{\Omega}_0}\lesssim h_0|v|_{1,\Omega},\\
\|v-Q_0v\|_{0,\widetilde{\Omega}_0}&\lesssim h_0|v|_{1,\Omega}. 
\end{align}
Thus, we have 
\begin{align}
\|v-Pv\|_{0,\Omega_0}\lesssim&\|v-Q_0v\|_{0,\widetilde{\Omega}_0}+\|v\|_{0,\Omega_0/\widetilde{\Omega}_0}+\|Q_0v\|_{0,\Omega_0/\widetilde{\Omega}_0}\lesssim h_0|v|_{1,\Omega},
\end{align}
which completes the proof.
\end{proof}

\begin{lemma}
The operator $P$ satisfies the properties \eqref{P} and \eqref{err_P}.
\end{lemma}

\begin{proof}
By using Lemma \ref{lem_P}, for all $v \in V_{h}$, we have
\begin{equation*}
\|P v\|_{A_0} \lesssim |P v|_{1,\Omega_0} \lesssim|P v-v|_{1,\Omega_0}+|v|_{1,\Omega_0} \lesssim|v|_{1,\Omega} \lesssim\|v\|_{A_h},
\end{equation*}
which leads to \eqref{P}. Therefore, we further to obtain
\begin{equation*}
\|v-\Pi P v\|_{0,\Omega} \lesssim\|v-Pv\|_{0,\Omega}+\|Pv-\Pi P v\|_{0,\Omega} \lesssim h|v|_{1,\Omega} \lesssim h\| v \|_{A_h}.
\end{equation*}
Together with the estimate $\rho_{A_h}=O\left(h^{-2}\right)$, the above conclusion yields the proof of the inequality \eqref{err_P}. 
\end{proof}

At this stage, the given V-ASMG construction meets all of the conditions in Theorem \ref{thm_mg}.

%%%%%%%%%%%%%%%%%%%%%%%%%%%%%%
\section{Numerical examples}
%%%%%%%%%%%%%%%%%%%%%%%%%%%%%%
In this section, we present a number of numerical experiments to test the performance of the proposed V-ASMG method as a preconditioner of the PCG method for solving the large sparse linear equations \eqref{equ}, which are derived from the linear elasticity equations \eqref{prob} in two and three dimensions.
The experimental configuration is as follows:
\begin{itemize}
\item[(1)] Algorithm implementation:
Adopting the Lagrange linear element to solve the linear elasticity equations \eqref{prob}, and then the stiffness matrix $A$, the load vector $F$ and the vertex coordinates of the initial grid are obtained.
The AMG method refers to the code framework in \cite{boss2023} with Ruge-St$\mathrm{\ddot{u}}$ben coarsening strategy.
All the relevant algorithms are implemented in C$++$ language and embedded in domestic finite element structure simulation analysis software ``BEFEM''.
Furthermore, the result obtained by LU decomposition method is used as the reference solution.
\item[(2)] Hardware configuration:\\
Computer model: i9-13900HX;\\
Internal memory: 64G;\\
Number of cores: 24;\\
Number of threads: 32.
\item[(3)] Convergence criterion: 
The relative residual $rel\_{res}=\frac{\Vert F-AU^{(k)} \Vert}{\Vert F\Vert}$ of a certain step $k$ is less than the given tolerance $\epsilon=10^{-6}$.
\item[(4)] Symbol description(All algorithms without additional description come from the ``BEFEM'' software):\\
``PCG(AMG-AMGCL)''represents the PCG method with classical AMG method as preconditioner in library AMGCL;\\
``PCG(AMG)''represents the PCG method with classical AMG method as preconditioner;\\
``PCG(B-ASMG)''represents the PCG method with barycenter-based auxiliary space multi- grid(B-ASMG) method in \cite{gwx2016} as preconditioner;\\
``PCG(V-ASMG)''represents the PCG method with our V-ASMG method as preconditioner.
\end{itemize}

\subsection{The two-dimensional problems}
This section shows four two-dimensional examples of different working conditions and different sizes, that is, $d=2$.

\noindent {\bf Example 4.1.} Round-hole plate tensile problem with Young's modulus $E=2.1\mathrm{E}+05$ and Poisson's ratio $\nu=0.3$. 
The aim of this example is to numerically simulate the mechanical behavior of rectangular thin plate with round hole under unidirectional tensile load. 
Since the structure and boundary condition of the problem are symmetric, it is sufficient to take the top right quarter model as the representative.
The 10$\times$10 computational region with a round hole of radius $r=1$ is divided by an unstructured triangular grid shown in \figurename \ref{figure1-1}(a). 
A uniform line force of 10(i.e. $q_1=10$) along the $x$-axis is applied on its right boundary, the displacement in the $x$-direction is constrained to be $u_1=0$ on its left boundary and the displacement in the $y$-direction is constrained to be $u_2=0$ on its bottom boundary. 
The above problem is solved based on the Lagrange linear element, and the corresponding sparse stiffness matrix $A$ is obtained, the relevant information of which is shown in Table \ref{table1-1} below. 
Using the AMG-AMGCL, AMG, B-ASMG and V-ASMG methods as the preconditioners of the PCG method to iteratively solve the equations corresponding to the above problem, the iteration steps, numerical error between the numerical solution and the reference solution under the $L_2$ norm, setup time, application time and total time are shown in Table \ref{table1-2} below.
The convergence history of the relative residual varying with the
iteration steps is shown in \figurename \ref{figure1-2}. 
The displacement nephogram is shown in \figurename \ref{figure1-1}(b).
The above results show that compared with the other three methods, V-ASMG method requires the smallest iteration steps within the shortest time to reach the given tolerance, and obtains the smallest error when the iteration terminates. 
The reason for the long setup time is that a large number of time-consuming search operations are involved in the implementation of the algorithm.
This is because when numbering the coarse grid vertices, the four vertices corresponding to a particular rectangular region may already be added to the list of coarse grid vertices, so it must be determined whether it is already in the list, and if so, what is the number. 
For this example, the numerical performance of above four kinds of preconditioners under different numbers of unknowns and Poisson's ratios has also been tested, as shown in \figurename \ref{figure1-3} and \figurename \ref{figure1-4}.
With the same number of unknowns or the same Poisson ratio, the corresponding iteration steps and total solution time of the V-ASMG scheme are the least among the four schemes. 
Moreover, for the same working condition, the number of iteration steps of the V-ASMG scheme almost does not change with the number of unknowns and Poisson's ratio, and its slope of time variation is also the smallest.
\begin{figure}[!ht]
	\centerline{
\subfigure[Geometric model.] {\includegraphics[width=0.500\textwidth]{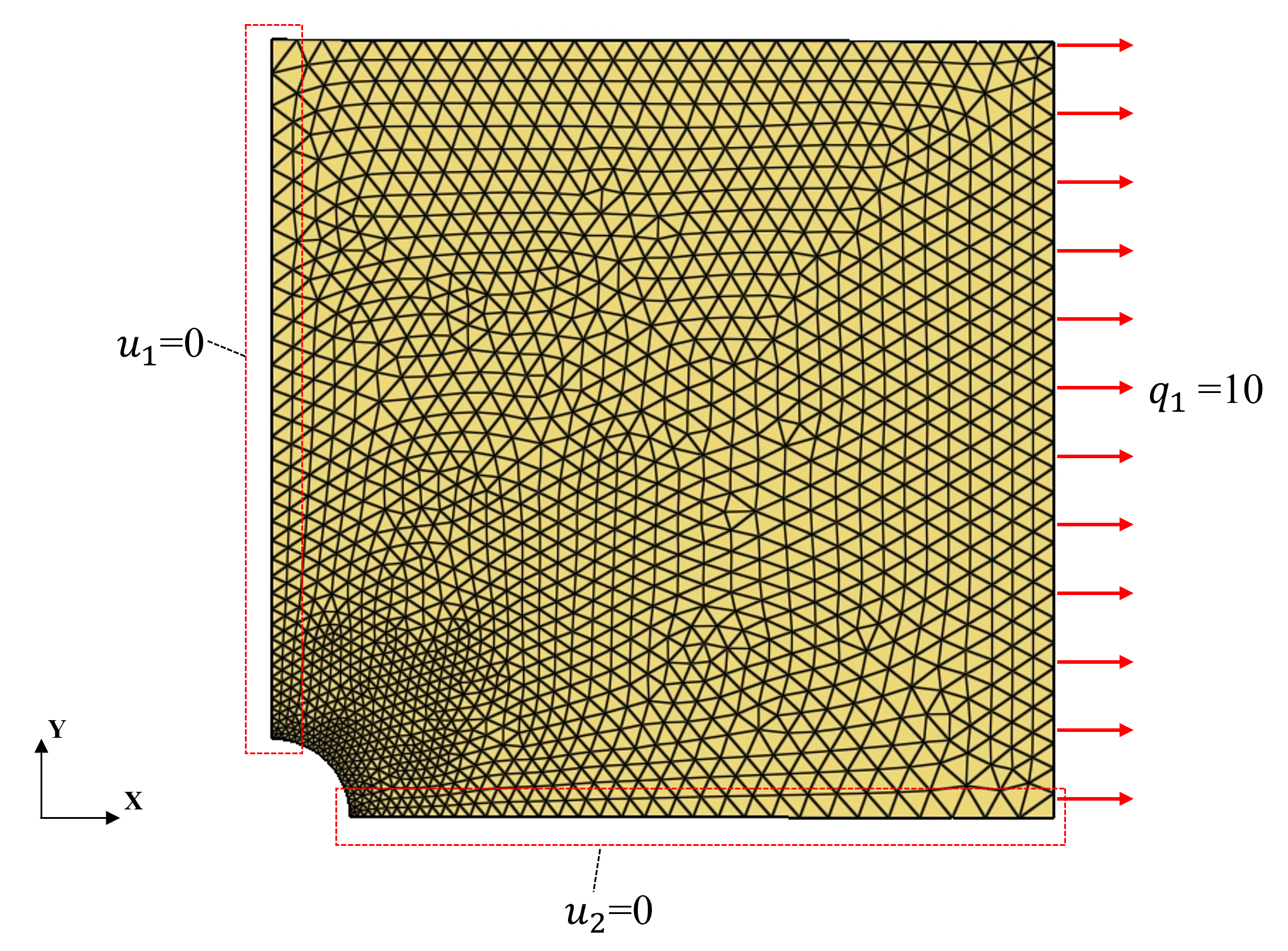}}
\subfigure[Calculation result.] {\includegraphics[width=0.425\textwidth]{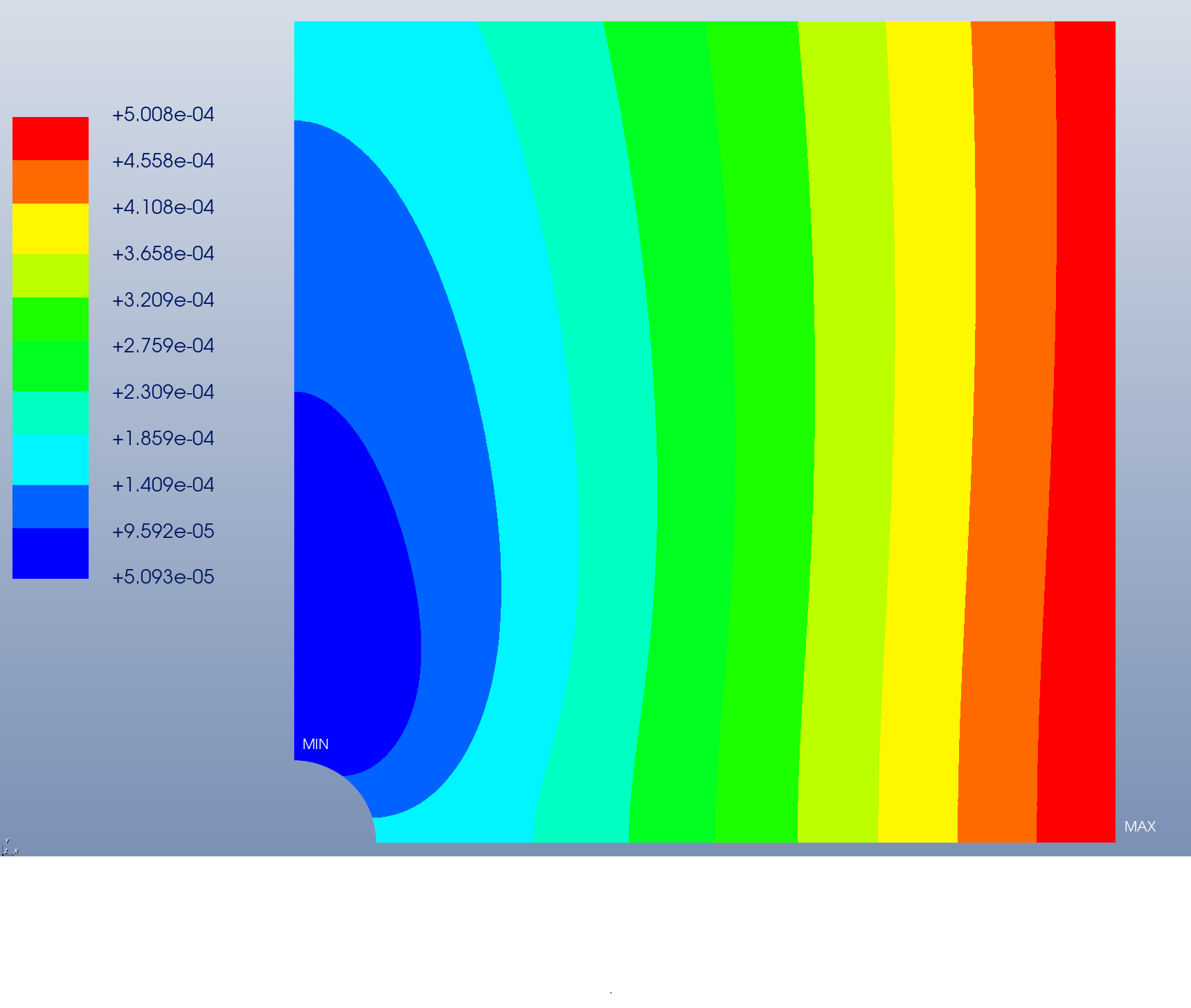}}}
	\caption{Round-hole plate tensile problem.}
	\label{figure1-1}
\end{figure}
\begin{table}[!ht]
\begin{center}
{\caption{
The relevant information of the sparse stiffness matrix $A$ corresponding to the round-hole plate tensile problem.}
\begin{tabular}{|c|c|}
\hline
Number of vertices           &99054
\\ \hline
Number of unknowns           &198108
\\ \hline
Number of nonzero elements   &2763008
\\ \hline
Symmetric or not             &Yes
\\ \hline
Positive definite or not     &Yes
\\ \hline
Condition number             &2.42278E+06
\\ \hline
Spectral radius              &1.27550E+06
\\ \hline
\end{tabular}
\label{table1-1}}
\end{center}
\end{table}
\begin{table}[!ht]
\begin{center}
{\caption{
The results of the round-hole plate tensile problem using the PCG method with different preconditioners.}
\begin{tabular}{|c|c|c|c|c|}
\hline
\multirow{2}{*}{Method} &AMGCL &\multicolumn{3}{c|}{BEFEM} 
\\ \cline{2-5} 
 &PCG(AMG) &PCG(AMG) &PCG(B-ASMG) &PCG(V-ASMG) 
\\ \hline
Iteration steps  &933 &596 &96 &9                           
\\ \hline
Numerical error  &1.16950E-07 &9.55033E-08 
                 &7.30630E-08 &1.00424E-08
\\ \hline
Setup time       &  1.636s & 2.264s & 2.970s &6.246s
\\ \hline
Application time &214.728s &48.934s & 8.936s &1.067s
\\ \hline
Total time       &216.364s &51.198s &11.906s &7.313s
\\ \hline 
\end{tabular}
\label{table1-2}}
\end{center}
\end{table}
\begin{figure}[!ht]
	\centerline{\includegraphics[width=0.5\textwidth]{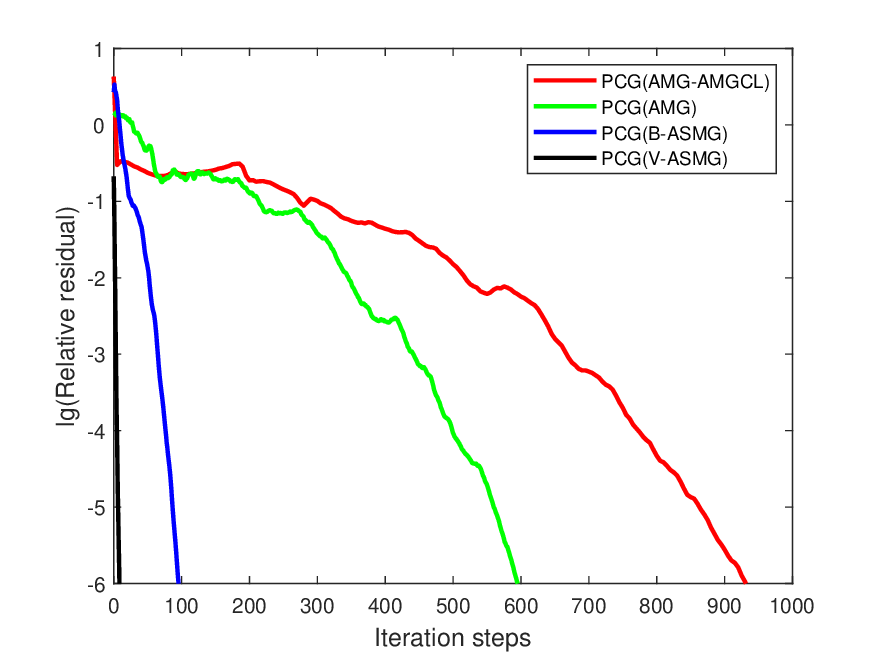}}
	\caption{The convergence history of the relative residual of the round-hole plate tensile problem.}
	\label{figure1-2}
\end{figure}
\begin{figure}[!ht]
	\centerline{
\subfigure[] 
{\includegraphics[width=0.5\textwidth]{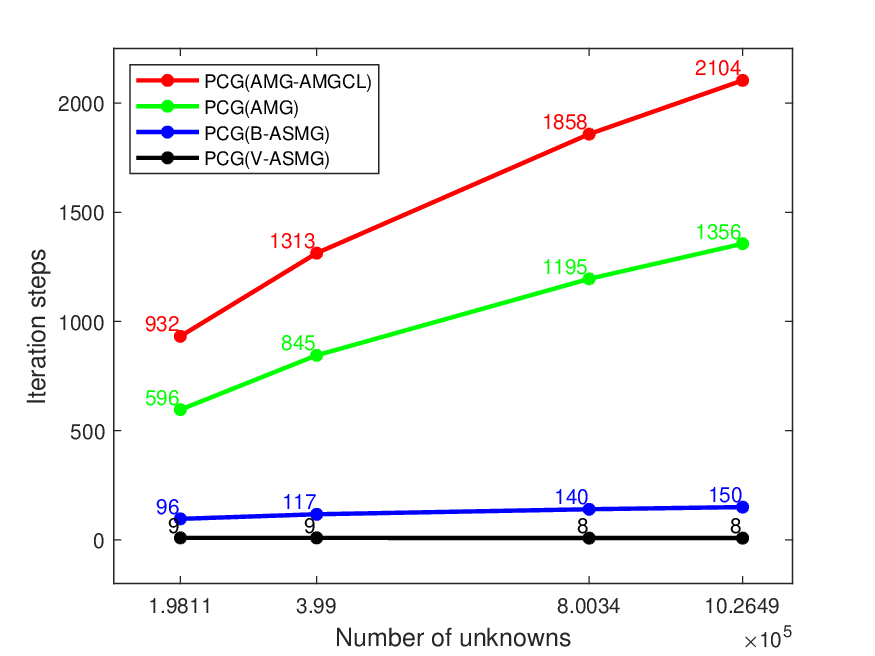}}
\subfigure[] 
{\includegraphics[width=0.5\textwidth]{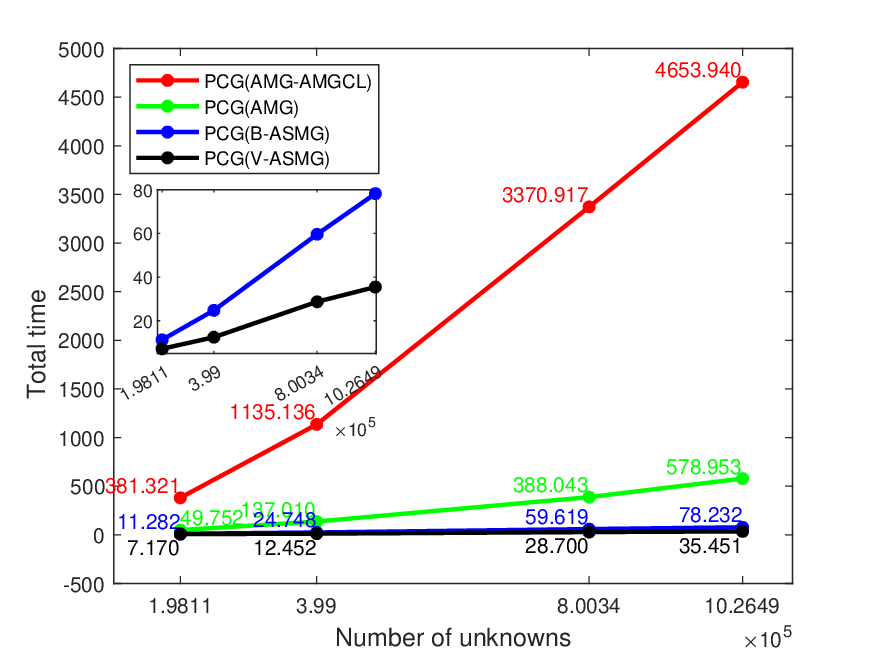}}}
	\caption{The line chart of iteration steps and total time along with the number of unknowns for the round-hole plate tensile problem.}
	\label{figure1-3}
\end{figure}
\begin{figure}[!ht]
	\centerline{
\subfigure[] 
{\includegraphics[width=0.5\textwidth]{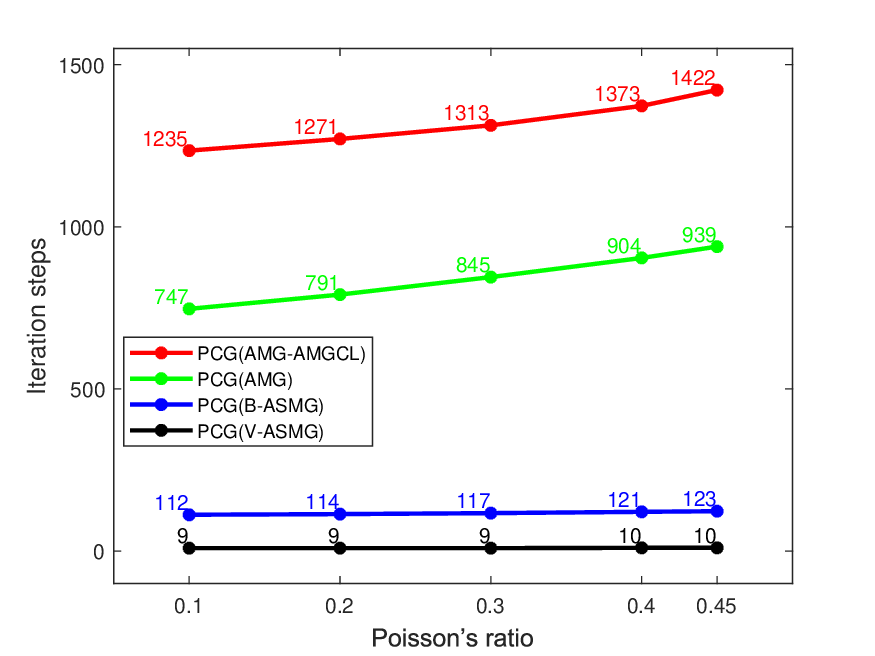}}
\subfigure[] 
{\includegraphics[width=0.5\textwidth]{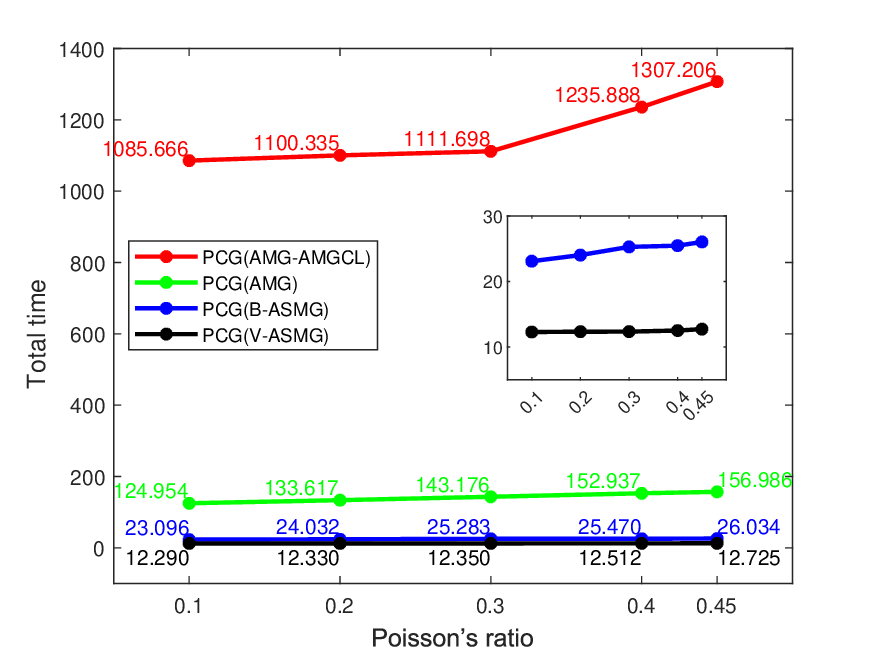}}}
	\caption{The line chart of iteration steps and total time along with the Poisson's ratio $\nu=0.1,0.2,0.3,0.4,0.45$ for the round-hole plate tensile problem.}
	\label{figure1-4}
\end{figure}

\noindent {\bf Example 4.2.} Toroidal stress problem with Young's modulus $E=2.1\mathrm{E}+05$ and Poisson's ratio $\nu=0.3$. 
In this example, the mechanical behavior of a quarter ring under circumferential load is numerically simulated. 
The radius of the outer ring is $R=10$ and the radius of the inner ring is $r=1$.
Divide its computational region by an unstructured triangular grid shown in \figurename \ref{figure2-1}(a). 
Apply a uniform line force of 10(i.e. $q_2=-10$) in the negative $y$-direction at the bottom boundary and constrain the displacements in the $x$-direction and $y$-direction to be $u_1=u_2=0$ at the left boundary. 
Adopt the Lagrange linear element to solve the above problem, and obtain the corresponding sparse stiffness matrix $A$, whose relevant information is presentated in Table \ref{table2-1} below. 
Use the AMG-AMGCL, AMG, B-ASMG and V-ASMG methods as the preconditioners of the PCG method to iteratively solve the equations corresponding to the above problem. 
Table \ref{table2-2} shows the iteration steps, numerical error between the numerical solution and the reference solution under the $L_2$ norm, setup time, application time and total time.
\figurename \ref{figure2-2} shows the convergence history of the relative residual varying with the iteration steps. 
\figurename \ref{figure2-1}(b) shows the displacement nephogram.
From the above results, it is not hard to see that the V-ASMG algorithm has significant advantages over the other three algorithms in terms of the number of iterative convergence steps, total solution time and error accuracy. 
\begin{figure}[!ht]
	\centerline{
\subfigure[Geometric model.] {\includegraphics[width=0.40\textwidth]{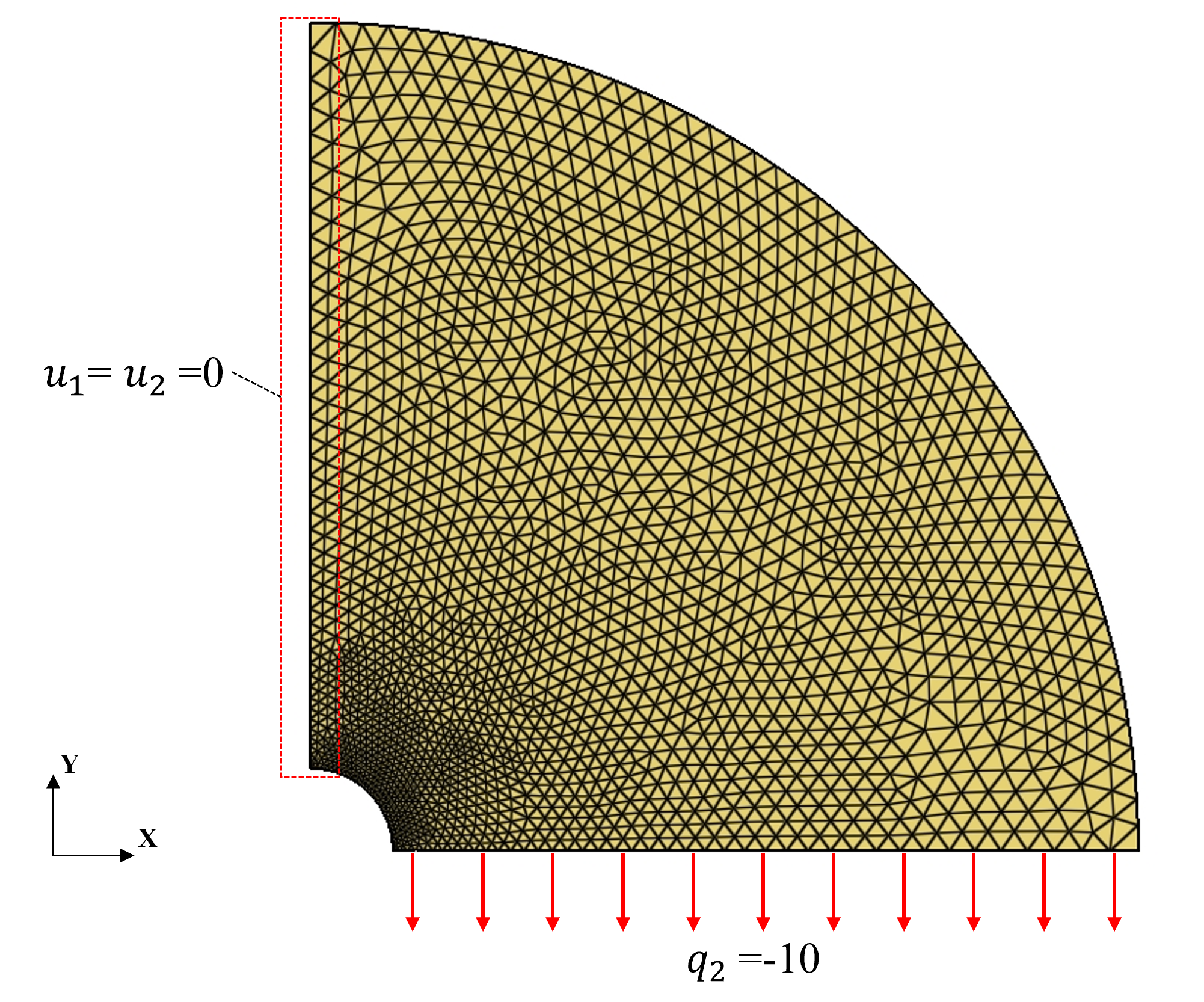}}
\subfigure[Calculation result.] {\includegraphics[width=0.40\textwidth]{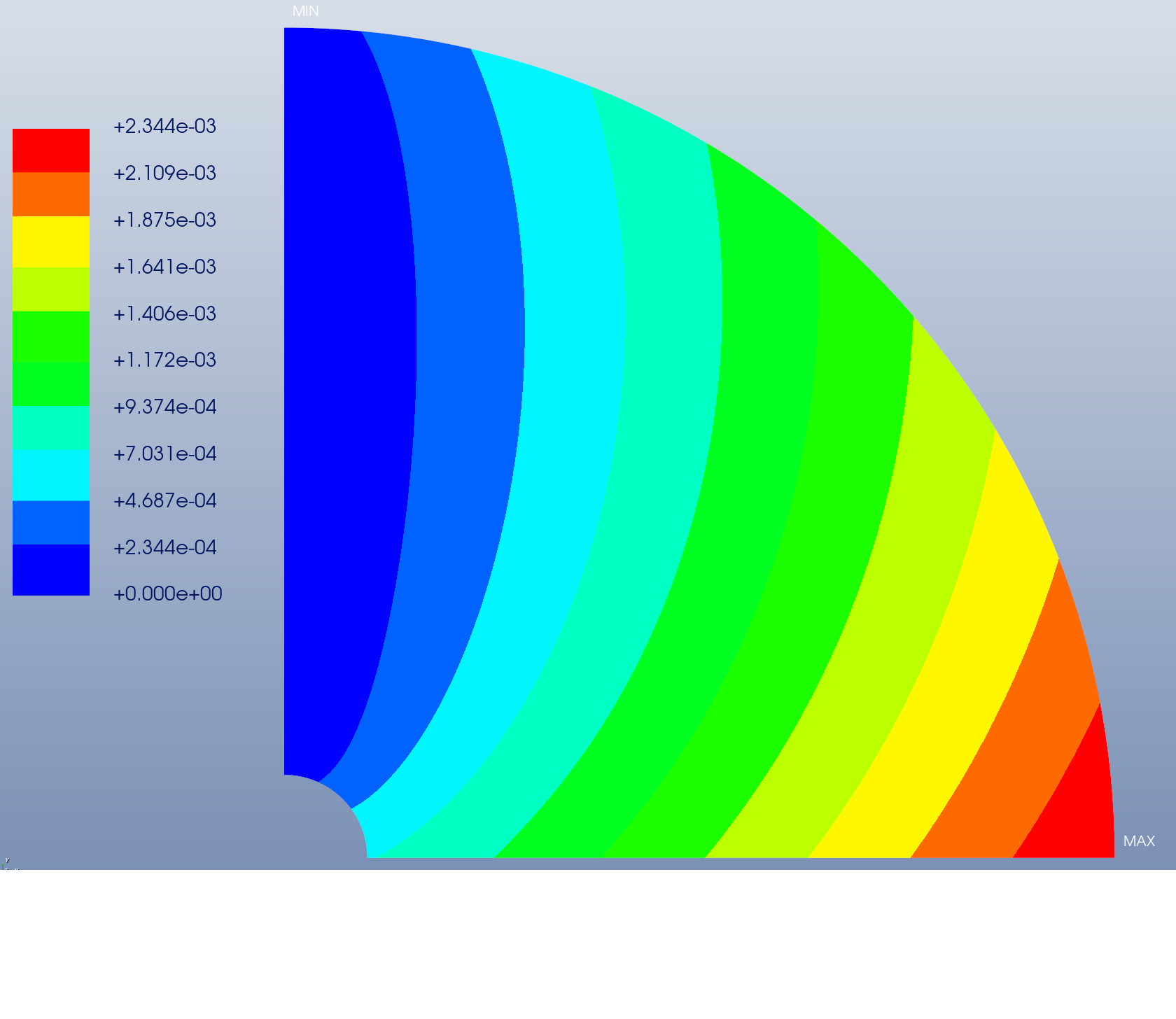}}}
	\caption{Toroidal stress problem.}
	\label{figure2-1}
\end{figure}
\begin{table}[!ht]
\begin{center}
{\caption{
The relevant information of the sparse stiffness matrix $A$ corresponding to the toroidal stress problem.}
\begin{tabular}{|c|c|}
\hline
Number of vertices           &150129
\\ \hline
Number of unknowns           &300258
\\ \hline
Number of nonzero elements   &4190788
\\ \hline
Symmetric or not             &Yes
\\ \hline
Positive definite or not     &Yes
\\ \hline
Condition number             &3.49371E+06
\\ \hline
Spectral radius              &1.36340E+06
\\ \hline
\end{tabular}
\label{table2-1}}
\end{center}
\end{table}
\begin{table}[!ht]
\begin{center}
{\caption{
The results of the toroidal stress problem using the PCG method with different preconditioners.}
\begin{tabular}{|c|c|c|c|c|}
\hline
\multirow{2}{*}{Method} &AMGCL &\multicolumn{3}{c|}{BEFEM} 
\\ \cline{2-5} 
 &PCG(AMG) &PCG(AMG) &PCG(B-ASMG) &PCG(V-ASMG) 
\\ \hline
Iteration steps  &1127 &688 &111 &11                           
\\ \hline
Numerical error  &2.28540E-08 &3.21710E-08 
                 &1.92695E-08 &2.14983E-09
\\ \hline
Setup time       &  3.573s & 3.563s & 4.156s &8.239s
\\ \hline
Application time &655.842s &85.889s &13.703s &1.679s
\\ \hline
Total time       &659.415s &89.452s &17.859s &9.918s
\\ \hline 
\end{tabular}
\label{table2-2}}
\end{center}
\end{table}
\begin{figure}[!ht]
	\centerline{\includegraphics[width=0.5\textwidth]{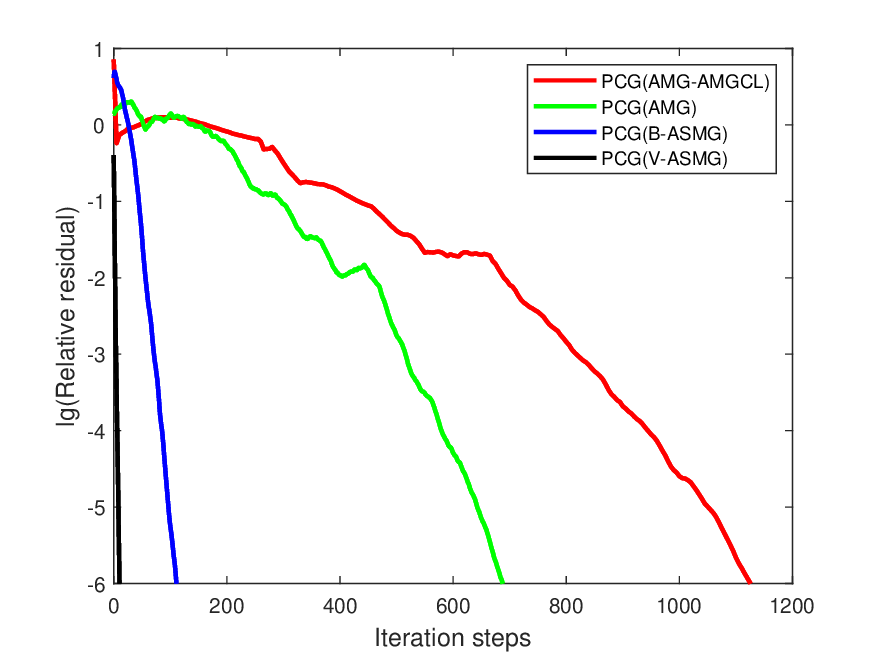}}
	\caption{The convergence history of the relative residual of the toroidal stress problem.}
	\label{figure2-2}
\end{figure}

\noindent {\bf Example 4.3.} Square-hole plate tensile problem with Young's modulus $E=2.1\mathrm{E}+05$ and Poisson's ratio $\nu=0.3$. 
The purpose of this example is to numerically simulate the mechanical behavior of rectangular thin plate with square hole under unidirectional tensile load.
Making use of the symmetry of structure and boundary condition, the top right quarter model is adopted to simplify the calculation. 
Again, an unstructured triangular grid shown in \figurename \ref{figure3-1}(a) is used to divide the 10$\times$10 computational region with a square hole of half length $l=1$.
Apply the same boundary conditions as in Example 4.1.
The Lagrange linear element is used to solve the above problem, and the corresponding sparse stiffness matrix $A$ with the properties shown in Table \ref{table3-1} is obtained. 
The equations corresponding to the above problem are solved by using the AMG-AMGCL, AMG, B-ASMG and V-ASMG methods as the preconditioners of the PCG method, the iteration steps, numerical error between the numerical solution and the reference solution under the $L_2$ norm, setup time, application time and total time are shown in Table \ref{table3-2} below.
The convergence history of the relative residual varying with the
iteration steps is shown in \figurename \ref{figure3-2}. 
The displacement nephogram is shown in \figurename \ref{figure3-1}(b).
The above results indicate that compared to the other three methods, V-ASMG method needs the smallest iteration steps within the shortest time to reach the given tolerance, and achieves the smallest error when the iteration stops. 
\begin{figure}[!ht]
	\centerline{
\subfigure[Geometric model.] {\includegraphics[width=0.500\textwidth]{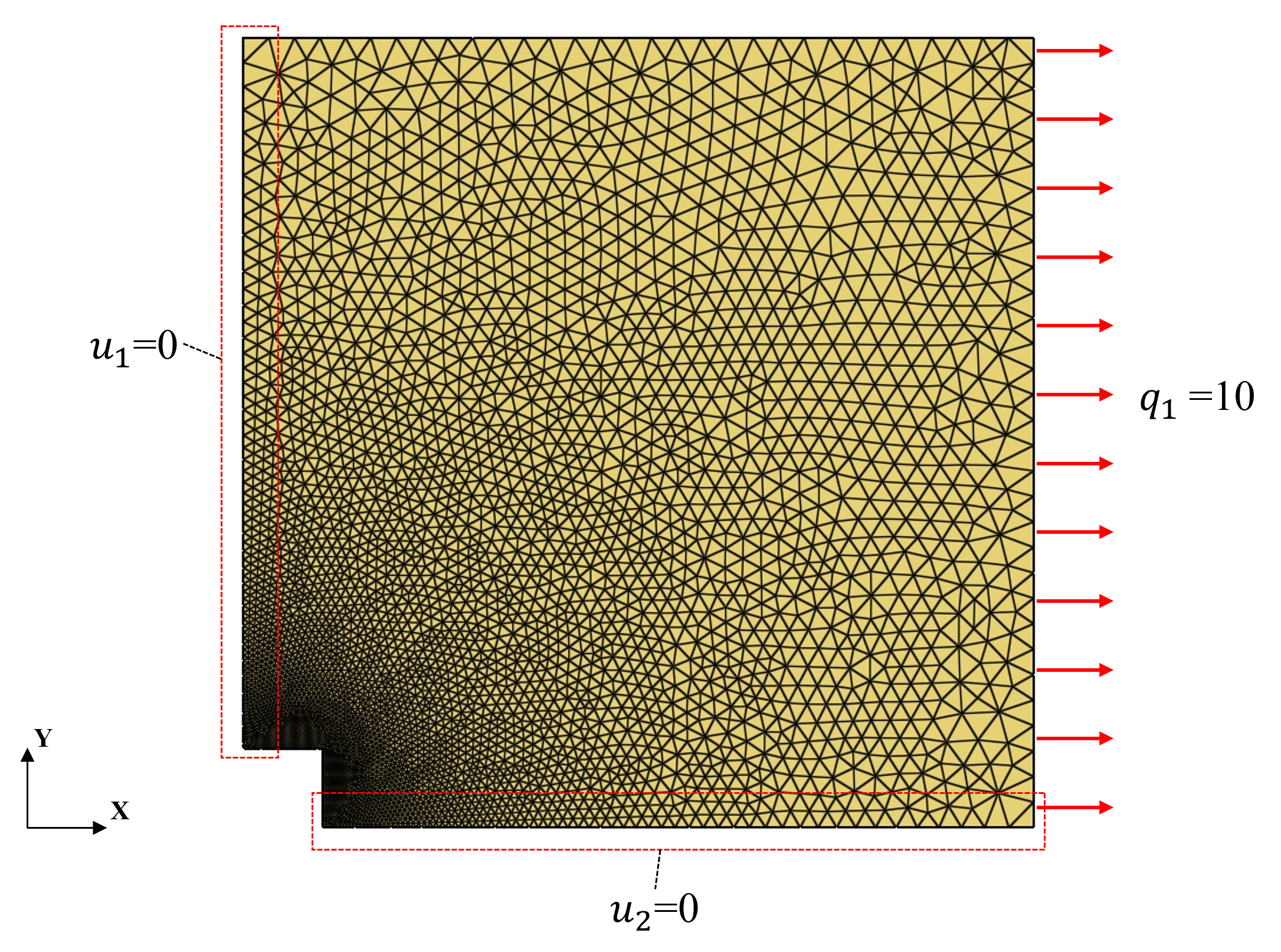}}
\subfigure[Calculation result.] {\includegraphics[width=0.425\textwidth]{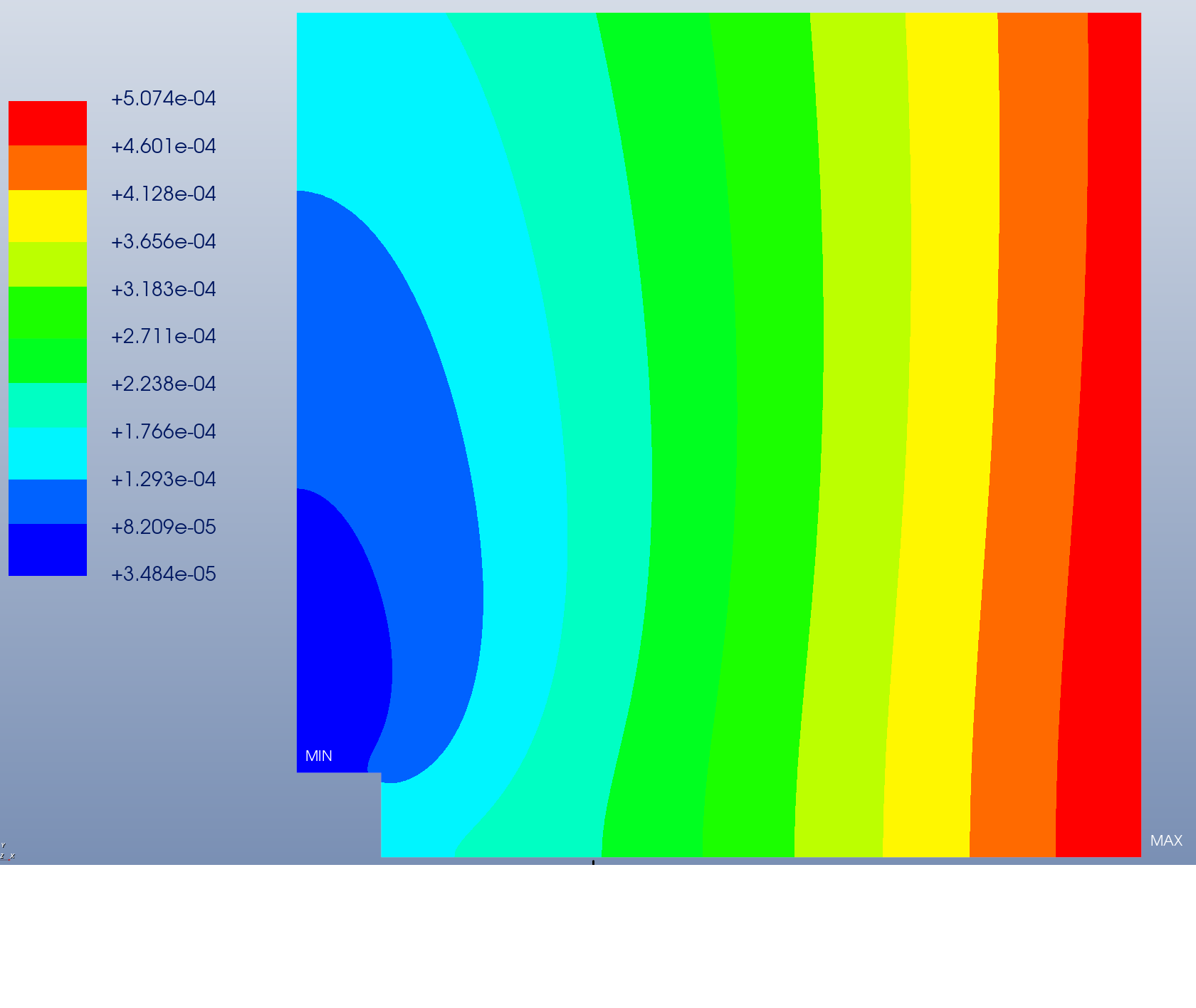}}}
	\caption{Square-hole plate tensile problem.}
	\label{figure3-1}
\end{figure}
\begin{table}[!ht]
\begin{center}
{\caption{
The relevant information of the sparse stiffness matrix $A$ corresponding to the square-hole plate tensile problem.}
\begin{tabular}{|c|c|}
\hline
Number of vertices           &200185
\\ \hline
Number of unknowns           &400370
\\ \hline
Number of nonzero elements   &5589868
\\ \hline
Symmetric or not             &Yes
\\ \hline
Positive definite or not     &Yes
\\ \hline
Condition number             &4.93885E+06
\\ \hline
Spectral radius              &6.74688E+05
\\ \hline
\end{tabular}
\label{table3-1}}
\end{center}
\end{table}
\begin{table}[!ht]
\begin{center}
{\caption{
The results of the square-hole plate tensile problem using the PCG method with different preconditioners.}
\begin{tabular}{|c|c|c|c|c|}
\hline
\multirow{2}{*}{Method} &AMGCL &\multicolumn{3}{c|}{BEFEM} 
\\ \cline{2-5} 
 &PCG(AMG) &PCG(AMG) &PCG(B-ASMG) &PCG(V-ASMG) 
\\ \hline
Iteration steps  &1317 &823 &123 &10                           
\\ \hline
Numerical error  &9.13710E-08 &7.96228E-08 
                 &5.82555E-08 &4.23854E-09
\\ \hline
Setup time       &   4.554s &  4.610s & 5.559s &12.787s
\\ \hline
Application time &1195.918s &135.111s &21.868s & 2.266s
\\ \hline
Total time       &1200.472s &139.721s &27.427s &15.053s
\\ \hline 
\end{tabular}
\label{table3-2}}
\end{center}
\end{table}
\begin{figure}[!ht]
	\centerline{\includegraphics[width=0.5\textwidth]{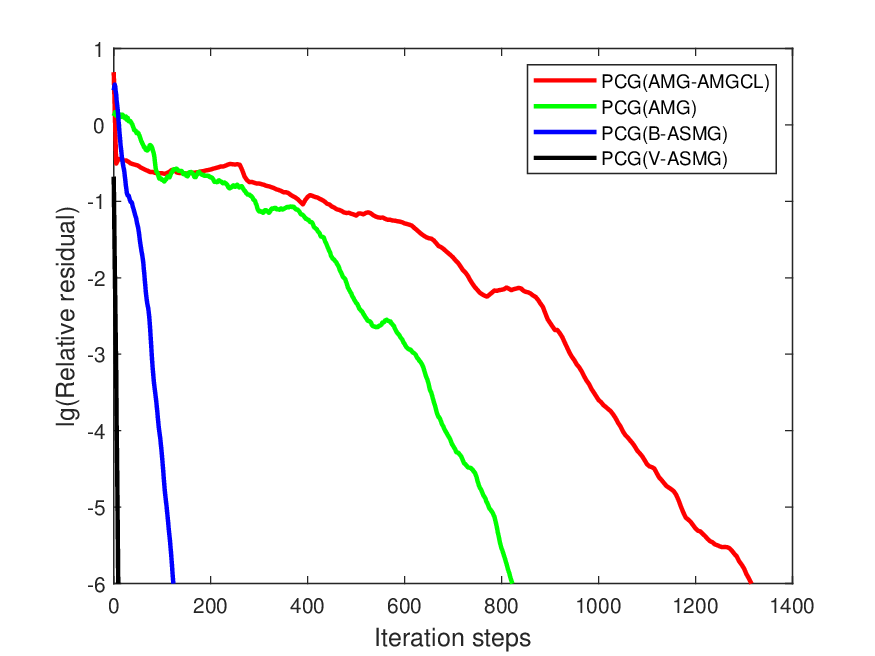}}
	\caption{The convergence history of the relative residual of the square-hole plate tensile problem.}
	\label{figure3-2}
\end{figure}

\noindent {\bf Example 4.4.} Thin-walled water retaining structure stress problem with Young's modulus $E=2.1\mathrm{E}+05$ and Poisson's ratio $\nu=0.3$. 
This example simulates the mechanical response of a thin-walled water retaining structure(such as a steel gate) under hydrostatic pressure.
Since the thickness of such structure is much smaller than other characteristic sizes, the plane stress hypothesis is used to simplify the analysis, which is suitable for evaluating the in-plane stress distribution and deformation characteristic of thin-walled  components.
The height of the calculated area is $H=20$, the width at the bottom is $W=10$, the width at the top is $w=3$, and the angle at the bottom is $\theta=60^{\circ}$, which is also partitioned by an unstructured triangular grid shown in \figurename \ref{figure4-1}(a). 
The bottom boundary is completely fixed, that is, the displacements in the $x$- and $y$- directions are constrained to be $u_1=u_2=0$.
In the meantime, the uniform line force of 10(i.e. $q_1=10$) along the $x$-direction is applied to its left boundary to simulate the effect of lateral load. 
Solving the above problem based on Lagrange linear element, the corresponding sparse stiffness matrix $A$ is obtained, and the relevant information is shown in Table \ref{table4-1}.
The AMG-AMGCL, AMG, B-ASMG and V-ASMG methods are used as the preconditioners of the PCG method to iteratively solve the equations corresponding to the above problem, the iteration steps, numerical error in $L_2$ norm between the numerical solution and the reference solution, setup time, application time and total time are shown in Table \ref{table4-2}. 
The convergence history of the relative residual varying with the
iteration steps is shown in \figurename \ref{figure4-2}. 
The displacement nephogram is shown in \figurename \ref{figure4-1}(b).
According to the above results, it is not difficult to find that compared with the other three algorithms, V-ASMG algorithm not only has the fastest relative residual reduction speed, but also has the highest error accuracy.
\begin{figure}[!ht]
	\centerline{
\subfigure[Geometric model.] {\includegraphics[width=0.40\textwidth]{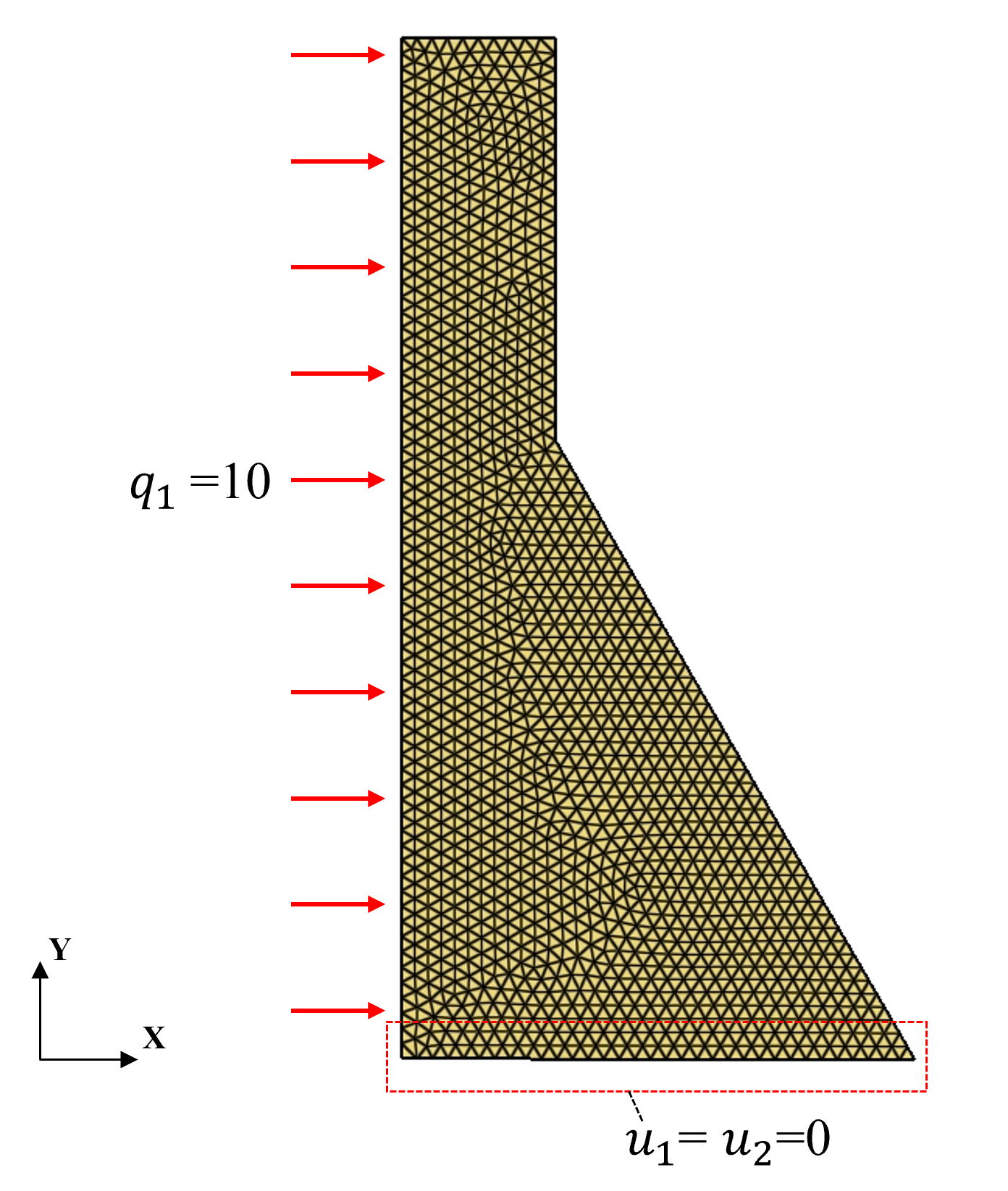}}
\subfigure[Calculation result.] {\includegraphics[width=0.40\textwidth]{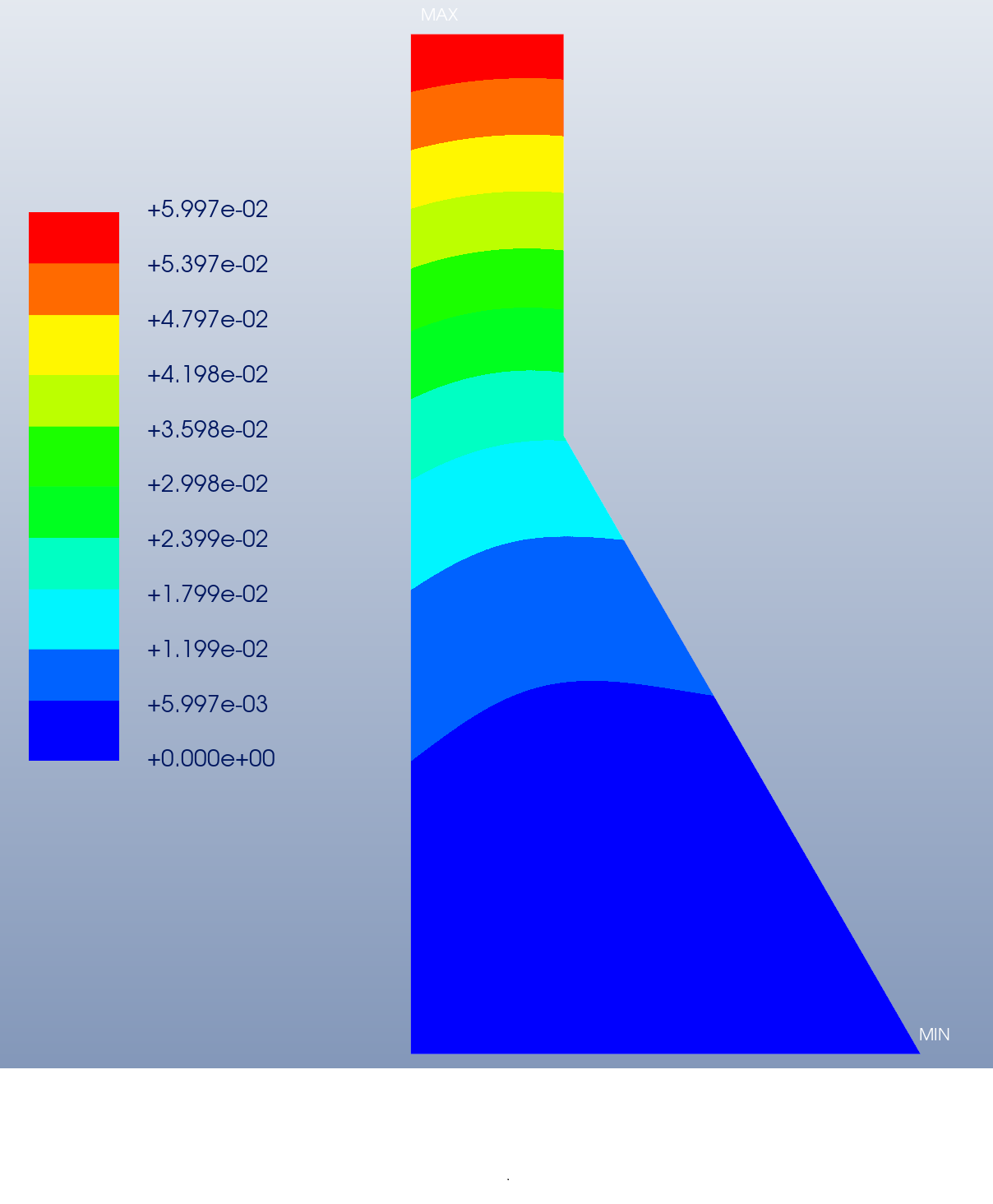}}}
	\caption{Thin-walled water retaining structure stress problem.}
	\label{figure4-1}
\end{figure}
\begin{table}[!ht]
\begin{center}
{\caption{
The relevant information of the sparse stiffness matrix $A$ corresponding to the thin-walled water retaining structure stress problem.}
\begin{tabular}{|c|c|}
\hline
Number of vertices           &249278
\\ \hline
Number of unknowns           &498556
\\ \hline
Number of nonzero elements   &6953400
\\ \hline
Symmetric or not             &Yes
\\ \hline
Positive definite or not     &Yes
\\ \hline
Condition number             &6.07387E+07
\\ \hline
Spectral radius              &1.41200E+06
\\ \hline
\end{tabular}
\label{table4-1}}
\end{center}
\end{table}
\begin{table}[!ht]
\begin{center}
{\caption{
The results of the thin-walled water retaining structure stress problem using the PCG method with different preconditioners.}
\begin{tabular}{|c|c|c|c|c|}
\hline
\multirow{2}{*}{Method} &AMGCL &\multicolumn{3}{c|}{BEFEM} 
\\ \cline{2-5} 
 &PCG(AMG) &PCG(AMG) &PCG(B-ASMG) &PCG(V-ASMG) 
\\ \hline
Iteration steps  &2101 &1358 &165 &12                           
\\ \hline
Numerical error  &3.04490E-07 &2.30523E-07 
                 &2.90871E-07 &2.07623E-08
\\ \hline
Setup time       &   3.295s &  5.871s & 5.356s &13.368s
\\ \hline
Application time &2336.259s &288.395s &27.270s & 2.823s
\\ \hline
Total time       &2339.554s &294.266s &32.626s &16.191s
\\ \hline 
\end{tabular}
\label{table4-2}}
\end{center}
\end{table}
\begin{figure}[!ht]
	\centerline{\includegraphics[width=0.5\textwidth]{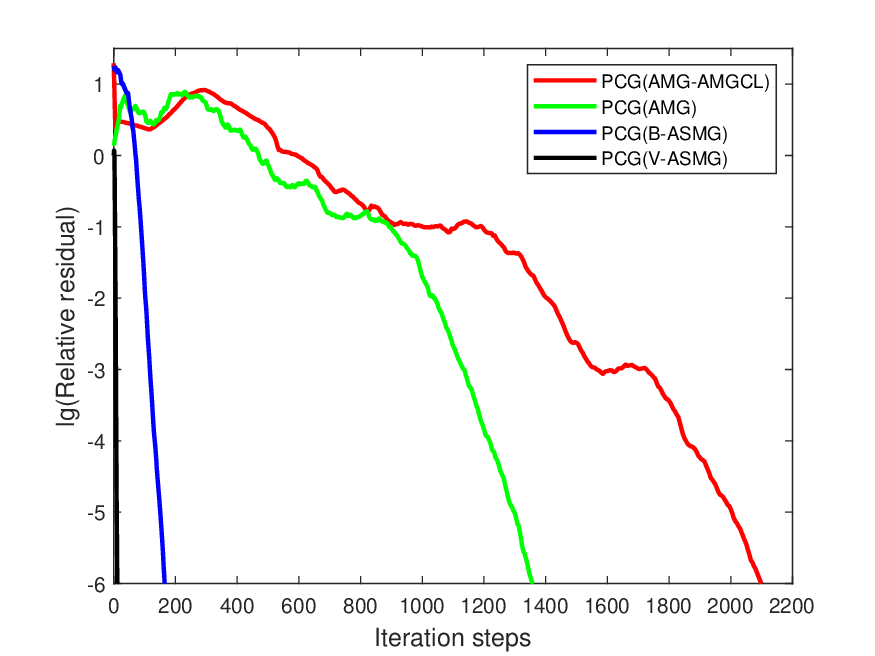}}
	\caption{The convergence history of the relative residual of the thin-walled water retaining structure stress problem.}
	\label{figure4-2}
\end{figure}

\subsection{The three-dimensional problems}
This section exhibits four three-dimensional examples of different working conditions and different sizes, that is, $d=3$.

\noindent {\bf Example 4.5.} Commercial vehicle plate spring bracket stress problem with Young's modulus $E=1.620\mathrm{E}+05$ and Poisson's ratio $\nu=0.293$. 
This example studies the mechanical response of the plate spring bracket of commercial vehicle under vertical load, and focuses on the analysis of the deformation characteristic of the bracket under the condition that the bolt mounting holes are fixed constraints and the pin mounting hole is applied load, which provides a theoretical basis for structural strength check and lightweight design.
The calculation area contains some key features: the body structure of the bracket, the bolt mounting holes connected to the frame, and the pin mounting hole connected to the plate spring.
The bolt mounting holes connected to the frame are completely fixed, and the displacements in $x$-, $y$- and $z$-directions are restricted to be $u_1=u_2=u_3=0$.
While a vertical downward nodal force of 50(i.e. $f_3=-50$) is applied to the pin mounting hole connected to the plate spring.
The computational domain is divided by an unstructured tetrahedral grid, and the problem is solved based on the Lagrange linear element. 
The corresponding geometric model and the obtained equation system information are shown in \figurename \ref{figure5-1}(a) and Table \ref{table5-1} below.
Four MG algorithms are used as the preconditioners of the PCG method to iteratively calculate the equation system obtained in this problem, and the results shown in Table \ref{table5-2} and the convergence history of the relative residual varying with the
iteration steps shown in \figurename \ref{figure5-2} are obtained.
The displacement nephogram is shown in \figurename \ref{figure5-1}(b).
As can be seen from the above figures and tables, for this three-dimensional problem, V-ASMG algorithm also has fewer iteration steps, shorter solving time and higher error accuracy than the other three algorithms.
\begin{figure}[!ht]
	\centerline{
\subfigure[Geometric model.] {\includegraphics[width=0.5\textwidth]{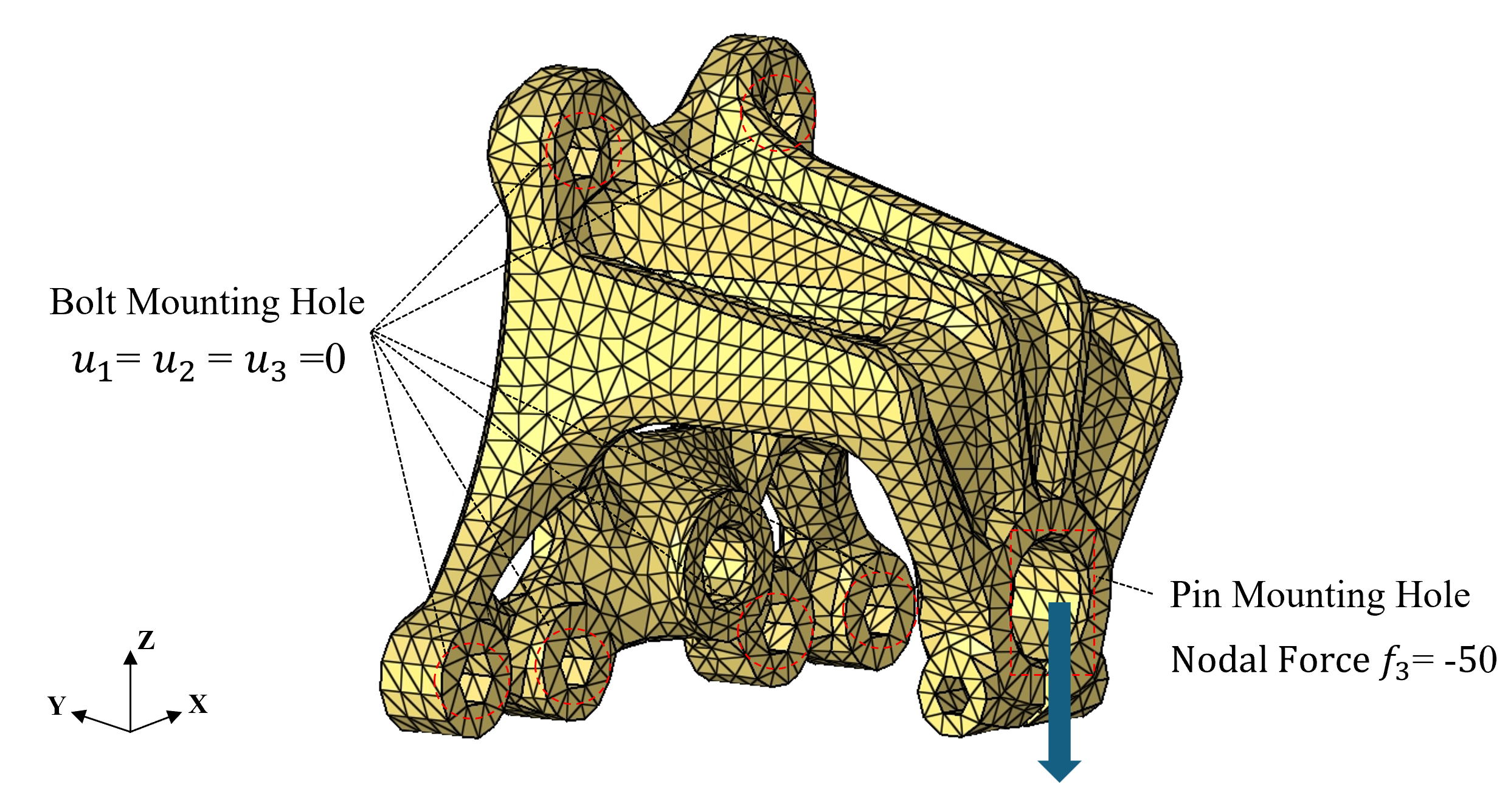}}
\subfigure[Calculation result.] {\includegraphics[width=0.4\textwidth]{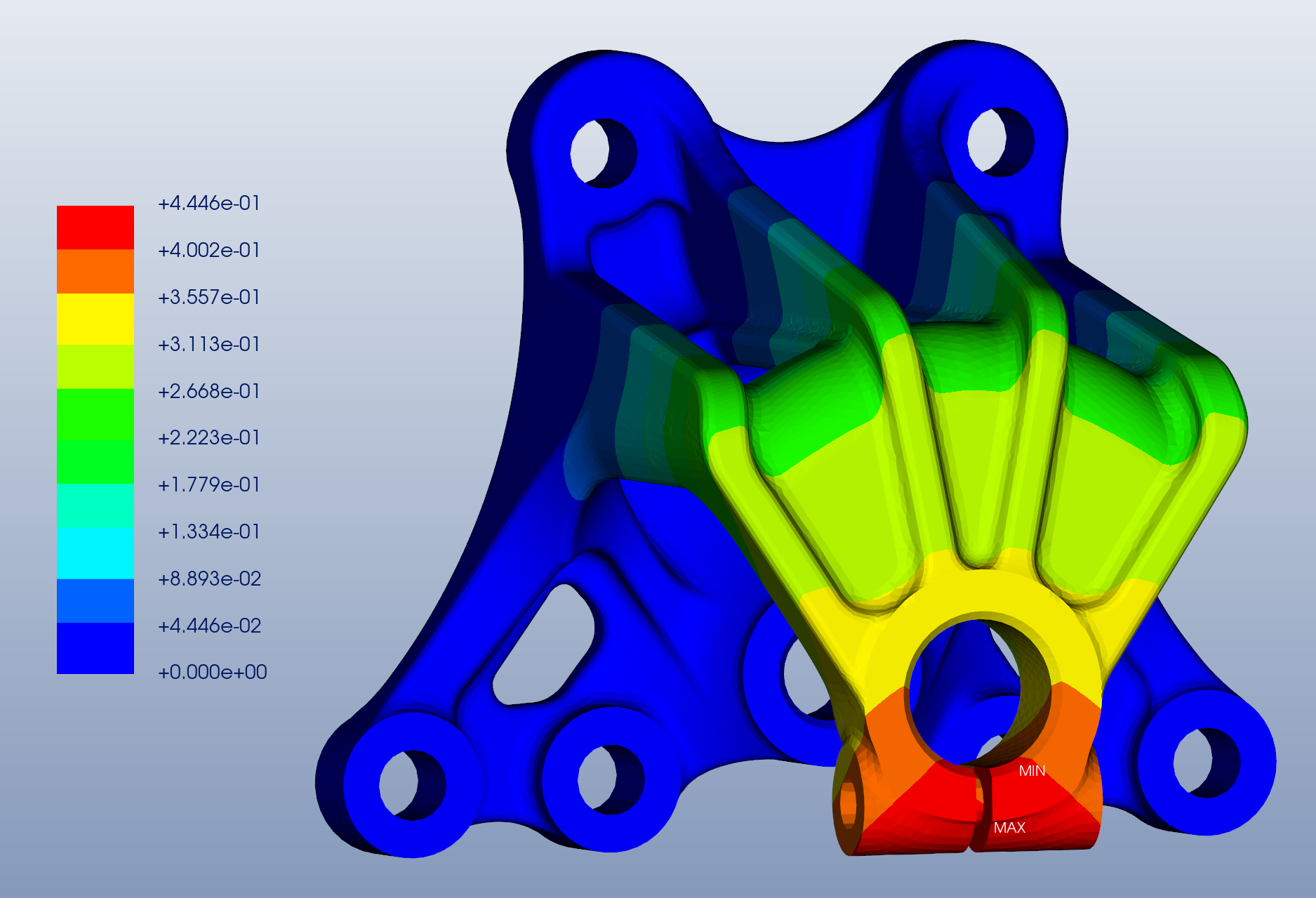}}}
	\caption{Commercial vehicle plate spring bracket stress problem.}
	\label{figure5-1}
\end{figure}
\begin{table}[!ht]
\begin{center}
{\caption{
The relevant information of the sparse stiffness matrix $A$ corresponding to the commercial vehicle plate spring bracket stress problem.}
\begin{tabular}{|c|c|}
\hline
Number of vertices           &85724
\\ \hline
Number of unknowns           &257172
\\ \hline
Number of nonzero elements   &9867594
\\ \hline
Symmetric or not             &Yes
\\ \hline
Positive definite or not     &Yes
\\ \hline
Condition number             &1.28926E+07
\\ \hline
Spectral radius              &5.37445E+06
\\ \hline
\end{tabular}
\label{table5-1}}
\end{center}
\end{table}
\begin{table}[!ht]
\begin{center}
{\caption{
The results of the commercial vehicle plate spring bracket stress problem using the PCG method with different preconditioners.}
\begin{tabular}{|c|c|c|c|c|}
\hline
\multirow{2}{*}{Method} &AMGCL &\multicolumn{3}{c|}{BEFEM} 
\\ \cline{2-5} 
 &PCG(AMG) &PCG(AMG) &PCG(B-ASMG) &PCG(V-ASMG) 
\\ \hline
Iteration steps  &1529 &721 &195 &58                          
\\ \hline
Numerical error  &4.21370E-07 &6.60971E-07 
                 &5.00573E-07 &3.86662E-07
\\ \hline
Setup time       &   6.364s &  7.964s & 5.294s &29.524s
\\ \hline
Application time &1880.198s &224.955s &49.967s &22.544s
\\ \hline
Total time       &1886.562s &232.919s &55.261s &52.068s
\\ \hline 
\end{tabular}
\label{table5-2}}
\end{center}
\end{table}
\begin{figure}[!ht]
	\centerline{\includegraphics[width=0.5\textwidth]{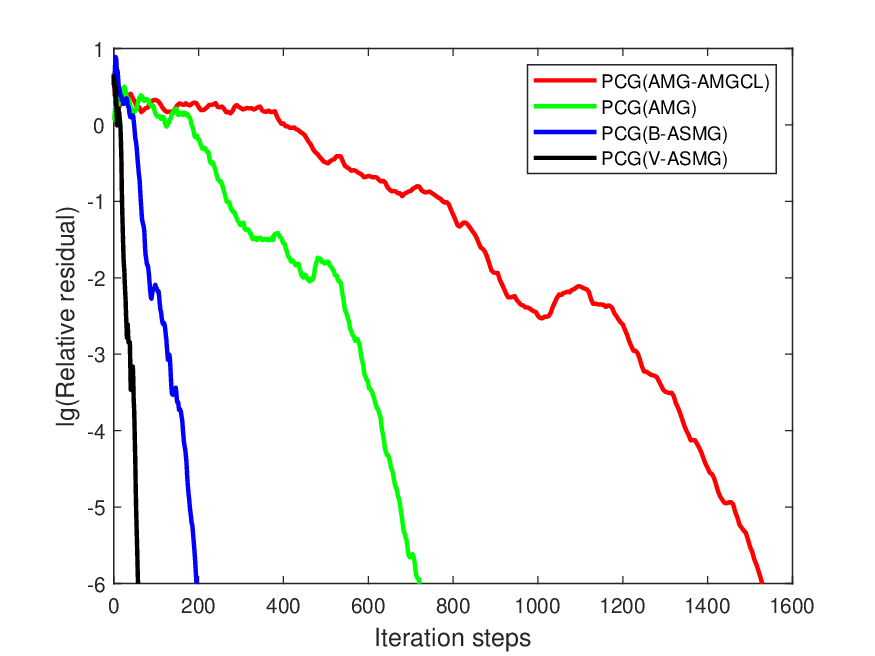}}
	\caption{The convergence history of the relative residual of the commercial vehicle plate spring bracket stress problem.}
	\label{figure5-2}
\end{figure}

\noindent {\bf Example 4.6.} Helicopter hub stress problem with Young's modulus $E=2.1\mathrm{E}+05$ and Poisson's ratio $\nu=0.3$. 
This example is about the structural response of the helicopter hub under flight load, and mainly analyzes the stress distribution and deformation characteristic of the hub under specific constraints and concentrated loads, which provides a basis for structural strength verification and fatigue life evaluation. 
The calculation area contains some key components: the main structure of the hub, the center bore connected to the main shaft, the mounting bores on the lower side, and the weight reduction  bores.
The center bore connected to the main shaft is completely fixed, and the displacements in $x$-, $y$- and $z$-directions are restricted to be $u_1=u_2=u_3=0$. 
The two mounting bores on the lower side are applied a nodal force of 10(i.e. $f_2=10$) along the horizontal direction.
The computational domain is divided based on the tetrahedral elements, and the problem is solved by the Lagrange linear element. 
\figurename \ref{figure6-1}(a) shows its corresponding geometric model and Table \ref{table6-1} shows the obtained equation system information. 
Using the AMG-AMGCL, AMG, B-ASMG and V-ASMG methods as the preconditioners of the PCG method to iteratively solve the equations corresponding to the above problem, the iteration steps, numerical error between the numerical solution and the reference solution under the $L_2$ norm, setup time, application time and total time are shown in Table \ref{table6-2} below.
The convergence history of the relative residual varying with the
iteration steps is shown in \figurename \ref{figure6-2}. 
The displacement nephogram is shown in \figurename \ref{figure6-1}(b).
The above results show that compared with the other three methods, V-ASMG method requires the smallest iteration steps within the shortest time to reach the given tolerance, and obtains the smallest error when the iteration terminates. 
\begin{figure}[!ht]
	\centerline{
\subfigure[Geometric model.] {\includegraphics[width=0.5\textwidth]{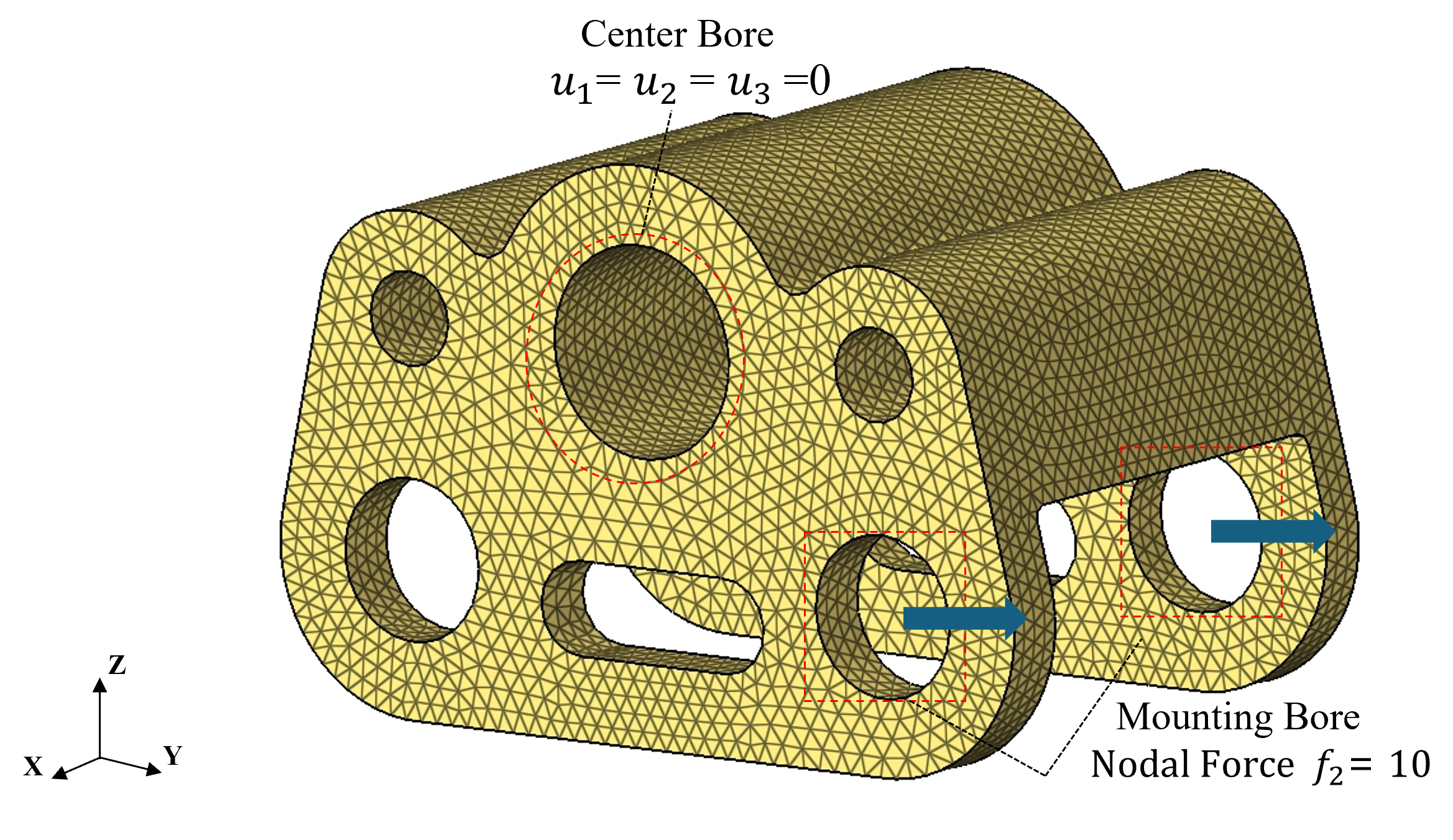}}
\subfigure[Calculation result.] {\includegraphics[width=0.5\textwidth]{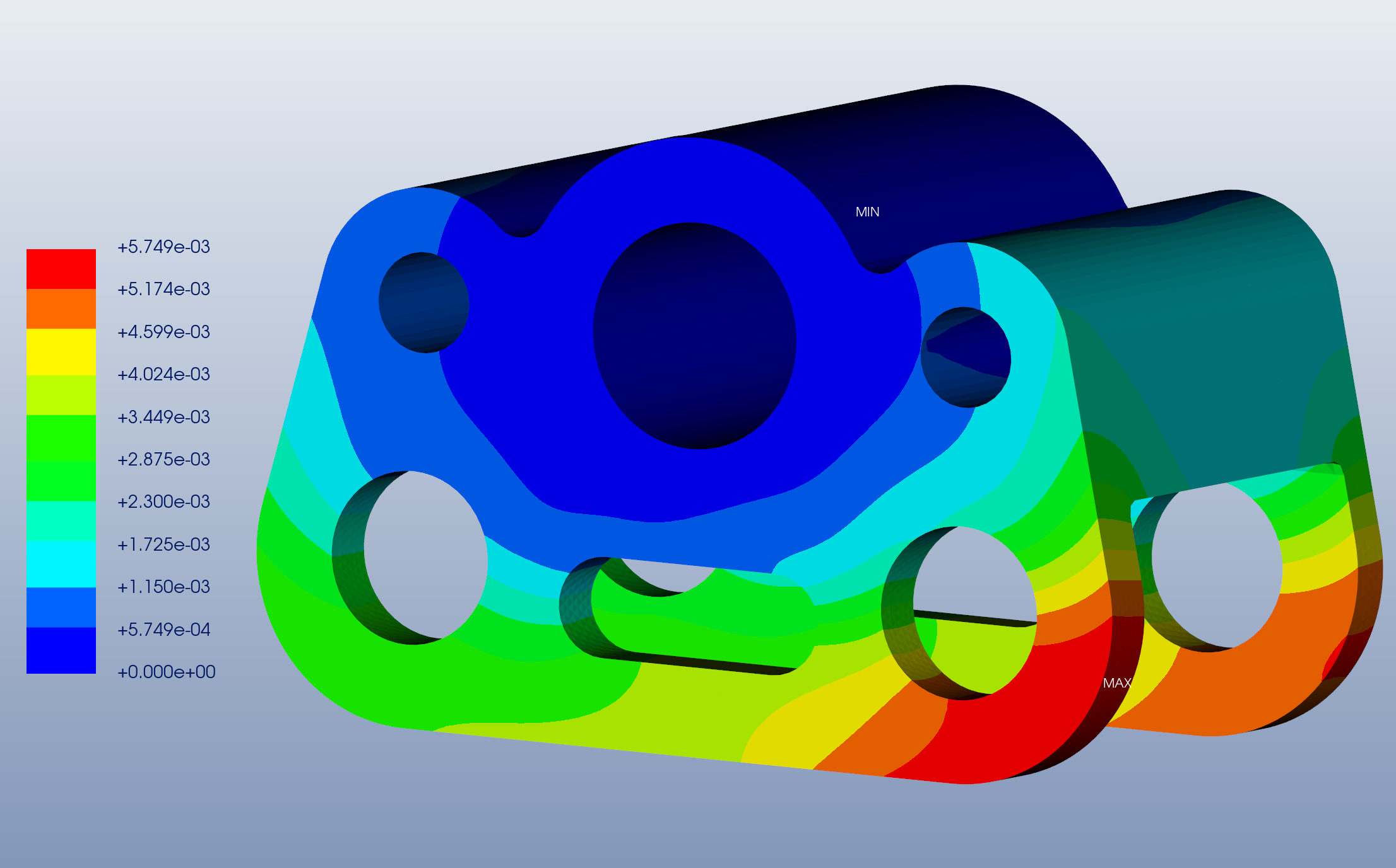}}}
	\caption{Helicopter hub stress problem.}
	\label{figure6-1}
\end{figure}
\begin{table}[!ht]
\begin{center}
{\caption{
The relevant information of the sparse stiffness matrix $A$ corresponding to the helicopter hub stress problem.}
\begin{tabular}{|c|c|}
\hline
Number of vertices           &113954
\\ \hline
Number of unknowns           &341862	
\\ \hline
Number of nonzero elements   &14244584
\\ \hline
Symmetric or not             &Yes
\\ \hline
Positive definite or not     &Yes
\\ \hline
Condition number             &5.19381E+06
\\ \hline
Spectral radius              &9.72260E+06
\\ \hline
\end{tabular}
\label{table6-1}}
\end{center}
\end{table}
\begin{table}[!ht]
\begin{center}
{\caption{
The results of the helicopter hub stress problem using the PCG method with different preconditioners.}
\begin{tabular}{|c|c|c|c|c|}
\hline
\multirow{2}{*}{Method} &AMGCL &\multicolumn{3}{c|}{BEFEM} 
\\ \cline{2-5} 
 &PCG(AMG) &PCG(AMG) &PCG(B-ASMG) &PCG(V-ASMG) 
\\ \hline
Iteration steps  &1160 &597 &179 &26                         
\\ \hline
Numerical error  &1.36560E-07 &3.10987E-07 
                 &8.35128E-07 &1.05761E-07
\\ \hline
Setup time       &  11.523s & 12.243s & 5.833s &46.691s
\\ \hline
Application time &2223.623s &280.855s &58.396s &14.916s
\\ \hline
Total time       &2235.146s &293.098s &64.229s &61.607s
\\ \hline 
\end{tabular}
\label{table6-2}}
\end{center}
\end{table}
\begin{figure}[!ht]
	\centerline{\includegraphics[width=0.5\textwidth]{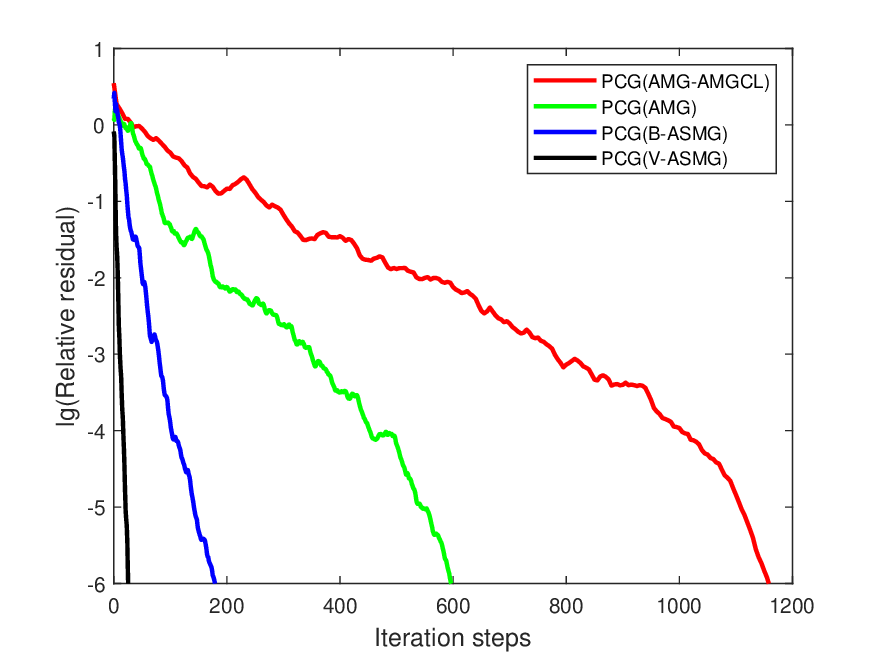}}
	\caption{The convergence history of the relative residual of the helicopter hub stress problem.}
	\label{figure6-2}
\end{figure}

\noindent {\bf Example 4.7.} Aircraft head structural part stress problem with Young's modulus $E=2.1\mathrm{E}+05$ and Poisson's ratio $\nu=0.3$. 
The numerical simulation of mechanical response of typical structural part of aircraft head under uniform load is carried out in this example, which provides a basis for structural optimization. 
Completely fix the right end face, that is, restrict the displacements in $x$-, $y$- and $z$-directions to be $u_1=u_2=u_3=0$.
At the same time, apply a uniform line force of 10000(i.e. $q_3=-10000$) along the negative $z$-axis to the middle circular hole convex mounting surface.
The computational domain is divided based on tetrahedral elements, and Lagrange linear element is used to solve the problem.
\figurename \ref{figure7-1}(a) shows its corresponding geometric model and Table \ref{table7-1} shows the obtained equation system information. 
It can be seen from the information in the last two rows of the Table \ref{table7-1} that the condition number and spectral radius of this example are very large, which is a great challenge for iterative solution. 
The equations corresponding to the above problem are solved by using the AMG-AMGCL, AMG, B-ASMG and V-ASMG methods as the preconditioners of the PCG method, the iteration steps, numerical error between the numerical solution and the reference solution under the $L_2$ norm, setup time, application time and total time are shown in Table \ref{table7-2} below.
The convergence history of the relative residual varying with the iteration steps is shown in \figurename \ref{figure7-2}. 
The displacement nephogram is shown in \figurename \ref{figure7-1}(b).
The above results indicate that compared to the other three methods, V-ASMG method needs the smallest iteration steps within the shortest time to reach the given tolerance, and achieves the smallest error when the iteration stops. 
\begin{figure}[!ht]
	\centerline{
\subfigure[Geometric model.] {\includegraphics[width=0.5\textwidth]{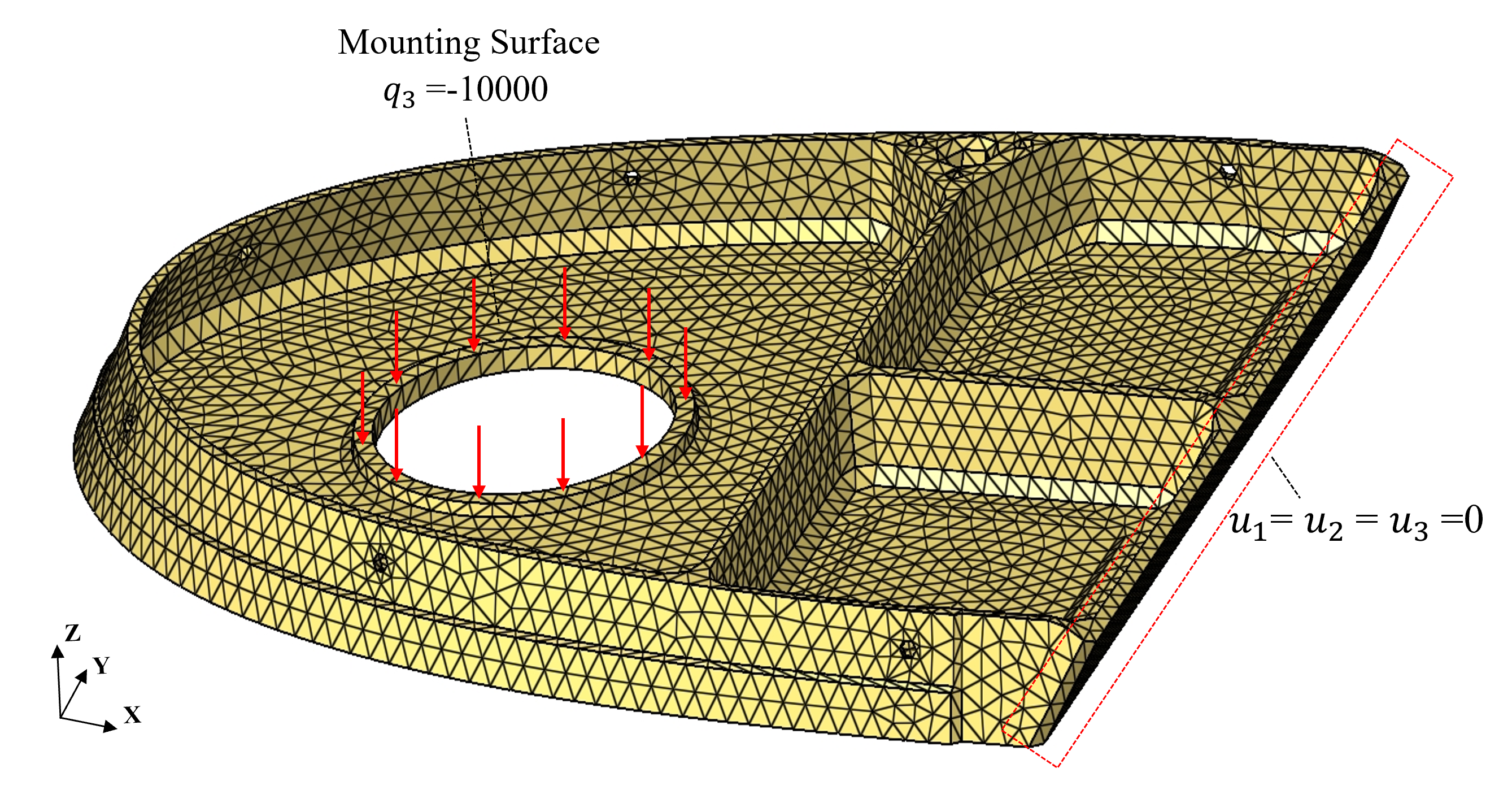}}
\subfigure[Calculation result.] {\includegraphics[width=0.5\textwidth]{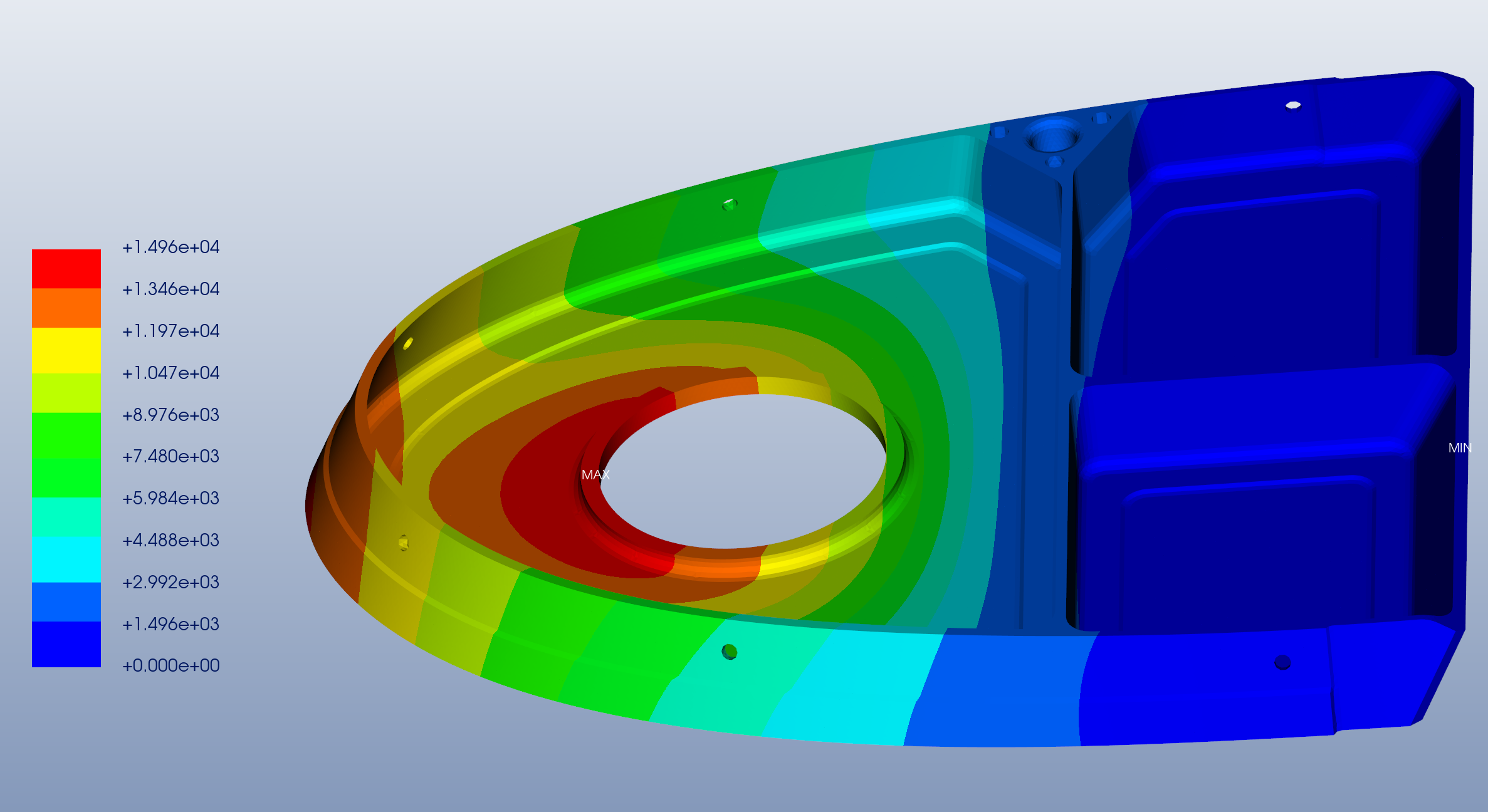}}}
	\caption{Aircraft head structural part stress problem.}
	\label{figure7-1}
\end{figure}
\begin{table}[!ht]
\begin{center}
{\caption{
The relevant information of the sparse stiffness matrix $A$ corresponding to the aircraft head structural part stress problem.}
\begin{tabular}{|c|c|}
\hline
Number of vertices           &139653		
\\ \hline
Number of unknowns           &418959	
\\ \hline
Number of nonzero elements   &14625043
\\ \hline
Symmetric or not             &Yes
\\ \hline
Positive definite or not     &Yes
\\ \hline
Condition number             &2.90196E+10
\\ \hline
Spectral radius              &4.33196E+08
\\ \hline
\end{tabular}
\label{table7-1}}
\end{center}
\end{table}
\begin{table}[!ht]
\begin{center}
{\caption{
The results of the aircraft head structural part stress problem using the PCG method with different preconditioners.}
\begin{tabular}{|c|c|c|c|c|}
\hline
\multirow{2}{*}{Method} &AMGCL &\multicolumn{3}{c|}{BEFEM} 
\\ \cline{2-5} 
 &PCG(AMG) &PCG(AMG) &PCG(B-ASMG) &PCG(V-ASMG) 
\\ \hline
Iteration steps  &8046 &5041 &948 &133
\\ \hline
Numerical error  &2.92160E-05 &3.45245E-05 
                 &3.11411E-05 &1.14639E-05
\\ \hline
Setup time       &    9.043s &  11.188s &  5.832s &37.990s
\\ \hline
Application time &13467.503s &1635.370s &197.140s &40.168s
\\ \hline
Total time       &13476.546s &1646.558s &202.972s &78.158s
\\ \hline 
\end{tabular}
\label{table7-2}}
\end{center}
\end{table}
\begin{figure}[!ht]
	\centerline{\includegraphics[width=0.5\textwidth]{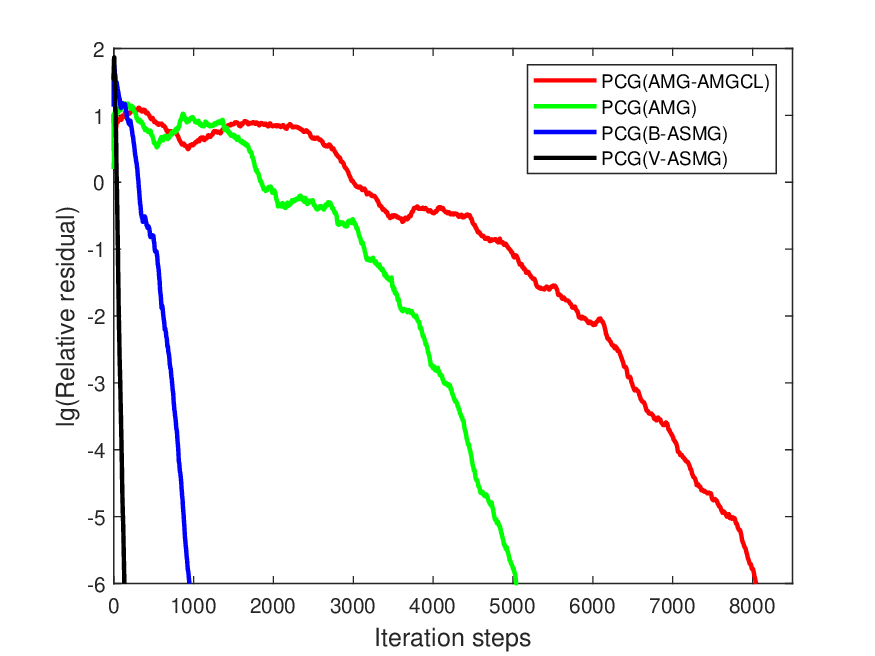}}
	\caption{The convergence history of the relative residual of the aircraft head structural part stress problem.}
	\label{figure7-2}
\end{figure}

\noindent {\bf Example 4.8.} Axle torsion analysis with Young's modulus $E=2.1\mathrm{E}+05$ and Poisson's ratio $\nu=0.3$. 
In this case, the mechanical behavior of the helicopter axle under torque condition is simulated, and the torsional deformation and stress distribution of the axle with keyway structure under tangential force are analyzed, which furnishes a foundation for the reliability evaluation of the transmission system.
The computational area contains some key features: axle, middle keyway, and two end keyways. 
The inner wall of the middle keyway is completely fixed, and the displacements in $x$-, $y$- and $z$-directions are restricted to be $u_1=u_2=u_3=0$. 
On the inner wall of the right end keyway, apply a tangential nodal force of 10(i.e. $f_1=10$).
Divide the computational domain by an unstructured tetrahedral grid, and then use the Lagrange linear element to solve the problem.
The corresponding geometric model of such problem is shown in \figurename \ref{figure8-1}(a) and the information of the obtained equation system is listed in Table \ref{table8-1}.
The four methods are used as the preconditioners of the PCG method to iteratively solve the equations corresponding to the above problem, the iteration steps, numerical error in $L_2$ norm between the numerical solution and the reference solution, setup time, application time and total time are shown in Table \ref{table8-2}. 
The convergence history of the relative residual varying with the iteration steps is shown in \figurename \ref{figure8-2}. 
The displacement nephogram is shown in \figurename \ref{figure8-1}(b).
According to the above results, it is not difficult to find that compared with the other three algorithms, V-ASMG algorithm not only has the fastest relative residual reduction speed, but also has the highest error accuracy.
Moreover, from \figurename \ref{figure8-2}, it can be seen that the relative residual decline curve of the V-ASMG method is very smooth, indicating that it is more stable than the other three methods. 
This advantage of the V-ASMG method is also demonstrated in other examples in this paper, but it is especially obvious in this case.
\begin{figure}[!ht]
	\centerline{
\subfigure[Geometric model.] {\includegraphics[width=0.5\textwidth]{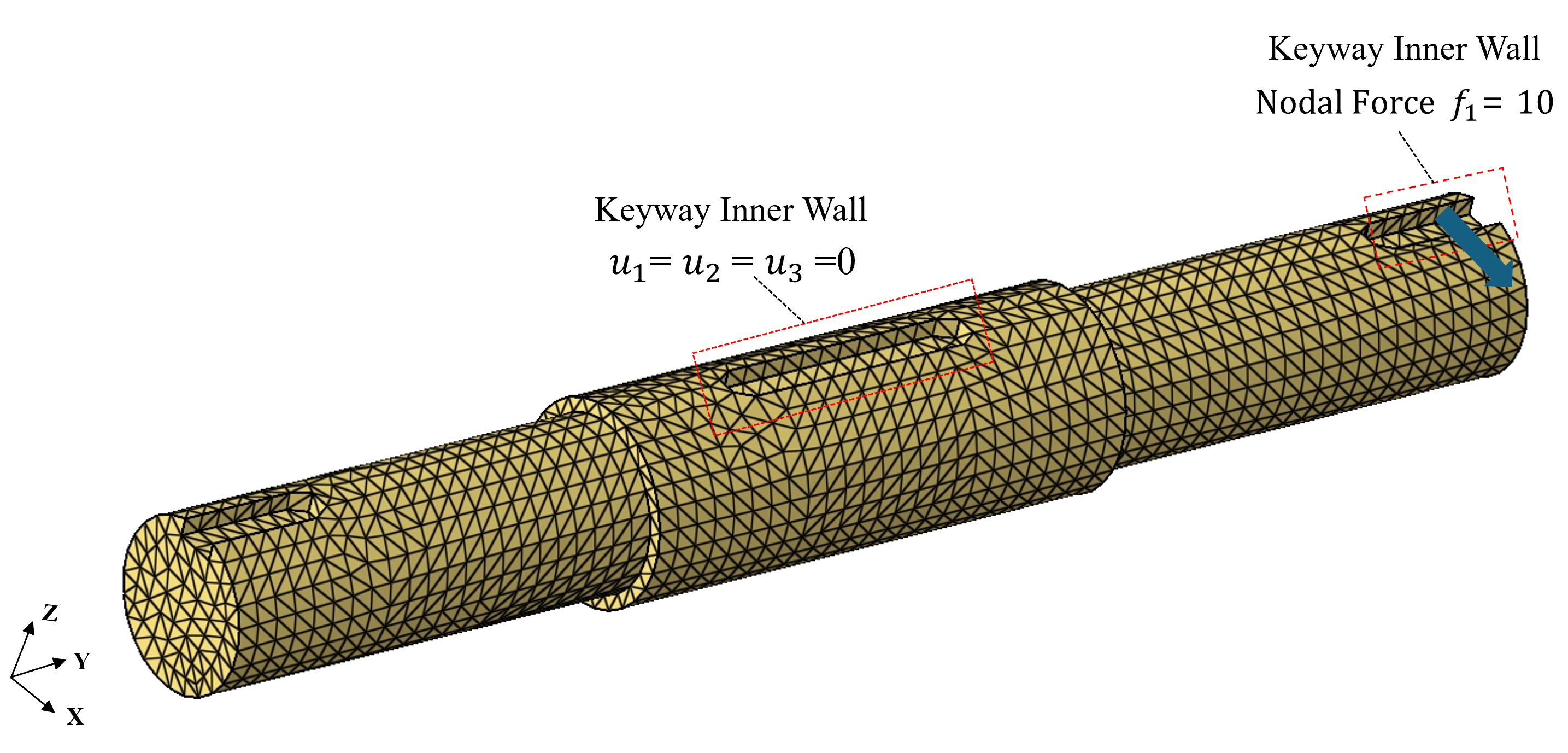}}
\subfigure[Calculation result.] {\includegraphics[width=0.5\textwidth]{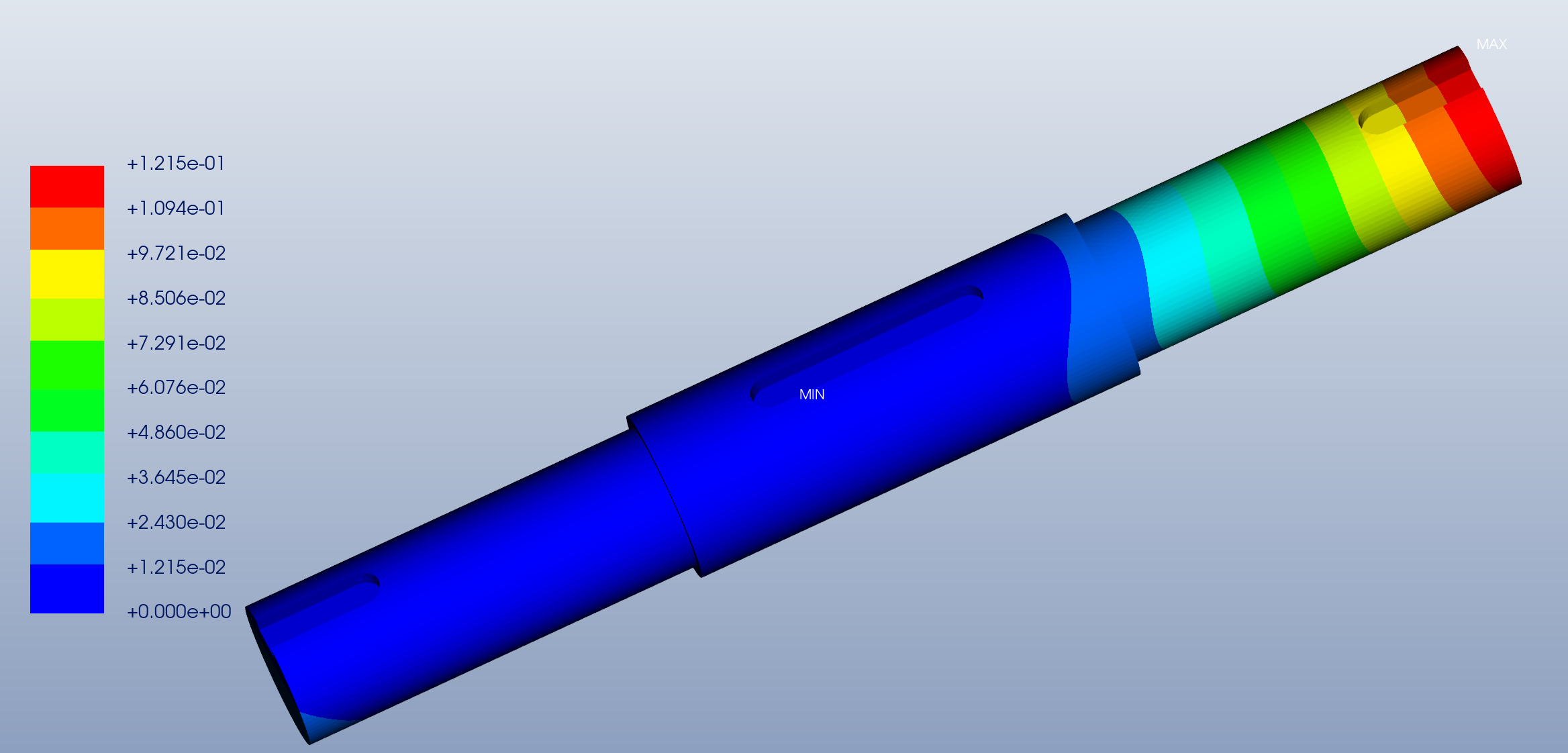}}}
	\caption{Axle torsion analysis.}
	\label{figure8-1}
\end{figure}
\begin{table}[!ht]
\begin{center}
{\caption{
The relevant information of the sparse stiffness matrix $A$ corresponding to the axle torsion analysis.}
\begin{tabular}{|c|c|}
\hline
Number of vertices           &177948		
\\ \hline
Number of unknowns           &533844
\\ \hline
Number of nonzero elements   &23941668
\\ \hline
Symmetric or not             &Yes
\\ \hline
Positive definite or not     &Yes
\\ \hline
Condition number             &2.69937E+07
\\ \hline
Spectral radius              &4.79185E+06
\\ \hline
\end{tabular}
\label{table8-1}}
\end{center}
\end{table}
\begin{table}[!ht]
\begin{center}
{\caption{
The results of the axle torsion analysis using the PCG method with different preconditioners.}
\begin{tabular}{|c|c|c|c|c|}
\hline
\multirow{2}{*}{Method} &AMGCL &\multicolumn{3}{c|}{BEFEM} 
\\ \cline{2-5} 
 &PCG(AMG) &PCG(AMG) &PCG(B-ASMG) &PCG(V-ASMG) 
\\ \hline
Iteration steps  &769 &705 &290 &29                           
\\ \hline
Numerical error  &6.53210E-09 &8.96215E-09 
                 &9.71928E-09 &5.75359E-09
\\ \hline
Setup time       &  17.458s & 22.881s &  9.739s &73.608s
\\ \hline
Application time &2495.338s &398.819s & 98.509s &15.155s
\\ \hline
Total time       &2512.796s &421.700s &108.248s &88.763s
\\ \hline 
\end{tabular}
\label{table8-2}}
\end{center}
\end{table}
\begin{figure}[!ht]
	\centerline{\includegraphics[width=0.5\textwidth]{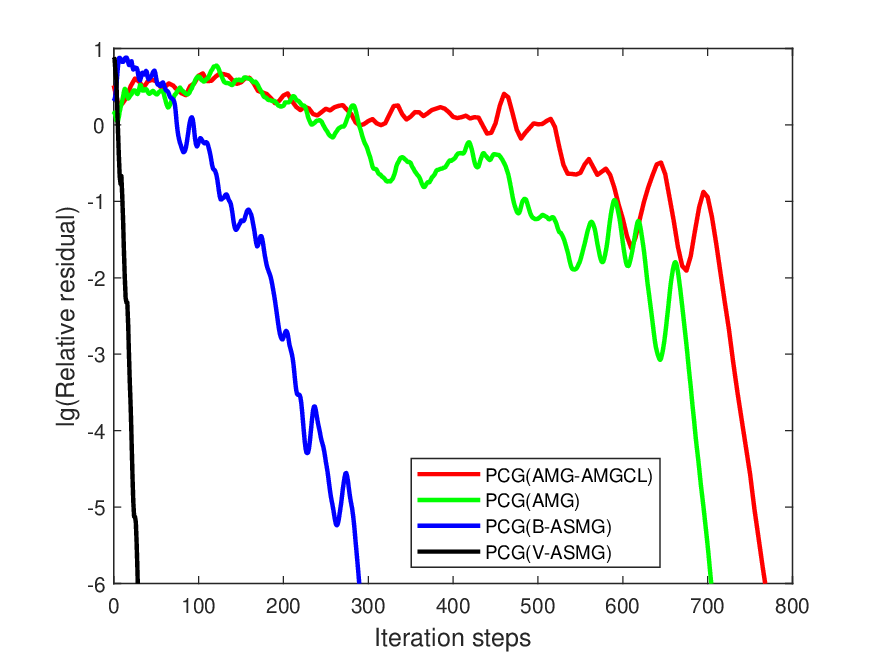}}
	\caption{The convergence history of the relative residual of the axle torsion analysis.}
	\label{figure8-2}
\end{figure}

%%%%%%%%%%%%%%%%%%%%%%%%%%%%%%
\section{Concluding remarks}
%%%%%%%%%%%%%%%%%%%%%%%%%%%%%%
In this paper, we propose a V-ASMG method as a preconditioner of the PCG method for solving the large sparse linear equations derived from the linear elasticity equations. 
Of course, this method can be naturally generalized to other equations, because all we need is the information of the grid vertices.
Such method works in the similar way as that in \cite{gwx2016} for constructing the auxiliary region-tree, which is just called a box-tree there.
The biggest difference between them is that the grid vertices instead of the barycenters are chosen as the representative points.
This makes it unnecessary to deal with the hanging points after building the auxiliary region-tree, and it is possible to construct the prolongation/restriction operator directly by using the bilinear interpolation function.
Then, the scale of the stiffness matrix corresponding to each grid layer drops more slowly, so that the error components of different frequencies can be eliminated better.
The numerical results in the previous section further show that the iteration steps of the PCG method with such V-ASMG method as a preconditioner are far less than that of the classical AMG method and B-ASMG method.

%%%%%%%%%%%%%%%%%%%%%%%%
\section*{Acknowledgement}
%%%%%%%%%%%%%%%%%%%%%%%%

The authors would like to thank the editor and anonymous referees for their valuable comments and suggestions, which have helped us improve this paper a lot.

%%%%%%%%%%%%%%%%%%%%%%%%%%%%%%%%%%%%%%%%%%%%%%%%


\begin{thebibliography}{10}
%%%%%%%%%%%%%%%%%%%%%%%%%%%%%%%%%%%%%%%%%%%%%%%%
\bibitem{b1977}
A. Brandt. 
Multi-Level Adaptive Solutions to Boundary-Value Problems.  
{\em Mathematics of Computation}, 
1977, 31(138): 333-390.

\bibitem{b1982}
A. Brandt. 
Guide to Multigrid Development. 
In W. Hackbusch, U. Trottenberg(Eds.), 
{\em Multigrid Methods}(pp. 220-312).
Springer-Verlag, 
1982.

\bibitem{b1984}
A. Brandt. 
{\em Multigrid Techniques: 1984 Guide with Applications to Fluid Dynamics}. 
GMD Studien, 
1984.

\bibitem{b1993} J. Bramble. 
Multigrid Methods, in: {\em Pitman Research Notes in Mathematics Series}. 
Pitman, London, 1993.

\bibitem{bbcdddeprv1994}
R. Barrett, M. Berry, T.F. Chan, J. Demmel, J. Donato, J. Dongarra, V. Eijkhout, R. Pozo, C. Romine, H. van der Vorst. 
Templates for the Solution of Linear Systems: Building Blocks for Iterative Methods. 
{\em SIAM}, 
1994.

\bibitem{bmr1983}
A. Brandt, S. McCormick, J. Ruge.  
Algebraic Multigrid (AMG) for Sparse Matrix Equations. 
In {\em Proceedings of the Copper Mountain Conference on Multigrid Methods}(pp. 1-24). 
Copper Mountain, CO, USA, % Meeting place
1983.

\bibitem{boss2023}
N. Bell1, L.N. Olson, J. Schroder, B. Southworth. 
PyAMG: Algebraic Multigrid Solvers in Python.
{\em Journal of Open Source Software},
2023, 8(87): 5495.

\bibitem{bp1993} 
J.H. Bramble, J.E. Pasciak. 
New estimates for multilevel algorithms including the V-cycle. 
{\em Mathematics of Computation},
1993, 60(202): 447-71.

\bibitem{bpx1991}
J.H. Bramble, J.E. Pasciak, J.C. Xu. 
The analysis of multigrid algorithms with nonnested spaces or
noninherited quadratic forms. 
{\em Mathematics of Computation}, 
1991, 56(193): 1-34.

\bibitem{bs1992} S.C. Brenner, L.Y. Sung. 
Linear finite element methods for planar linear elasticity. 
{\em Mathematics of Computation}, 
1992, 59(200): 321-38.

\bibitem{csz1996} 
T.F. Chan, B.F. Smith, J. Zou. 
Overlapping Schwarz methods on unstructured meshes using non-matching coarse grids. 
{\em Numerische Mathematik}, 
1996, 73: 149-67.

\bibitem{f1964}
R.P. Fedorenko. 
The Speed of Convergence of One Iterative Process. 
{\em USSR Computational Mathematics and Mathematical Physics}, 
1964, 4(3): 227-35.

\bibitem{fy2002}
R.D. Falgout, U.M. Yang. 
Hypre: A Library of High Performance Preconditioners. 
In {\em International Conference on Computational Science}(pp. 632-641). 
Springer, 
2002.

\bibitem{ghjrw2016}
B. Gmeiner, M. Huber, L. John, U. R$\mathrm{\ddot{u}}$de, B. Wohlmuth. 
A quantitative performance study for Stokes solvers at the extreme scale. 
{\em SIAM Journal on Scientific Computing}, 
2016, 17(3): 509-521.

\bibitem{grsww2015}
B. Gmeiner, U. R$\mathrm{\ddot{u}}$de, H. Stengel, C. Waluga, B. Wohlmuth.
Towards textbook efficiency for parallel multigrid.
{\em Numerical Mathematics: Theory, Methods and Applications},
2015, 8(1): 22-46.

\bibitem{gv2013}
G.H. Golub, C.F. Van Loan. 
Matrix Computations(4th ed.).  
{\em Johns Hopkins University Press}, 
2013.

\bibitem{gwx2016}
L. Grasedyck, L. Wang, J.C. Xu. 
A nearly optimal multigrid method for general unstructured
grids. 
{\em Numerische Mathematik}, 
2016, 134: 637-666.

\bibitem{h1980}
W. Hackbusch. 
Convergence of Multi-Grid Iterations Applied to Difference Equations.
{\em Mathematics of Computation}, 
1980, 35(152): 1083-1098.

\bibitem{h1985}
W. Hackbusch. 
{\em Multi-Grid Methods and Applications}. 
Springer-Verlag, 
1985. 

\bibitem{hs1952}
M.R. Hestenes, E. Stiefel. 
Methods of Conjugate Gradients for Solving Linear Systems. 
{\em Journal of Research of the National Bureau of Standards}, 
1952, 49(6): 409-436.

\bibitem{hy1981}
L.A. Hageman, D.M. Young. 
Applied Iterative Methods.  
{\em Academic Press}, 
1981.

\bibitem{m1987}
S.F. McCormick. 
{\em Multigrid Methods}. 
Society for Industrial and Applied Mathematics, 
1987.

\bibitem{rs1987}
J.W. Ruge, K. St$\mathrm{\ddot{u}}$ben.
Algebraic Multigrid(AMG). 
In S.F. McCormick(Ed.), 
{\em Multigrid Methods}(pp. 73-130). 
Society for Industrial and Applied Mathematics, 
1987.

\bibitem{rsc2014}
C. Richter, S. Sch$\mathrm{\ddot{o}}$ps, M. Clemens. 
GPU Acceleration of Algebraic Multigrid Preconditioners for Discrete Elliptic Field Problems. 
{\em IEEE Transactions on Magnetics}, 
2014, 55(2): 7011304.

\bibitem{s1999}
K. St$\mathrm{\ddot{u}}$ben. 
A Review of Algebraic Multigrid. 
{\em Journal of Computational and Applied Mathematics}, 
1999, 128(1-2): 281-309.

\bibitem{s2003}
Y. Saad.
Iterative Methods for Sparse Linear Systems(2nd ed.). 
{\em SIAM},
2003.

\bibitem{ss1986}
Y. Saad, M.H. Schultz. 
GMRES: A Generalized Minimal Residual Algorithm for Solving Nonsymmetric Linear Systems. 
{\em SIAM Journal on Scientific and Statistical Computing},  
1986, 7(3): 856-869.

\bibitem{tb1997}
L.N. Trefethen, D. Bau.
Numerical Linear Algebra. 
{\em SIAM},
1997.

\bibitem{v2000}
R.S. Varga. 
Matrix Iterative Analysis(2nd ed.). 
{\em Springer}, 
2000.

\bibitem{vmb1996}
P. Van$\mathrm{\check{e}}$k, J. Mandel, M. Brezina. 
Algebraic Multigrid by Smoothed Aggregation. 
{\em SIAM Journal on Scientific Computing}, 
1996, 17(2): 317-340.

\bibitem{whcx2013}
L. Wang, X.Z. Hu, J. Cohen, J.C. Xu. 
A parallel auxiliary grid algebraic multigrid method for graphic
processing units. 
{\em SIAM Journal on Scientific Computing}, 
2013, 35(3): C263-C283.

\bibitem{x1996}
J.C. Xu. 
The auxiliary space method and optimal multigrid preconditioning techniques for unstructured grids. 
{\em Computing}, 
1996, 56: 215-235.

\bibitem{xz2017}
J. Xu, L. Zikatanov.
Algebraic Multigrid Methods.
{\em Acta Numerica}, 
2017, 26: 591-721.

\bibitem{y1971}
D.M. Young. 
Iterative Solution of Large Linear Systems.
{\em Academic Press}, 
1971.
\end{thebibliography}
\end{document}